\documentclass{article}
\usepackage{arxiv}
\usepackage[utf8]{inputenc} 
\usepackage[T1]{fontenc}    
\usepackage{hyperref}       
\usepackage{amsfonts}       
\usepackage{amsthm}
\usepackage{amsmath}
\usepackage{amssymb}
\usepackage[utf8]{inputenc} 
\usepackage[T1]{fontenc}    
\usepackage{hyperref}       
\usepackage{url}            
\usepackage{booktabs}       
\usepackage{amsfonts}       
\usepackage{nicefrac}       
\usepackage{microtype}      
\usepackage{lipsum}     
\usepackage{graphicx}
\usepackage{doi}
\usepackage{float}

\theoremstyle{definition}

\newtheorem{remark}{Remark}
\newtheorem{definition}{Definition}
\newtheorem*{attention*}{ATTENTION}
\newtheorem{example}{Example}
\newtheorem*{example*}{Example}
\theoremstyle{theorem}
\newtheorem{statement}{Statement}
\newtheorem{theorem}{Theorem}
\newtheorem{lemma}{Lemma}
\newtheorem{corollary}{Corollary}
\usepackage{xcolor}

\usepackage[title]{appendix}

\title{PINNs error estimates for nonlinear equations in $\mathbb{R}$-smooth Banach spaces}

\author{ \href{https://orcid.org/0000-0001-6514-1208}{\includegraphics[scale=0.06]{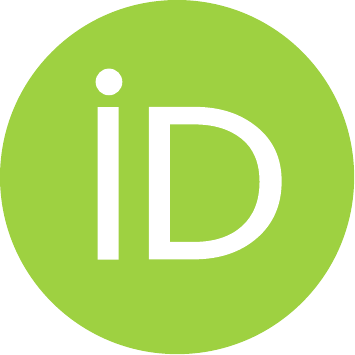}\hspace{1mm}Jiexing Gao} \\
    Moscow Research Center, 2012 Labs\\
    Huawei Technologies Co., Ltd\\
    \texttt{gaojiexing@huawei.com} \\
    \And
    \href{https://orcid.org/0000-0003-0042-3372}{\includegraphics[scale=0.06]{orcid.pdf}\hspace{1mm}Yurii Zakharian} \\
    Moscow Research Center, 2012 Labs\\
    Huawei Technologies Co., Ltd\\
    \texttt{zakharyan.yuriy1@huawei-partners.com} \\
}

\date{}

\renewcommand{\headeright}{}
\renewcommand{\undertitle}{}
\renewcommand{\shorttitle}{PINNs error estimates in $\mathbb{R}$-smooth Banach spaces}

\hypersetup{
pdftitle={PINNs error estimates in R-smooth Banach spaces},
pdfsubject={math},
pdfauthor={Jiexing Gao, Yurii Zakharian}
}

\begin{document}
\maketitle

\begin{abstract}
    In the paper, we describe in operator form classes of PDEs that admit PINN's error estimation. Also, for $L^p$ spaces, we obtain a Bramble-Hilbert type lemma that is a tool for PINN's residuals bounding.
\end{abstract}


\section{Introduction}
    In 2017, M. Raissi et al. introduced the Physics-informed neural network (PINN) approximating solutions to partial differential equations (PDEs) \cite{raissiperdikariskarniadakis}. It reduces losses related to PDE and boundary/initial conditions. In recent years, the number of papers dedicated to deep learning methods for solving PDEs, including PINNs, is constantly increasing (see, for instance, \cite{ehanjentzen, kutyniokpetersenraslanschneider, lyemishraray, lujinkarniadakis} for deep learning methods and \cite{hujagtapkarniadakiskawaguchi, jagtapkarniadakis, jagtapkharazmikarniadakis, jagtapmaoadamskarniadakis, shuklajagtapblackshiresparkmankarniadakis, shuklajagtapkarniadakis} for PINN). Consequently, a thorough exploration of the theoretical aspects associated with PINNs is of great significance.

    For instance, the question arises as to why PINN's training algorithm leads us to an accurate approximation. In other words, is it possible to control total error for sufficiently small residuals/training errors? In \cite{mishramolinaro}, S. Mishra and R. Molinaro presented an error estimation answering this question and offered an operator description of the sufficient conditions for applying such a method. However, while these conditions were initially outlined rather generally, in practice, verifying them requires obtaining the estimate itself. Also, although these conditions were outlined for $L^p$ spaces, examples in \cite{mishramolinaro} were all related to the problems in $L^2$. 

    In \cite{hillebrechtunger1}, the authors studied a similar issue and presented PINN's error estimation for a homogeneous initial value problem with a linear generator of a strongly continuous semigroup as an operator. Furthermore, they applied this method to several nonlinear equations. Yet, they used another approach for a narrower class of PDEs. 

    Another related question is whether a neural network can have sufficiently small residuals. In \cite{deryckjagtapmishra}, the authors answered this question and derived an $L^2$-bound on residuals using the Bramble-Hilbert lemma. For this reason, they used the Bramble-Hilbert lemma \cite{verfurth}. The challenging moment for our research is to obtain the Bramble-Hilbert lemma for $L^p$ to extend results in \cite{deryckjagtapmishra} to such spaces. Also, refer to \cite{baikoleymishramolinaro, derycklanthalermishra, deryckmishramolinaro}.

    Therefore, the paper aims to resolve the following issues:
    \begin{enumerate}
        \item Specify PDE classes suitable for PINN's error estimation due to S. Mishra et al.

        \item At the same time, to consider a general framework for a broad class of non-Hilbert spaces (e.g., $L^p$). 

        \item Extend the approach in \cite{hillebrechtunger1} to a broader PDEs class. Compare the respective estimate to the one obtained by the estimation for the PDE classes of question 2.

        \item Extend results in \cite{deryckjagtapmishra} on PINN's residuals bounding to the case of $L^p$ spaces. 
    \end{enumerate}

    Answering the first question, we consider the following three natural equation types and define used operators.
    \begin{itemize}
        \item(Parabolic-type equation)
        \begin{equation}
            \label{partype}
            \frac{dw}{dt}=A(t)(w)
        \end{equation}

        \item(Generalized parabolic-type equation)
        \begin{equation}
            \label{genpartype}
            \frac{d}{dt}\sum_{k=1}^{K}U_kw=A(t)(w)
        \end{equation}

        \item(Hyperbolic-type equation)
        \begin{equation}
            \label{hyptype}
            \frac{d^2w}{dt^2}=-Uw+F(t)(w)+A(t)\left(\frac{dw}{dt}\right)
        \end{equation}

        \item(Elliptic-type equation)
        \begin{equation}
            \label{elltype}
            A(y)=0
        \end{equation}
    \end{itemize}
    
    Regarding the second question, in \cite{baikoleymishramolinaro, deryckjagtapmishra, mishramolinaro}, the authors repeated one procedure: they estimated the time derivative of the total error squared $L^2$-norm using the residuals. After that, they applied the Gr\"onwall-Bellman lemma to obtain the desired estimate. To extend this idea to a Banach space $X$, we need an existence of the time derivative of the total error $p$-powered $X$-norm for some $p>1$. For this reason, in subsection \ref{rsmooth}, we consider so-called \textit{$\mathbb{R}$-smooth} Banach spaces and introduce a \textit{real $p$-form}. In subsections \ref{submss}--\ref{coerss}, we use this form to define operators of (\ref{partype})--(\ref{elltype}) in such spaces. As a result, in subsections \ref{parss} and \ref{gparss}--\ref{ellss}, we obtain a total error estimate in general (operator) form, answering the second question.

    To answer the third question in subsection \ref{parnss}, we expand the approach in \cite{hillebrechtunger1} to semi-linear equations and compare the respective estimate with the estimate derived from differentiating the $p$-powered norm. 

    In section \ref{residss}, since the Bramble-Hilbert lemma was a core tool for PINN's residuals bounding in \cite{deryckjagtapmishra}, we derive a Bramble-Hilbert type lemma for $L^p$ and outline obtaining $L^p$ bounds on PINN's residuals. 

    Finally, section \ref{numerical} is devoted to experimental results, which demonstrate how the core results of the paper work.   

\section{Preliminary}
    In this section, we first define admissible spaces to describe PDE classes (\ref{partype})--(\ref{elltype}). Then, in such spaces, we introduce a \textit{real $p$-form} that is a core tool in the paper and derive its properties. Using this form, we define three types of operators used in (\ref{partype})--(\ref{elltype}). Moreover, in these definitions, we use function $\psi$, which allows us to factor in additional conditions (boundary conditions in most cases). Namely, we introduce \textit{$(p,\psi)$-submonotone}, \textit{$(p,\epsilon)$-powered}, \textit{$(p,\psi)$-submonotone w.r.t. norm} operarors and consider several examples. Also, we define \textit{$(p,\kappa)$-dissipative} operators as one essential subtype. 

    Let $\mathbb{K}$ be either $\mathbb{C}$ or $\mathbb{R}$. Let $\mathbb{R}^{+}$ stand for $[0;+\infty)$. Let $X$ be a Banach space over $\mathbb{K}$ with norm $\|\cdot\|$, and $S=\{y\in X\mid \|y\|=1\}$ be a unit sphere. Also, let $I=(0;T)$ and $\overline{I}=[0;T]$ be time intervals for some $T>0$. 

    \subsection{$\mathbb{R}$-smooth Banach space and real $p$-form}\label{rsmooth}
        An estimation in \cite{baikoleymishramolinaro, deryckjagtapmishra, mishramolinaro} relied on the Gr\"onwall-Bellman lemma. To apply the lemma, the authors estimated the time derivative of squared norm $\frac{d}{dt}\|w-\tilde{w}\|_{L^2}^2$, where $w$ is the solution of a PDE, and $\tilde{w}$ is its approximation. To extend this idea to a Banach space $X$, we need an existence of the derivative $\frac{d}{dt}\|w-\tilde{w}\|^p$ for some $p>1$. For this reason, we consider $\mathbb{R}$-smooth Banach spaces as admissible spaces. This existence is necessary only for (\ref{partype})--(\ref{hyptype}), yet $\mathbb{R}$-smooth Banach spaces appear to be useful also for (\ref{elltype}).
        \begin{definition}(\cite{giles})
            A real-valued function $\varphi:X\to\mathbb{R}$ is said to be \textit{Gateaux $\mathbb{R}$-differentiable at $y\in X$} if
            \begin{equation*}
                \forall\chi\in X:~\exists\lim_{\mathbb{R}\ni s\to 0}\frac{\varphi(y+s\chi)-\varphi(y)}{s}=:D(\varphi)(y;\chi)
            \end{equation*} 
            called \textit{Gateaux $\mathbb{R}$-derivative at $y$ w.r.t. direction $\chi$}.
        \end{definition}
        \begin{remark}
            We use the Gateaux derivative as in the paper \cite{giles} since we don't need $s$ to be complex. Let us note that the Gateaux $\mathbb{R}$-derivative still satisfies the chain rule: if $w:\mathbb{R}^{+}\to X$ is differentiable, $D(\varphi)(y,\cdot)$ exists and continuous, for $\varphi:X\to\mathbb{R}$ and every $y\in X$, then
            \begin{equation}
                \label{chainrule}
                \frac{d\varphi(w)}{dt}=D(\varphi)\left(w;\frac{dw}{dt}\right)
            \end{equation}
        \end{remark}
        \begin{definition}(\cite{giles})
            A Banach space $X$ is said to be \textit{$\mathbb{R}$-smooth} if its norm is Gateaux $\mathbb{R}$-differentiable at every $y\in S$. 
        \end{definition}
        \begin{remark}
            For the properties of $D(\|\cdot\|)$, see \cite{ishihara}. In \cite{lumer}, G. Lumer introduced the notion of semi-inner-product. Later, J. R. Giles showed the relation between any homogeneous and continuous semi-inner product and the Gateaux $\mathbb{R}$-derivative of the respective norm \cite{giles}. Thus, any $\mathbb{R}$-smooth Banach space can be endowed with the following semi-inner product
            \begin{equation*}
                [\chi,y]:=\|y\|\left(D(\|\cdot\|)(y;\chi)-iD(\|\cdot\|)(y;i\chi)\right).
            \end{equation*}
            In the paper, we don't use explicitly the Gateaux $\mathbb{R}$-derivative of the norm or even the semi-inner product. Instead, we deal with $\|y_2\|^{p-2}{\rm Re}[y_1,y_2]$, for some $p>1$, relying on the semi-inner product and the Gateaux $\mathbb{R}$-derivative properties. 
        \end{remark}
        \begin{definition}
            Let $p>1$, $X$ be a $\mathbb{R}$-smooth Banach space. We will call a form $\langle\cdot,\cdot\rangle_{p}:X\times X\to\mathbb{R}$, where $\langle y_1,y_2\rangle_p:=\|y_2\|^{p-2}{\rm Re}[y_1,y_2]$, a \textit{real $p$-form}.
        \end{definition}
        \begin{remark}
            The real $p$-form is correctly defined. Really, even though at $y_2=0$, the Gateaux $\mathbb{R}$-derivative $D(\|\cdot\|)(y_2;y_1)$ does not exists, for $p>1$, we have 
            \begin{equation*}
                \langle y_1,y_2\rangle_p=\|y_2\|^{p-1}D(\|\cdot\|)(y_2;y_1)\to 0, y_2\to0.
            \end{equation*}
            Let us note that $\|y_2\|^{p-2}[y_1,y_2]$ (without real part) appears in literature as a \textit{semi-inner product of type $p$} \cite{pappavlovic}.
        \end{remark}
        Now, we formulate essential properties of real $p$-form. 
        \begin{lemma}   
            \label{props}
            A real $p$-form on a $\mathbb{R}$-smooth Banach space satisfies the following properties.
            \begin{enumerate}
                \item $\langle y_1+y_2,y_3\rangle_{p}=\langle y_1,y_3\rangle_{p}+\langle y_2,y_3\rangle_{p}$, and $\langle\lambda y_1,y_2\rangle_{p}=\left({\rm Re}\lambda\right)\langle y_1,y_2\rangle_{p}$, for $\lambda\in\mathbb{K}$.

                \item $\langle y_1,\lambda y_2\rangle_{p}=\left({\rm Re}\lambda\right)|\lambda|^{p-2}\langle y_1,y_2\rangle_{p}$, for any $\lambda\in\mathbb{K}$.

                \item $\langle y,y\rangle_{p}=\|y\|^p$.

                \item $|\langle y_1,y_2\rangle_{p}|\leq \|y_1\|\|y_2\|^{p-1}$

                \item $\|\cdot\|^p$ is Gateaux $\mathbb{R}$-differentiable at any point and 
                \begin{equation*}
                    D(\|\cdot\|^p)(y_2;y_1)=p\langle y_1,y_2\rangle_p,~\forall y_1,y_2\in X.
                \end{equation*}

                \item $\langle\chi_n,y\rangle_p\to\langle\chi,y\rangle_p$, if $\chi_n\to\chi$.

                \item  If $w:\overline{I}\to X$ is differentiable, then
                \begin{equation}
                    \label{dnormpowp}
                    \frac{d}{dt}\|w\|^p=p\left\langle \frac{dw}{dt},w\right\rangle_p
                \end{equation} 

                \item Let $w:\overline{I}\to X$ be differentiable, $\mathcal{D}(U)\subset X$ be a subspace, $U:\mathcal{D}(U)\to X$ be a linear operator such that $Uw$ is differentiable and $\frac{d}{dt}Uw=U\frac{dw}{dt}$. Then 
                \begin{equation}
                    \label{dnormoperpowp}
                    \frac{d}{dt}\|Uw\|^p=p\left\langle U\frac{dw}{dt},Uw\right\rangle_p
                \end{equation} 

                \item If $D(\|\cdot\|)(\cdot,\chi)$ is continuous on the ring $\{y\in X\mid r_0<\|y\|<r_1\}$, for every $\chi\in X$ and $0<r_0<r_1$, then the real $p$-form is continuous w.r.t. the second variable, that is, 
                \begin{equation*}
                    \langle\chi,y_n\rangle_p\to\langle\chi,y\rangle_p,~ \forall y\in X,\forall y_n\to y.
                \end{equation*}
            \end{enumerate}
        \end{lemma} 
        \begin{proof}
            Properties 1--6 follow from the definition and semi-inner product properties. Property 7 is a consequence of 5, 6, and (\ref{chainrule}). 

            In property 8, we cannot apply the chain rule since $D(\|U\cdot\|^p)(y;\chi)=p\langle U\chi, Uy\rangle_p$ is not continuous w.r.t. $\chi$ in general. However,
            \begin{equation*}
                \begin{aligned}
                    &\frac{\|Uw(t+\Delta t)\|^p-\|Uw(t)\|^p}{\Delta t}=D(\|U\cdot\|^{p})\left(w(t);U\frac{w(t+\Delta t)-w(t)}{\Delta t}\right)+\frac{o(\Delta t)}{\Delta t}=\\
                    &=p\left\langle U\frac{w(t+\Delta t)-w(t)}{\Delta t},w(t)\right\rangle_p+\frac{o(\Delta t)}{\Delta t}.
                \end{aligned}
            \end{equation*}
            With $U\frac{w(t+\Delta t)-w(t)}{\Delta t}\to \frac{d}{dt}Uw=U\frac{dw}{dt}$ and property 6, we obtain (\ref{dnormoperpowp}).

            To prove property 9, we need to consider two cases. If $y\neq 0$, then starting from some $n$, $y_n$ all together with $y$ contain in a ring, where the real $p$-form is continuous. If $y_n\to 0$, then $\langle\chi,y_n\rangle_p\to 0 = \langle\chi, 0\rangle_p$ by property 4.
        \end{proof}
        Let us describe two examples. For other examples, see, for instance, \cite{johnsgibson}.
        \begin{example}
            Every Hilbert space is a $\mathbb{R}$-smooth space. The inner product coincides with the semi-inner product, and the real $2$-form is
            \begin{equation*}             
                \langle y_1,y_2\rangle_{2}={\rm Re}\langle y_1,y_2\rangle
            \end{equation*}
        \end{example}
        \begin{example}
            If $X=L^p(\Omega;\mathbb{C}^{n})$ is a space of vector-valued functions, in a case of $\sigma$-finite measure, $p>1$, then the real $p$-form is
            \begin{equation*}
                \langle\mathbf{y}_1,\mathbf{y}_2\rangle_{p}={\rm Re}\sum_{i=1}^{n}\int_{\Omega}|y_{2,i}|^{p-2}\overline{y_{2,i}}y_{1,i}dx=\sum_{i=1}^{n}\langle y_{1,i},y_{2,i}\rangle_{p,L^p(\Omega;\mathbb{C})}
            \end{equation*}
            Let us note that it is continuous w.r.t. the second variable. For real spaces, see \cite{diestel, sundaresanswaminathan}.
        \end{example}
    \subsection{Submonotone operators}\label{submss}
        First, we define operators $A(t)$ for (\ref{partype}) and (\ref{hyptype}).
        \begin{definition}
            \label{submonotone}
                Let $X$ be a $\mathbb{R}$-smooth Banach space, $\mathcal{D}(A)\subset X$ be a subspace. Let $A:\mathcal{D}(A)\to X$ be an operator on $X$ and $\mathcal{M}\subset\mathcal{D}(A)$ be a subspace. We will say that $-A$ is \textit{$(p,\psi)$-submonotone on $\mathcal{M}$} for some $\psi:\mathcal{D}(A)\times \mathcal{M}\to\mathbb{R}$ and $p>1$ if it is true that  
                \begin{equation*}
                    \langle A(\chi)-A(y),\chi-y\rangle_{p}\leq \psi(\chi,y)+\Lambda(\chi,y)\|\chi-y\|^p,
                \end{equation*}
                for every $\chi\in \mathcal{D}(A)$, $y\in\mathcal{M}$, and some $\Lambda(\chi,y)\geq0$. 
        \end{definition}
        \begin{remark}
            We use the fact that not $A$ but $-A$ is ''submonotone'' due to a similar term in the paper \cite{spingarn}.
        \end{remark}
        Here, $\psi$ and $\mathcal{M}$ are related to some additional conditions (boundary in most cases). Namely, $\psi(\chi,y)$ should vanish if $\chi$ satisfies these conditions, and $\mathcal{M}$ contains functions that meet some of the conditions. Also, in the case of a solution $w$ and some PINN approximation $\tilde w$, $\psi(\tilde{w},w)$ should be estimated in terms of respective additional condition residual. For this reason, we give the following definition.
        \begin{definition}
            \label{subordinate}
            Let $\mathcal{M}\subset \mathcal{D}_\psi\subset \mathcal{D}(B)\subset X$ be subspaces. Let $B:\mathcal{D}(B)\to Y$ be an operator on $X$ and $\psi:\mathcal{D}_\psi\times\mathcal{M}\to \mathbb{R}$ be a function. We say that $\psi$ is \textit{subordinate to $B$ on $\mathcal{M}$}, if there exists $\rho\in C(Y;\mathbb{R}^{+})$ such that $\xi\to 0~\Rightarrow~\rho(\xi)\to 0$ and
            \begin{equation*}
                |\psi(\chi,y)|\leq \gamma(\chi,y)\rho(B(\chi)-B(y))~
            \end{equation*}
            for every $\chi\in \mathcal{D}_\psi$, $y\in\mathcal{M}$, and some $\gamma(\chi,y)\geq0$. 
        \end{definition}
        \begin{remark}
            In most cases, operator $-A$ is $(p,\psi)$-submonotone on $\mathcal{M}=\mathcal{D}(A)$, or function $\psi$ is subordinate on $\mathcal{M}=\mathcal{D}_\psi$. For these cases, we will say \textit{$(p,\psi)$-submonotone} and \textit{subordinate}, respectively. However, the subspace $\mathcal{M}$ which does not coincide with $\mathcal{D}(A)$ appears sometimes, for instance, in the Navier-Stokes equation (Example \ref{nsnonl}).
        \end{remark}
        Let us consider an essential case of $A$, where $-A$ is submonotone. In \cite{lumerphillips}, G. Lumer and R. S. Phillips introduced the concept of the dissipative operator while working with the semi-inner product. With a real $p$-form, we define an operator that is dissipative on the kernel of some function $\kappa$. 
        \begin{definition}
             Let $\mathcal{D}(A)\subset X$ be a subspace of a $\mathbb{R}$-smooth Banach space $X$. An operator $A:\mathcal{D}(A)\to X$ is said to be \textit{$(p,\kappa)$-dissipative} for some $\kappa:\mathcal{D}(A)\to\mathbb{R}$ and $p>1$ if it is true that  
            \begin{equation*}
                \langle A(y),y\rangle_p\leq \kappa(y),~\forall y\in\mathcal{D}(A).
            \end{equation*}
        \end{definition}
        \begin{example}
            Now, let $A=A_0+F$, where $A_0:\mathcal{D}(A_0)\to X$ is linear $(p,\kappa)$-dissipative, and $F:\mathcal{D}(F)\to X$ is conditionally Lipshitz. The last means that
            \begin{equation*}
                \|F(y_1)-F(y_2)\|\leq \Lambda(y_1,y_2)\|y_1-y_2\|,
            \end{equation*}
            for every $y_1,y_2\in \mathcal{D}(F)$ and some $\Lambda(y_1,y_2)\geq 0$. If we denote $\psi(\chi,y):=\kappa(\chi-y)$, then $-A$ is $(p,\psi)$-submonotone. 
        \end{example}
        Let us consider operators in functional spaces. Not to complicate, we deal with simple derivatives only. The following statement allows us to extend submonotonicity to a positive linear combination and a closure (and, therefore, to week derivatives). The proof is standard. The second property follows from the real $p$-form's continuous property (9 of Lemma \ref{props}). 
        \begin{statement}   
            \label{closure1}
            Let $X$ be a $\mathbb{R}$-smooth Banach space. The following two properties hold.
            \begin{enumerate}
                \item
                If $-A_{1,2}$ are $(p,\psi_{1,2})$-submonotone on $\mathcal{M}$ and ${\rm Re}(\lambda_{1,2})> 0$, then for $A:=\lambda_1A_1+\lambda_2A_2$, its negative $-A$ is $(p,\psi)$-submonotone on $\mathcal{M}$ for some $\psi$, $\Lambda$.
                \item
                Moreover, let $D(\|\cdot\|)(\cdot,\chi)$ be continuous on the ring $\{y\in X\mid r_0<\|y\|<r_1\}$, for every $0<r_0<r_1$ and $\chi\in X$. If $-A$ is $(p,\psi)$-submonotone, there exists a closure $\overline{A}$, $\psi$ is continuous w.r.t. $\overline{A}$-norm, that is
                \begin{equation*}
                    y_n\to y,~\chi_n\to \chi,~Ay_n\to \overline{A}y,~A\chi_n\to \overline{A}\chi~\Rightarrow~\psi(\chi_n,y_n)\to\psi(\chi,y),
                \end{equation*}
                and $\Lambda$ is continuous w.r.t. $\overline{A}$-norm, then $-\overline{A}$ is $(p,\psi)$-submonotone. 
            \end{enumerate}
        \end{statement}
        We start with $(p,\kappa)$-dissipative operators.
        \begin{example}
            Let $X=L^p((a;b);\mathbb{C})$, $p>1$. Then we have a $(p,\kappa)$-dissipative operator
            \begin{equation*}
                Ay=sy',~s\in\mathbb{R}.
            \end{equation*}
            Really, 
            \begin{equation*}
                \begin{aligned}
                    \langle Ay, y\rangle_p=\frac{s}{p}|y|^p\Big|_{a}^{b}
                \end{aligned}
            \end{equation*}
        \end{example}
        \begin{example}
            Let $X=L^p((a;b);\mathbb{C})$, $p\geq 4$ or $p=2$. Then we have a $(p,\kappa)$-dissipative
            \begin{equation*}
                Ay=y''
            \end{equation*}
            Really, 
            \begin{equation*}
                \begin{aligned}
                    &\langle Ay, y\rangle_p\leq |y|^{p-2}{\rm Re}\left(\overline{y}y'\right)\Big|_{a}^{b}
                \end{aligned}
            \end{equation*}
        \end{example}
        \begin{example}
            Let $X=L^2((a;b);\mathbb{C})$. Then we have a $(2,\kappa)$-dissipative
            \begin{equation*}
                Ay=(s_1^2+is_2)y'',~s_1,s_2\in\mathbb{R}
            \end{equation*}
            Really, 
            \begin{equation*}
                \begin{aligned}
                    &\langle Ay, y\rangle_{2}\leq (s_1^2+is_2)\overline{y}y'\Big|_{a}^{b}
                \end{aligned}
            \end{equation*}
        \end{example}
        \begin{example}
            Let $\Omega\subset\mathbb{R}^m$, $\Gamma=\partial\Omega$, $X=L^2(\Omega;\mathbb{R}^n)$, $1\leq j\leq m$. Then
            \begin{equation*}
                A\mathbf{y}=\partial^\sigma_{x_j}Q\mathbf{y},
            \end{equation*}
            is $(2,\kappa)$-dissipative, where $\sigma\geq0$ and $Q\in\mathbb{R}^{n\times n}$ is a such matrix that
            \begin{equation*}
                \begin{cases}
                    (-1)^{\frac{\sigma}{2}}Q\leq 0, & \sigma=2l \\
                    Q=Q^T, & \sigma=2l +1
                \end{cases}
            \end{equation*}
            Really, 
            \begin{equation*}
                \begin{aligned}
                    \langle A\mathbf{y},\mathbf{y}\rangle_{2}&\leq \sum_{\nu=0}^{l-1}\int_{\Gamma}(-1)^{\nu}\left(\partial^{\sigma-1-\nu}_{x_j}Q\mathbf{y}\cdot\partial^\nu_{x_j}\mathbf{y}\right)\left(\mathbf{e}_j\cdot\mathbf{n}_{\Gamma}\right)d\Gamma+\\
                    &+ 
                    \begin{cases}
                        0, &~\sigma=2l\\
                        \frac{(-1)^{l}}{2}\int_{\Gamma}\left(\partial^{l}_{x_j}Q\mathbf{y}\cdot\partial^l_{x_j}\mathbf{y}\right)\left(\mathbf{e}_j\cdot\mathbf{n}_{\Gamma}\right)d\Gamma, &~\sigma=2l+1
                    \end{cases}
                \end{aligned}
            \end{equation*}
        \end{example}
        \begin{example}
            Let $\Omega\subset\mathbb{R}^m$, $\Gamma=\partial\Omega$, $X=L^p(\Omega)$, $p\geq2$. Also, let for every $x\in\Omega$, matrix $Q(x)\in\mathbb{R}^{m\times m}$ satisfy $Q(x)\geq0$ and $Q_{ij}\in L^{\infty}(\Omega)$, for every $1\leq i,j\leq m$. Then we have a $(p,\kappa)$-dissipative operator
            \begin{equation*}
                Ay=\nabla\cdot\left(Q(x)\nabla y\right)
            \end{equation*} 
            Really, 
            \begin{equation*}
                \begin{aligned}
                    &\langle Ay, y\rangle_p\leq  \int_{\Gamma}\left(\left(Q(x)\nabla y\right)\cdot \mathbf{n}_\Gamma\right)|y|^{p-2}yd\Gamma
                \end{aligned}
            \end{equation*}
        \end{example}
        \begin{example}\label{thirdderiv}
            Let $X=L^p((a;b))$, $p>3$ or $p=2$. Then we have a $(p,\kappa)$-dissipative
            \begin{equation*}
                Ay=-y'''
            \end{equation*}
            Really, 
            \begin{equation*}
                \begin{aligned}
                    &\langle Ay, y\rangle_p=-y''|y|^{p-2}y\Big|_{a}^{b}+\frac{(p-1)}{2}(y')^{2}|y|^{p-2}\Big|_{a}^{b}-\frac{(p-1)(p-2)}{2}\int_{a}^{b}(y')^3|y|^{p-4}ydx
                \end{aligned}
            \end{equation*}
            Applying Poincare inequality in $L^3((a;b))$ (Lemma \ref{1dpoincare}) for $\Xi(y)=|y|^{\frac{p}{3}}$,
            \begin{equation*}
                \begin{aligned}
                &\langle Ay, y\rangle_p\leq-y''|y|^{p-2}y\Big|_{a}^{b}+\frac{(p-1)}{2}(y')^{2}|y|^{p-2}\Big|_{a}^{b}+\\
                &+\frac{27(p-1)(p-2)}{4p^3(b-a)^3}\left(\frac{b-a}{2}\left(|y(a)|^{\frac{p}{3}}+|y(b)|^{\frac{p}{3}}\right)^3-\|y\|^{p}_{L^p(a;b)}\right)\leq\\
                &\leq-y''|y|^{p-2}y\Big|_{a}^{b}+\frac{(p-1)}{2}(y'(b))^{2}|y(b)|^{p-2}+\frac{27(p-1)(p-2)}{8p^3(b-a)^2}\left(|y(a)|^{\frac{p}{3}}+|y(b)|^{\frac{p}{3}}\right)^3
                \end{aligned}
            \end{equation*}
        \end{example}
        \begin{example}\label{maxwellop}
            Let $\Omega\subset\mathbb{R}^3$, $X=L^2(\Omega;\mathbb{R}^3)\times L^2(\Omega;\mathbb{R}^3)$. Then we have a $(2,\kappa)$-dissipative
            \begin{equation*}
                A\mathbf{y}=\nabla\times\zeta - \nabla \times \eta,~\mathbf{y}=(\eta,\zeta)
            \end{equation*}
            Really,
            \begin{equation*}
                \begin{aligned}
                    \langle A\mathbf{y},\mathbf{y}\rangle_2=\int_{\Gamma}(\zeta\times \eta)\cdot n_\Gamma d\Gamma
                \end{aligned}
            \end{equation*}
        \end{example}   
        Now, let us consider other submonotone operators. 
        \begin{example}
            \label{yyderiv}
            Let $X=L^p((a;b))$, $p>1$. We consider the following operator
            \begin{equation*}
                A(y)=syy',~s\in\mathbb{R}
            \end{equation*}
            Then $-A$ is $(p,\psi)$-submonotone. Really, if $\hat{y}=y_1-y_2$, then 
            \begin{equation*}
                \begin{aligned}
                    &\langle A(y_1)-A(y_2),\hat{y}\rangle_p\leq \left(\frac{s}{p+1}|\hat{y}|^{p+1}+\frac{s}{p}|\hat{y}|^{p}y_2\right)\Big|_{a}^{b}+|s|\left(1+\frac{1}{p}\right)\|y'_2\|_{C([a;b])}\|\hat{y}\|^{p}_{L^p((a;b))}.
                \end{aligned}
            \end{equation*}
        \end{example}
        The following example is a multidimensional case of Example \ref{yyderiv}; yet it's negative is submonotone on a subspace.
        \begin{example}
            \label{nsnonl}
            Here we consider in $L^p(\Omega;\mathbb{R}^m)$, $p>1$. Let $\Omega\subset\mathbb{R}^m$, $\Gamma=\partial\Omega$, $X=L^p(\Omega;\mathbb{R}^m)$, and 
            \begin{equation*}
                A(\mathbf{y})=s(\mathbf{y}\cdot\nabla)\mathbf{y},~s\in\mathbb{R}.
            \end{equation*}
            Then $-A$ is $(p,\psi)$-submonotone on $\mathcal{M}=\{\mathbf{y}\mid \nabla\cdot \mathbf{y}=0\}$. Really, if we put $\hat{\mathbf{y}}=\mathbf{\chi}-\mathbf{y}$, $\mathbf{y}\in\mathcal{M}$, then
            \begin{equation*}
                \begin{aligned}
                    &\langle A(\mathbf{\chi})-A(\mathbf{y}),\hat{\mathbf{y}}\rangle_p\leq \frac{s}{p}\int_{\Gamma}|\hat{\mathbf{y}}|^{p}\left(\mathbf{\hat{\mathbf{y}}}\cdot\mathbf{n}_\Gamma\right)d\Gamma+\\
                    &+|s|\frac{\|\mathbf{\chi}\|_{C(\overline{\Omega};\mathbb{R}^m)}+\|\mathbf{y}\|_{C(\overline{\Omega};\mathbb{R}^m)}}{p^2}\left((p-1)\|\hat{\mathbf{y}}\|^{p}_{L^p(\Omega;\mathbb{R}^m)}+\|\nabla\cdot\hat{\mathbf{y}}\|^p_{L^p(\Omega)}\right)+\\
                    &+m|s|\|\nabla\mathbf{y}\|_{L^\infty(\Omega;\mathbb{R}^{m\times m})}\|\hat{\mathbf{y}}\|^p_{L^p(\Omega;\mathbb{R}^m)}+\frac{s}{p}\int_{\Gamma}|\hat{\mathbf{y}}|^{p}\left(\mathbf{y}\cdot\mathbf{n}_\Gamma\right)d\Gamma
                \end{aligned}
            \end{equation*}
            Let us note that $\psi$ contains the term with $\nabla\cdot\hat{\mathbf{y}}$, which is subordinate to the continiuty condition operator.
        \end{example}
    \subsection{Powered operators and submonotone operators w.r.t. norm}\label{powss}
        In \cite{baikoleymishramolinaro}, the authors considered the one-dimensional Camassa-Holm equation involving mixed derivative $\partial_t\partial^2_{x}$. For a sufficiently regular solution, one can take it as $-\frac{d}{dt}U$, where $U=-\partial^2_x$. For such cases, we extend a concept of submonotonicity/subordination (for simplicity, we deal with the whole domain only). Precisely, we define operators $U$ and $A(t)$ used in (\ref{genpartype}). Remarkably, we can use such a definition of $U$ in a hyperbolic-type equation (\ref{hyptype}). 

        \begin{definition}
            Let $X$ be a $\mathbb{R}$-smooth Banach space, and $\mathcal{D}(U)\subset X$ be a subspace. Let $U:\mathcal{D}(U)\to X$ be a linear operator on $X$. We will say that $U$ is \textit{$(p,\epsilon)$-powered} for some $\epsilon:\mathcal{D}(U)\times\mathcal{D}\left(U^{\frac{1}{p}}\right)\to\mathbb{R}$, if there exists a linear operator $U^{\frac{1}{p}}:\mathcal{D}\left(U^{\frac{1}{p}}\right)\to X$, $\mathcal{D}(U)\subset \mathcal{D}\left(U^{\frac{1}{p}}\right)$, satisfying
            \begin{equation*}
                \begin{aligned}
                    \langle U\chi,y\rangle_p=\epsilon(\chi,y)+\left\langle U^{\frac{1}{p}}\chi,U^{\frac{1}{p}}y\right\rangle_p,~\forall \chi\in\mathcal{D}(U),~\forall y\in\mathcal{D}\left(U^{\frac{1}{p}}\right).
                \end{aligned}
            \end{equation*}
        \end{definition}
        Now, we unveil the following extension of Definition \ref{submonotone}.
        \begin{definition}
            Let $X$ be a $\mathbb{R}$-smooth Banach space, $\{U_k\}_{k=1}^{K}$ be a set of $(p,\psi_k)$-powered operators, and $\mathcal{D}(A)\subset\mathcal{D}(U_k)\subset\mathcal{D}\left(U^{\frac{1}{p}}_k\right)$ be subspaces. We say that $-A$ is \textit{($p,\psi)$-submonotone, w.r.t. norm $\|\cdot\|_{\sum U_k^{\frac{1}{p}}}$} for some $\psi:\left(\mathcal{D}(A)\right)^2\to\mathbb{R}$, if 
            \begin{equation*}
                \langle A(\chi-y),\chi-y\rangle_p\leq \psi(\chi,y)+\Lambda(\chi,y)\sum_{k=1}^{K}\left\|U^{\frac{1}{p}}_k\chi-U^{\frac{1}{p}}_ky\right\|^p,
            \end{equation*}
            for all $\chi,y\in\mathcal{D}(A)$ and some $\Lambda(\chi,y)\geq 0$.
        \end{definition}
        For generalized parabolic and hyperbolic cases, we need different definitions of subordination. 
        \begin{definition}
            Let $\mathcal{D}_{\epsilon_{2}}\subset\mathcal{D}(B),\mathcal{D}_{\epsilon_1}\subset X$ be subspaces. Let $B:\mathcal{D}(B)\to Y$ be an operator on $X$ and $\epsilon:\mathcal{D}_{\epsilon_1}\times\mathcal{D}_{\epsilon_2}\to \mathbb{R}$ be a function. We say that $\epsilon$ is \textit{subordinate to $B$} if there exists $\rho\in C(Y;\mathbb{R}^{+})$ such that $\xi\to0~\Rightarrow~\rho(\xi)\to 0$ and the following condition hold,
            \begin{equation*}
                |\epsilon(\chi,y_1-y_2)|\leq \gamma(\chi, y_1,y_2)\rho(B(y_1)-B(y_2))\,
            \end{equation*}
            for every $\chi\in\mathcal{D}_{\epsilon_1}$, $y_1,y_2\in\mathcal{D}_{\epsilon_2}$ and some $\gamma(\chi,y_1,y_2)\geq0$.
        \end{definition}
        \begin{definition}
            Let $\mathcal{D}_{\epsilon_{1,2}}\subset\mathcal{D}(B_{1,2})\subset X$, $\mathcal{D}_{\epsilon_{1}}\subset \mathcal{D}_{\epsilon_{2}}$ be subspaces. Let $B_{1,2}:\mathcal{D}(B_{1,2})\to Y_{1,2}$ be operators on $X$ and $\epsilon:\mathcal{D}_{\epsilon_1}\times\mathcal{D}_{\epsilon_2}\to \mathbb{R}$ be a function. We say that $\epsilon$ is \textit{subordinate to $B_1$ and $B_2$} if there exist $\rho_{1,2}\in C(Y_{1,2};\mathbb{R}^{+})$ such that $\xi\to0~\Rightarrow~\rho_{1,2}(\xi)\to 0$ and the following condition holds,
            \begin{equation*}
                |\epsilon(\chi_1-\chi_2,y_1-y_2)|\leq \gamma(\chi_1,\chi_2,y_1,y_2)\left[\rho_1(B_1(\chi_1)-B_1(\chi_2))+\rho_2(B_2(y_1)-B_2(y_2))\right]\,
            \end{equation*}
            for every $\chi_1,\chi_2\in\mathcal{D}_{\epsilon_1}$, $y_1,y_2\in\mathcal{D}_{\epsilon_2}$ and some $\gamma(\chi_1,\chi_2,y_1,y_2)\geq0$.
        \end{definition}
        \begin{remark}
            One can also consider properties of submonotone w.r.t. norm operators and an extension of $(p,\epsilon)$-powered property to a closure, similar to Statement \ref{closure1}.
        \end{remark}
        We start with $(p,\epsilon)$-powered operators. 
        \begin{example}
            Let $X=L^p(\Omega;\mathbb{R}^n)$, and $U={\rm diag}(\varepsilon_1,\dots,\varepsilon_n)$, $\varepsilon_i>0$, that is
            \begin{equation*}
                U:\mathbf{y}=(y_1,\dots,y_n)^{T}\mapsto (\varepsilon_1y_1,\dots,\varepsilon_ny_n)^{T}=U\mathbf{y},~\forall \mathbf{y}\in L^p(\Omega;\mathbb{R}^n).
            \end{equation*}
            Then $U$ is $(p,\epsilon)$-powered with $U^{\frac{1}{p}}={\rm diag}\left(\varepsilon_1^{\frac{1}{p}},\dots,\varepsilon_n^{\frac{1}{p}}\right)$ and $\epsilon\equiv0$. Such operators appear, for instance, in the Maxwell equation. Furthermore, if $p=2$, $U$ can be any symmetric, positively defined matrix. 
        \end{example}
        \begin{example}
            \label{derivpow}
            Let $X=L^2((a;b))$ (one can also consider $\mathbb{C}$-valued functions), $U=(-1)^\sigma\partial^{2\sigma}_{x}$. Then $U$ is $(2,\epsilon)$-powered and $\sqrt{U}=\partial^{\sigma}_x$. Really,
            \begin{equation*}
                \langle U\chi,y\rangle_2=\sum_{\nu=1}^{\sigma}(-1)^{\sigma+\nu-1}\chi^{(2\sigma-\nu)}y^{(\nu)}\Big|_{a}^{b}+\left\langle\chi^{(\sigma)},y^{(\sigma)}\right\rangle_2
            \end{equation*}
            Moreover, respective $\sqrt{U}$-norm is the Sobolev $H^{\sigma}((a;b))$-norm.
        \end{example}
        Now, let us turn to a submonotone w.r.t. a norm operator.
        \begin{example}
            \label{nonlinearities}    
            Let $X=L^2((a;b))$. Then for
            \begin{equation*}
                A(y)=s\frac{\frac{d}{dx}\left(y^2y^{(\sigma)}\right)}{y} = 2sy'y^{(\sigma)}+syy^{(\sigma+1)},~s\in\mathbb{R}
            \end{equation*}
            where $\sigma\geq 1$, $-A$ is $(2,\psi)$-submonotone w.r.t. Sobolev $H^{\left\lceil\frac{\sigma}{2}\right\rceil}((a;b))$-norm (see Example \ref{derivpow}). 
            Really, if $\hat{y}=y_1-y_2$,
            \begin{equation*}
                \begin{aligned}
                    &\langle A(y_1)-A(y_2), \hat{y}\rangle_{2}\leq s\hat{y}^2y_1^{(\sigma)}\Big|_{a}^{b}+sy_2\hat{y}\hat{y}^{(\sigma)}\Big|_{a}^{b}+s\sum_{\nu=0}^{l-1}(-1)^{\nu}\hat{y}^{(\sigma-1-\nu)}\left(\hat{y}y'_2-\hat{y}'y_2\right)^{(\nu)}\Big|_{a}^{b}+\\
                    &+s\begin{cases}
                        2^{l+1}\|y_2\|_{C^{l+1}([a;b])}\|\hat{y}\|^2_{H^{l+1}((a;b))},~&\sigma=2l+1\\
                        (2^{l+2}+1)\|y_2\|_{C^{l+2}([a;b])}\|\hat{y}\|^2_{H^{l+1}((a;b))}+\frac{1}{2}\left|\left(\hat{y}^{(l+1)}\right)^2y_2\Big|_{a}^{b}\right|,~&\sigma=2l+2
                    \end{cases}
                \end{aligned}
            \end{equation*}
        \end{example}
    \subsection{Coercive operators}\label{coerss}
        The last type of PDE we consider is elliptic (\ref{elltype}). In \cite{mishramolinaro}, the authors dealt with similar equations and took operator $A$ to have a conditional Lipschitz inverse operator. With a $\mathbb{R}$-smooth Banach space, we consider instead an extended version. Again, we assume the case of the whole domain only. 
        
        \begin{definition}
            Let $X$ be a $\mathbb{R}$-smooth Banach space, $\mathcal{D}(A)\subset X$ be a subspace. We will say that an operator $A:\mathcal{D}(A)\to X$ is \textit{$(p,\psi)$-coercive}, for some $\psi:\mathcal{D}(A)\times\mathcal{D}(A)\to\mathbb{R}$ and $p>1$, if it is true that  
            \begin{equation*}
                \|\chi-y\|^{p}\leq \psi(\chi,y)+\Lambda(\chi,y)\langle A(\chi)-A(y),\chi-y\rangle_{p}
            \end{equation*}
            for every $\chi,y\in \mathcal{D}(A)$, and some $\Lambda(\chi,y)\geq 0$. 
        \end{definition}
        \begin{remark}\label{remclos}
            Again, we can consider an extension of $(p,\psi)$-coercivity to a positive linear combination and a closure similarly to Statement \ref{closure1}. 
        \end{remark}

        It is natural to consider an elliptic operator. For simplicity, we deal with its first term only.
        \begin{example}\label{ellipticop}
            Let $\Omega\subset\mathbb{R}^m$ be a Lipschitz domain, $\Gamma=\partial\Omega$, $X=L^p(\Omega)$, $p\geq2$. Also, let for every $x\in\Omega$ and some $\varepsilon>0$, matrix $Q(x)\in\mathbb{R}^{m\times m}$ satisfy 
            \begin{equation*}
                Q^T(x)=Q(x)~\text{and}~Q(x)\xi\cdot\xi\geq \varepsilon\|\xi\|^2_{\mathbb{R}^m},~\xi\in\mathbb{R}^m.
            \end{equation*}
            If $Q_{ij}\in L^{\infty}(\Omega)$, for every $1\leq i,j\leq m$, then an elliptic operator
            \begin{equation*}  
                Ay=-\nabla\cdot (Q(x)\nabla y)
            \end{equation*}
            is $(p,\psi)$-coercive.

            Really, if $\hat{y}=y_1-y_2$, and applying the $L^2$-Poincare inequality with trace term for $\Xi(\hat{y})={\rm sgn(\hat{y})}|\hat{y}|^{\frac{p}{2}}$ with constant $\pi_{2,tr}>0$, we have
            \begin{equation*} 
                \begin{aligned}
                    &\|\hat{y}\|^p_{L^p(\Omega)}\leq 2\pi_{2,tr}\left(\int_{\Omega}|\hat{y}|^{p-2}\|\nabla\hat{y}\|^2_{\mathbb{R}^m}dx+\|\hat{y}\|^p_{L^p(\Gamma)}\right)\leq\\
                    &\leq 2\pi_{2,tr}\left(\frac{p}{4(p-1)\varepsilon}\int_{\Gamma}\left((Q(x)\nabla \hat{y})\cdot\mathbf{n}_\Gamma\right)\hat{y}|\hat{y}|^{p-2}d\Gamma + \int_{\Gamma}|\hat{y}|^pd\Gamma\right)+\frac{\pi_{2,tr}p^2}{2(p-1)\varepsilon}\langle A\hat{y},\hat{y}\rangle_{p}
                \end{aligned}
            \end{equation*}
        \end{example}

        \begin{remark}
            Even though, we introduced the real $p$-form to work with time derivatives of $\|\cdot\|^p$, we use it to define $(p,\psi)$-coercive operators in elliptic equations that don't require this. One can proceed even further, dealing with the form $\langle\cdot,\cdot\rangle_{p}:X_1\times X_2\to\mathbb{R}$ that satisfies the following property
            \begin{equation*}
                |\langle y,\chi\rangle_p|\leq \|y\|_{X_1}\|y\|_{X_2}^{p-1},~\forall y\in X_1,~\forall \chi\in X_2
            \end{equation*}
            Thus, the case of the real $p$-form is the particular case $X_1=X_2=X$. Another example is a Banach space $X_2=X$ with its dual space $X_1=X^{*}$, where $\langle \varphi, y\rangle_2 := {\rm Re}[\varphi(y)]$. Theorem \ref{theq} will still hold for such a generalization.
        \end{remark}

\section{PINN's error estimation}
    This section presents PINN's error estimates for PDEs' types (\ref{partype})--(\ref{elltype}). Also, in the case of semilinear parabolic-type equations, we compare the estimation with an approach in \cite{hillebrechtunger1}. 
    
    Let us note that we describe in detail only the case of parabolic-type equations (\ref{partype}). In other cases, for simplicity, we omit some conditions. For instance, we give an estimate based on training error (Corollary \ref{corpar}) for the parabolic-type equations (\ref{partype}) only. Also, we consider the case of submonotonicity/subordination on a subspace for this type only. For other cases, we deal with the whole domain. 

    \subsection{Parabolic-type equation}\label{parss}
        Consider the problem (\ref{eq}) in a $\mathbb{R}$-smooth Banach space $X$
        \begin{equation}
            \label{eq}
            \begin{aligned}
                &\frac{dw}{dt}=A(t)(w)\\
                &w\Big|_{t=0}=w_0\\
                &B(t)(w)=0
            \end{aligned}
        \end{equation}
        where $A(t):\mathcal{D}(A(t))\to X$ and $B(t):\mathcal{D}(B(t))\to Y$ are nonlinear operators, $\mathcal{D}(A(t))\subset \mathcal{D}(B(t))\subset X$ are subspaces, $Y$ is a Banach space. We also consider the following subspace of $\mathcal{D}(A(t))$
        \begin{equation}
            \label{preimage}
            \mathcal{M}(A(t),B(t)):={\rm Ker}(B(t))\cap\mathcal{D}(A(t))
        \end{equation}
        For a solution of \ref{eq}, we assume $w\in C(\overline{I};X)\cap C^1(I;X)$ and $w(t)\in \mathcal{M}(A(t),B(t))$ for $t\in I$.

        Let us consider a neural network $w_\theta$ with parameter $\theta$, approximating solution $w$ of (\ref{eq}), and the following PINN residuals
        \begin{equation}
            \label{resid}
            \begin{aligned}
                &\mathcal{R}_{eq}=\frac{dw_\theta}{dt}-A(t)(w_\theta),\\
                &\mathcal{R}_{in}=w_\theta|_{t=0}-w_0,\\
                &\mathcal{R}_{bn}=B(t)(w_\theta).
            \end{aligned}
        \end{equation}
        We are interested in the following total error
        \begin{equation}
            \label{totalerr}
            \mathcal{E}:=\|w_\theta-w\|^{q}_{L^q(I; X)}
        \end{equation}

        Also, we need to consider approximating rules for powered norms of residuals (quadrature rules in particular).
        \begin{enumerate}
            \item
            Let $\mathcal{T}_{eq}\subset \overline{I}\times\mathbb{R}^m$, and let there exist $\mathcal{Q}_{M,eq}:L^p(I; X)\times \mathcal{T}_{eq}^{M}\to\mathbb{R}$ with
            \begin{equation*}
                \left|\|w\|^p_{L^p(I; X)}-\mathcal{Q}_{M,eq}(w,(t_1,x_1),\dots,(t_M,x_M))\right|\leq \beta^{*}_{eq}(w)M^{-\alpha_{eq}},
            \end{equation*}
            for some $\beta^{*}_{eq}(w)\geq 0$, $\alpha_{eq}>0$ and every $M\in\mathbb{N}$, $w\in L^p(I; X)$, and $\{(t_i,x_i)\}_{i=1}^{M}\subset \mathcal{T}_{eq}$. 
            \item
            Let $\mathcal{T}_{in}\subset \mathbb{R}^m$, and let there exist $\mathcal{Q}_{M,in}:X\times \mathcal{T}_{in}^{M}\to\mathbb{R}$ with
            \begin{equation*}
                \left|\|y\|^p_{X}-\mathcal{Q}_{M,in}(y,x_1,\dots,x_M)\right|\leq \beta^{*}_{in}(y)M^{-\alpha_{in}}
            \end{equation*}
            for some $\beta^{*}_{in}(y)\geq 0$, $\alpha_{in}>0$ and every $M\in\mathbb{N}$, $y\in X$, and $\{x_i\}_{i=1}^{M}\subset \mathcal{T}_{in}$. 
            \item
            Let $\mathcal{T}_{bn}\subset \overline{I}$, and let there exist $\mathcal{Q}_{M,bn}:L^1(I)\times \mathcal{T}_{bn}^{M}\to\mathbb{R}$ with
            \begin{equation*}
                \left|\|g\|^p_{L^p(I;Y)}-\mathcal{Q}_{M,bn}(g,t_1,\dots,t_M)\right|\leq \beta^{*}_{bn}(g)M^{-\alpha_{bn}}
            \end{equation*}
            for some $\beta^{*}_{bn}(g)\geq 0$, $\alpha_{bn}>0$ and every $M\in\mathbb{N}$, $g\in L^p(I;Y)$, and $\{t_i\}_{i=1}^{M}\subset \mathcal{T}_{bn}$. 
        \end{enumerate}

        Given training sets $\left\{(t_i,x_i)\right\}_{i=1}^{M_{eq}}\subset \mathcal{T}_{eq}$, $\left\{x_i\right\}_{i=1}^{M_{in}}\subset \mathcal{T}_{in}$, $\left\{t_i\right\}_{i=1}^{M_{bn}}\subset \mathcal{T}_{bn}$, we have the following training errors
        \begin{equation*}
            \begin{aligned}
                &\mathcal{E}_{T,eq}=\mathcal{Q}_{M_{eq},eq}(\mathcal{R}_{eq},(t_1,x_1),\dots,(t_{M_{eq}},x_{M_{eq}})),~\\
                &\mathcal{E}_{T,in}=\mathcal{Q}_{M_{in},in}(\mathcal{R}_{in},x_1,\dots,x_{M_{in}}),~\\
                &\mathcal{E}_{T,bn}=\mathcal{Q}_{M_{bn},bn}(\mathcal{R}_{bn},t_1,\dots,t_{M_{bn}})
            \end{aligned}
        \end{equation*}
        \begin{theorem}
            \label{th2}
            Let us given a solution $w$ of problem (\ref{eq}), and a neural network $w_\theta$ with residuals (\ref{resid}) and total error (\ref{totalerr}). Let, moreover, $-A(t)$ be $(p,\psi(t))$-submonotone on $\mathcal{M}(A(t),B(t))$, with $\psi(t)$ subordinate to $B(t)$ on $\mathcal{M}(A(t),B(t))$, for every $t\in\overline{I}$. ($\mathcal{M}(A(t),B(t))$ is defined in (\ref{preimage})). Furthermore, let respective $\gamma(\cdot,w_\theta(\cdot),w(\cdot))\in C(\overline{I})$, $\rho(t)\equiv \rho$, and $\Lambda(\cdot,w_\theta(\cdot),w(\cdot))\in L^1(I)$. Then
            \begin{equation*}
                \mathcal{E}\leq \mathcal{C}^{\frac{q}{p}}\left(\frac{p(e^{\frac{q(p-1)T}{p}}-1)}{q(p-1)}\right)e^{q\|\Lambda(\cdot,w_\theta(\cdot),w(\cdot))\|_{L^1(I)}},
            \end{equation*}
            where
            \begin{equation*}
                \begin{aligned}
                    \label{cth2}
                    \mathcal{C}&= \|\mathcal{R}_{eq}\|^p_{L^p(I; X)}+\|\mathcal{R}_{in}\|^p+p\|\gamma(\cdot,w_\theta(\cdot),w(\cdot))\|_{C(\overline{I})}\|\rho(\mathcal{R}_{bn})\|_{L^1(I)}
                \end{aligned}
            \end{equation*}
        \end{theorem} 
        \begin{proof}
            Let $\hat{w}:=w_\theta-w$. Then with (\ref{dnormpowp}) and other properties of $p$-form, the Young inequality, $(p,\psi(t))$-submonotonicity of $-A(t)$ and subordination of $\psi(t)$, we have
            \begin{equation*}  
                \frac{d}{dt}\|\hat{w}\|^p\leq \|\mathcal{R}_{eq}\|^p+p\gamma(t,w_\theta(t),w(t))\rho(\mathcal{R}_{bn})+(p-1+p\Lambda(t,w_\theta(t),w(t)))\|\hat{w}\|^p
            \end{equation*}
            Integrating from $0$ to $t\leq T$ and applying the Gr\"onwall-Bellman lemma, we obtain
            \begin{equation*}
                \|\hat{w}\|^p\leq \mathcal{C}e^{(p-1)t+p\int_{0}^{t}(\Lambda(t,w_\theta(t),w(t)))d\tau},
            \end{equation*}
            where $\mathcal{C}$ is described in (\ref{cth2}). It remains to take $\frac{q}{p}$-power and integrate on $t$, $1\leq q<\infty$.
        \end{proof}
        \begin{remark}
            In Theorem \ref{th2}, instead of $L^q(I; X)$ norm, one can take another one. Also, we assume that solution $w$ and network $w_\theta$ are sufficiently regular to have finite $L^q$-norm.
        \end{remark} 
        \begin{corollary}
            \label{corpar}
            Let us consider conditions of Theorem \ref{th2} and training errors. If, moreover, 
            \begin{equation*}
                \|\rho(g)\|_{L^1(I)}\leq h(\|g\|_{L^p(I;Y)}),
            \end{equation*} 
            for some monotone $h:\mathbb{R}^{+}\to\mathbb{R}^{+}$, $\mu\to0~\Rightarrow~h(\mu)\to0$, then we have
            \begin{equation*}
                \mathcal{E}\leq \tilde{\mathcal{C}}^{\frac{q}{p}}\left(\frac{p(e^{\frac{q(p-1)T}{p}}-1)}{q(p-1)}\right)e^{q\|\Lambda(\cdot,w_\theta(\cdot),w(\cdot))\|_{L^1(I)}},
            \end{equation*}
            where
            \begin{equation}
                \label{parexact}
                \begin{aligned}
                    \tilde{\mathcal{C}}&=\mathcal{E}_{T,eq}+\beta^{*}_{eq}(w_\theta,w)M_{eq}^{-\alpha_{eq}}+\mathcal{E}_{T,in}+\beta^{*}_{in}(w_\theta,w)M_{in}^{-\alpha_{in}}+\\
                    &+p\|\gamma(\cdot,w_\theta(\cdot),w(\cdot))\|_{C(\overline{I})}h\left(\left(\mathcal{E}_{T,bn}+\beta^{*}_{bn}(w_\theta,w)M_{bn}^{-\alpha_{bn}}\right)^{\frac{1}{p}}\right)
                \end{aligned}
            \end{equation}
            or asymptotically
            \begin{equation}
                \label{parasymp}
                \mathcal{E}=\mathcal{O}\left(\mathcal{E}_{T,eq}+M_{eq}^{-\alpha_{eq}}+\mathcal{E}_{T,in}+M_{in}^{-\alpha_{in}}+h\left(\left(\mathcal{E}_{T,bn}+M_{bn}^{-\alpha_{bn}}\right)^{\frac{1}{p}}\right)\right)
            \end{equation}
        \end{corollary} 
        \begin{remark}
            Approximation (quadrature) constants depend on some norm of residuals. $h$ can depend on domain. 
        \end{remark}  
        \begin{remark}
            One can substitute (\ref{eq}) (and therefore Theorem \ref{th2}) with the more general equation, dealing instead with parametrized family $\{A(t,z)\}_{t\in I, z\in Z}$, where the family is uniformly submonotone. The main feature of the parameter is its ability to be approximated by neural network $z_\theta$, even if we don't estimate its error $\|z-z_\theta\|$. One example of such a problem is an eigenvalue problem, where $Z=\mathbb{C}$, $z$ is an eigenvalue approximated by two network parameters $z_\theta=\theta_{z,1}+i\theta_{z,2}$. Another example is the Navier-Stokes equation, where $z$ represents the pressure term. 
        \end{remark}
    \subsection{Parabolic-type equation, non-smooth Banach space}\label{parnss}
        We briefly describe an extension of the approach presented in \cite{hillebrechtunger1} to the semilinear problem (\ref{eqalt}) in arbitrary Banach space $X$ and compare the obtained estimate with that of the previous subsection. 
        \begin{equation}
            \label{eqalt}
            \begin{aligned}
                &\frac{dw}{dt}=Aw+F(t)(w)\\
                &w\Big|_{t=0}=w_0\\
                &B(w)=0
            \end{aligned}
        \end{equation}
        where $A:\mathcal{D}(A)\to X$ and $B:\mathcal{D}(B)\to Y$ are linear operators, $\mathcal{D}(A)\subset \mathcal{D}(B)\subset X$ are subspaces, $Y$ is a Banach space. Again, we consider the following subspace of $\mathcal{D}(A)$
        \begin{equation*}
            \mathcal{M}(A,B):={\rm Ker}(B)\cap\mathcal{D}(A).
        \end{equation*}
        Residuals are
        \begin{equation}
            \label{residalt}
            \begin{aligned}
                &\mathcal{R}_{eq}=\frac{dw_\theta}{dt}-Aw_\theta-F(t)(w_\theta),\\
                &\mathcal{R}_{in}=w_\theta|_{t=0}-w_0,\\
                &\mathcal{R}_{bn}=B(w_\theta).
            \end{aligned}
        \end{equation}
        \begin{remark}
            For basics of the operator semigroup theory, see, for instance, \cite{pazy}. For generator $A$ of a strongly continuous semigroup $V(t)$, $-A$ is ''submonotone'' in the following sense. In a $\mathbb{R}$-smooth Banach space, an operator of form $-(A_0+s I)$, where $A_0$ is dissipative and $s\in\mathbb{R}$, is submonotone. Since $A$ is a generator of strongly continuous semigroup, then for some $\varepsilon\geq 1$ and $s\in\mathbb{R}$, it is true that
            \begin{equation*}
                \|(\lambda I - A)^{-1}\|\leq \frac{\varepsilon}{\lambda-s},~\forall \lambda>s.
            \end{equation*}
            Then, for $A_0:=A-s I$, we have
            \begin{equation*}
                \|(\lambda I - A_0)^{-1}\|\leq \frac{\varepsilon}{\lambda},~\forall \lambda>0.
            \end{equation*}
            There exists an equivalent norm $\|\cdot\|_{V}$ for which
            \begin{equation*}
                \|(\lambda I - A_0)^{-1}\|_{V}\leq \frac{1}{\lambda},~\forall \lambda>0,
            \end{equation*}
            that is, $A_0$ is dissipative (in generalized meaning). 
        \end{remark}
        Returning to problem (\ref{eqalt}), we have the following result.
        \begin{theorem} 
            Let us given a solution $w$ of problem (\ref{eqalt}), and a neural network $w_\theta$ with residuals (\ref{residalt}) and total error (\ref{totalerr}). Let $A\Big|_{\mathcal{M}(A,B)}$ be a generator of a semigroup $V(t)$, $\|V(t)\|\leq \varepsilon e^{s t}$. Let $B$ be right invertible, and for a right inverse $\Theta$, let $\left(A\Big|_{\mathcal{M}(A,B)}\right)\cdot \Theta$ be bounded. Also, let $F(t)$ be conditionally Lipschitz for every $t\in\overline{I}$ and $\Lambda(\cdot,w_\theta(\cdot),w(\cdot))\in L^1(I)$. Then we have an estimate
            \begin{equation}
                \label{sharpest}
                \mathcal{E}\leq \mathcal{C}\int_{0}^{T}e^{\frac{\varepsilon\|\Lambda(\cdot,w_\theta(\cdot),w(\cdot))\|_{L^1(I)}e^{s t}}{s}}dt
            \end{equation}
            where
            \begin{equation*}
                \begin{aligned}
                    &\mathcal{C}=\varepsilon e^{s T}\|\mathcal{R}_{eq}\|_{L^1(I; X)}+\varepsilon e^{s T}\|\mathcal{R}_{in}\|+\|\Theta\mathcal{R}_{bn}\|_{C(\overline{I};X)}+\\
                    &+\varepsilon e^{s T}\|\Theta\mathcal{R}_{bn}(0)\|+\varepsilon e^{s T}\|A\Theta\|\|\mathcal{R}_{bn}\|_{L^1(I; X)}+\varepsilon e^{s T}\left\|\frac{d\Theta\mathcal{R}_{bn}}{dt}\right\|_{L^1(I; X)}
                \end{aligned}
            \end{equation*} 
        \end{theorem}
        \begin{proof}
            Given fixed $w$ and $w_\theta$, let again $\hat{w}=w_\theta-w$. Then $\tilde{w}:=\hat{w}-\Theta\mathcal{R}_{bn}\in\mathcal{M}(A,B)$ is a solution to the following problem 
            \begin{equation*}
                \begin{aligned}
                    &\frac{d\tilde{w}}{dt}=A\tilde{w}+v(t),\\
                    &\tilde{w}|_{t=0}=\mathcal{R}_{in}-\Theta\mathcal{R}_{bn}(0),
                \end{aligned}
            \end{equation*}
            where 
            \begin{equation*}
                \begin{aligned}
                    &v=\mathcal{R}_{eq}+F(t)(w_\theta)-F(t)(w)+A\Theta\mathcal{R}_{bn}-\frac{d\Theta\mathcal{R}_{bn}}{dt}
                \end{aligned}
            \end{equation*}
            Then $\tilde{w}$ is also a ''mild'' solution, that is, it satisfies the following integral equation
            \begin{equation*}
                \tilde{w}(t)=V(t)\tilde{w}|_{t=0}+\int_{0}^{t}V(t-\tau)v(\tau)d\tau.
            \end{equation*}
            Hence,
            \begin{equation*}
                \begin{aligned}
                     \|\hat{w}\|&\leq \mathcal{C}+\varepsilon\int_{0}^{t}\Lambda(t,w_\theta(t),w(t))e^{s(t-\tau)}\|\hat{w}\|d\tau,
                \end{aligned}
            \end{equation*}
            for a.e. $t\in\overline{I}$. With the Gr\"onwall-Bellman lemma we have desired estimate. 
        \end{proof}
        \begin{remark}
            Let us compare this estimate with the one of Theorem \ref{th2}. (\ref{sharpest}) contains the $W^{1,1}(I; Y)$ norm of the additional conditions residual and a double exponent on $T$ if the semigroup is not uniformly bounded.
        \end{remark}
        The assumption of right invertibility is natural, as the following example demonstrates.
        \begin{example}
            Let $y\in C^1([a;b])$ and $By=(y(b)-y(a),y'(b)-y'(a))^{T}\in \mathbb{R}^2$. Let us denote $\xi_1:=y(b)-y(a)$, $\xi_2:=y'(b)-y'(a)$. We need to construct such a linear operator $\Theta:\mathbb{R}^2\to \mathcal{D}(B)$, that $A\Theta$ is bounded for some $A$. 

            We just put $\Theta$ to be a polynomial of degree $\leq2$, and 
            \begin{equation*}
                \Theta(\xi)(x):=\frac{\xi_2}{2(b-a)}x^2+\frac{2\xi_1-(b+a)\xi_2}{2(b-a)}x.
            \end{equation*}
            One can show that $B\Theta(\xi)=\xi$. Since ${\rm ran}(\Theta)$ is finite dimensional, $A\Theta$ is bounded for every linear operator $A$.
        \end{example}
    \subsection{Generalized parabolic-type equation}\label{gparss}
        For simplicity, we deal with submonotonicity/subordination on the whole domain. Also, for this reason, we put $q=p$. Consider problem (\ref{eqmixed}) in a $\mathbb{R}$-smooth Banach space $X$.
        \begin{equation}
            \label{eqmixed}
            \begin{aligned}
                &\frac{d}{dt}\sum_{k=1}^{K}U_kw=A(t)(w)\\
                &w\Big|_{t=0}=w_0\\
                &B(t)(w)=0,
            \end{aligned}
        \end{equation}
        where $A(t):\mathcal{D}(A(t))\to X$ and $B(t):\mathcal{D}(B(t))\to Y$ are nonlinear operators, $\mathcal{D}(A(t))\subset \mathcal{D}(B(t))\subset X$ are subspaces, $Y$ is a Banach space. Also, $U_k$ are $(p,\epsilon_k)$-powered, ${\rm Im}(U_k)\subset\mathcal{D}(B(t))$, $\mathcal{D}(A(t))\subset\mathcal{D}(U_k)\subset\mathcal{D}\left(U_{k}^{\frac{1}{p}}\right)$, $U_k$ and $U^{\frac{1}{p}}_k$ commute with $\frac{d}{dt}$, $U^{\frac{1}{p}}_k$ commute with $\Big|_{t=0}$, $1\leq k\leq K$. For a solution of (\ref{eqmixed}), we assume $w\in C(\overline{I};X)\cap C^1(I;X)$. 

        Consider a neural network $w_\theta$ with parameter $\theta$, approximating solution $w$ of (\ref{eqmixed}), and the PINN residuals
        \begin{equation}
            \label{residmixed}
            \begin{aligned}
                &\mathcal{R}_{eq}=\frac{d}{dt}\sum_{k=1}^{K}U_kw_\theta-A(t)(w_\theta),\\
                &\mathcal{R}_{in}=w_\theta|_{t=0}-w_0,\\
                &\mathcal{R}_{bn}=B(t)(w_\theta)\\
            \end{aligned}
        \end{equation}
        The total error is given by formula
        \begin{equation}
            \label{totalerrmixed}
            \mathcal{E}:=\int_{0}^{T}\|w_\theta-w\|^{p}dt
        \end{equation}
        \begin{theorem}\label{thgenpar}
            Let us given a solution $w$ of problem (\ref{eqmixed}), and a neural network $w_\theta$ with residuals (\ref{residmixed}) and total error (\ref{totalerrmixed}). Let also $-A(t)$ be a $(p,\psi(t))$-submonotone w.r.t. norm $\|\cdot\|_{\sum U_k^{\frac{1}{p}}}$, with $(p,\epsilon_k)$-powered operators $U_{k}$, where $\epsilon_k$ are subordinate to $B(t)$, $\psi(t)$ is subordinate to $B(t)$, for every $t\in\overline{I}$. Also, let respective $\tilde{\gamma}_k\left(\frac{d\hat{w}}{dt}(\cdot),w_\theta(\cdot),w(\cdot)\right)$,$\gamma(\cdot,w_\theta(\cdot),w(\cdot))\in C(\overline{I})$, $\rho(t)\equiv \rho$, $\tilde{\rho}_k(t)\equiv \tilde{\rho}_k$, and $\Lambda(\cdot,w_\theta(\cdot),w(\cdot))\in L^1(I)$. Also, let $\|\cdot\|^p\leq R\|\cdot\|^p_{\sum U_k^{\frac{1}{p}}}$. Then we have an estimate
            \begin{equation*}
                \mathcal{E}\leq \mathcal{C}\left(\frac{e^{R(p-1)T}-1}{R(p-1)}\right)e^{p\|\Lambda(\cdot,w_\theta(\cdot),w(\cdot))\|_{L^1(I)}},
            \end{equation*}
            where
            \begin{equation*}
                \begin{aligned}
                    \mathcal{C}&= \|\mathcal{R}_{eq}\|_{L^p(I; X)}^p+\|\mathcal{R}_{in}\|_{\Sigma U_k^{\frac{1}{p}}}^p+\\
                    &+p\|\gamma(\cdot,w_\theta(\cdot),w(\cdot))\|_{C(\overline{I})}\|\rho\left(\mathcal{R}_{bn}\right)\|_{L^1(I)}+\\
                    &+p\left\|\tilde{\gamma}_k\left(\frac{d\hat{w}}{dt}(\cdot),w_\theta(\cdot),w(\cdot)\right)\right\|_{L^1(I)}\|\tilde{\rho}_k\left(\mathcal{R}_{bn}\right)\|_{L^1(I)}
                \end{aligned}
            \end{equation*}
        \end{theorem} 
        The proof is similar to the one of Theorem \ref{th2}. 
    \subsection{Hyperbolic-type equation}\label{hypss}
        We apply the classical technique to reduce hyperbolic-type equations to the parabolic-type (see, for instance, \cite{brezis}). We consider the Hilbert space case only. Also, for simplicity, we work with only one $(2,\epsilon)$-powered operator, assume submonotonicity/subordination on the whole domain and put $q=p$. Consider problem (\ref{eqhyp}) in a Hilbert space $X$.
        \begin{equation}
            \label{eqhyp}
            \begin{aligned}
                &\frac{d^2w}{dt^2}=-Uw+F(t)(w)+A(t)\left(\frac{dw}{dt}\right)\\
                &w\Big|_{t=0}=w_0,~\frac{dw}{dt}\Big|_{t=0}=w_{t,0}\\
                &B_1(t)(w)=0,~B_2(t)\left(\frac{dw}{dt}\right)=0\\
            \end{aligned}
        \end{equation}
        where $A(t):\mathcal{D}(A(t))\to X$ and $B_{1,2}(t):\mathcal{D}(B_{1,2}(t))\to Y_{1,2}$ are nonlinear operators, $\mathcal{D}(F(t))\subset \mathcal{D}(B_{1}(t))\subset X$ and $\mathcal{D}(A(t))\subset \mathcal{D}(B_{2}(t))\subset X$  are subspaces, $Y_{1,2}$ are Banach spaces. Also, $U$ is $(2,\epsilon)$-powered, $\mathcal{D}(F(t))\subset\mathcal{D}(U)\subset\mathcal{D}\left(U^{\frac{1}{2}}\right)$, $\mathcal{D}(A(t))\subset\mathcal{D}\left(U^{\frac{1}{2}}\right)$, $U$ and $U^{\frac{1}{2}}$ commute with $\frac{d}{dt}$, $U^{\frac{1}{2}}$ commute with $\Big|_{t=0}$. For a solution of (\ref{eqhyp}), we assume $w\in C^1(\overline{I};X)\cap C^2(I;X)$.

        Consider a neural network $w_\theta$ with parameter $\theta$, approximating solution $w$ of (\ref{eqhyp}), and the PINN residuals
        \begin{equation}
            \label{reshyp}
            \begin{aligned}
                &\mathcal{R}_{eq}=\frac{d^2w_\theta}{dt^2}+Uw_\theta-F(t)(w_\theta)-A(t)\left(\frac{dw_\theta}{dt}\right)\\
                &\mathcal{R}_{in}=w_\theta\Big|_{t=0}-w_0\\
                &\mathcal{R}_{in,t}=\frac{dw_{\theta}}{dt}\Big|_{t=0}-w_{t,0}\\
                &\mathcal{R}_{bn}=B_1(t)(w_\theta)\\
                &\mathcal{R}_{bn,t}=B_2(t)\left(\frac{dw_\theta}{dt}\right)
            \end{aligned}
        \end{equation}
        The total error is given by formula
        \begin{equation}
            \label{totalerrhyp}
            \mathcal{E}:=\int_{0}^{T}\|w_\theta-w\|^{2}dt
        \end{equation}
        \begin{theorem}\label{thhyp}
            Let us given a solution $w$ of problem (\ref{eqhyp}), and a neural network $w_\theta$ with residuals (\ref{reshyp}) and total error (\ref{totalerrhyp}). Let also $-A(t)$ be $(2,\psi(t))$-submonotone, $F(t)$ be conditionally Lipschitz w.r.t. to ${\rm Id}+U^{\frac{1}{2}}$-norm (see Remark \ref{remhyp}), with $(2,\epsilon)$-powered operator $U$, where $\epsilon$ is subordinate to $B_1(t)$ and $B_2(t)$, $\psi(t)$ is subordinate to $B_2(t)$, for every $t\in\overline{I}$. Also, let respective $\tilde{\gamma}\left(w_{\theta}(\cdot),w(\cdot),\frac{dw_\theta}{dt}(\cdot),\frac{dw}{dt}(\cdot)\right)$,$\gamma(\cdot,w_\theta(\cdot),w(\cdot))\in C(\overline{I})$, $\rho(t)\equiv \rho$, $\tilde{\rho}_{1,2}(t)\equiv \tilde{\rho}_{1,2}$, $\Lambda_A\left(\cdot,\frac{dw_\theta}{dt}(\cdot),\frac{dw}{dt}(\cdot)\right)\in L^1(I)$ and $\Lambda_F\left(\cdot,w_\theta(\cdot),w(\cdot)\right)\in L^2(I)$. Then we have an estimate
            \begin{equation*}
                \mathcal{E}\leq \mathcal{C}\frac{(e^{2T}-1)}{2}e^{\left(2\left\|\Lambda_A\left(\cdot,\frac{dw_\theta}{dt}(\cdot),\frac{dw}{dt}(\cdot)\right)\right\|_{L^1(I)}+\left\|\Lambda_F\left(\cdot,w_\theta(\cdot),w(\cdot)\right)\right\|_{L^1(I)}\right)},
            \end{equation*}
            where
            \begin{equation*}
                \begin{aligned}
                    &\mathcal{C}=\|\mathcal{R}_{eq}\|_{L^2(I; X)}^2+\|\mathcal{R}_{in}\|^2+\|\mathcal{R}_{in,t}\|^2+\|U^{\frac{1}{2}}\mathcal{R}_{in}\|^2+2\|\gamma(\cdot,w_{2,\theta}(\cdot),w_2(\cdot))\|_{C(\overline{I})}\|\rho\left(\mathcal{R}_{bn,t}\right)\|_{L^1(I)}+\\
                    &+2\|\tilde{\gamma}(w_{1,\theta}(\cdot),w_1(\cdot),w_{2,\theta}(\cdot),w_2(\cdot))\|_{C(\overline{I})}\left(\|\tilde{\rho}_1\left(\mathcal{R}_{bn}\right)\|_{L^1(I)}+\|\tilde{\rho}_2\left(\mathcal{R}_{bn,t}\right)\|_{L^1(I)}\right)
                \end{aligned}
            \end{equation*}
        \end{theorem} 
        \begin{proof}
            Again we use the denotation $\hat{w}:=w_\theta-w$. 

            Let us consider $w_{\theta,1}=w_\theta$, $w_1=w$, $\hat{w}_1=\hat{w}$, $w_{\theta,2}=\frac{dw_\theta}{dt}$, $w_2=\frac{dw}{dt}$, $\hat{w}_2=\frac{d\hat{w}}{dt}$. Then we have a parabolic-type equation, 
            \begin{equation*}
                \begin{cases}
                    &\frac{d\hat{w}_1}{dt}=\hat{w}_2\\
                    &\frac{d\hat{w}_2}{dt}=\mathcal{R}_{eq}-U\hat{w}_1+F(t)(w_{\theta,1})-F(t)(w_1)+A(t)(w_{\theta,2})-A(t)(w_2)
                \end{cases}
            \end{equation*}
            Hence, one can deal with the space $\tilde{X}=\mathcal{D}(U^{\frac{1}{2}})\times X$ endowed the following norm
            \begin{equation*}
                \left\|\begin{pmatrix}
                    y_1\\
                    y_2\\
                \end{pmatrix}\right\|^2_{\tilde{X}}:=\|y_1\|^2+\|y_2\|^2+\|U^{\frac{1}{2}}y_1\|^2,
            \end{equation*}
            The rest of the proof is similar to the one of Theorem \ref{th2}.
        \end{proof}
        \begin{remark}
            \label{remhyp}
            Conditionally Lipschitz continuity of $F$ w.r.t. ${\rm Id}+U^{\frac{1}{2}}$-norm stands for 
            \begin{equation*}
                \|F(\chi_1)-F(\chi_2)\|^2\leq \Lambda_{F}(\chi_1,\chi_2)\left(\|\chi_1-\chi_2\|^2+\|U^{\frac{1}{2}}\chi_1-U^{\frac{1}{2}}\chi_2\|^2\right)
            \end{equation*}
        \end{remark}
    \subsection{Elliptic equation}\label{ellss}
        Consider problem (\ref{eqell}) in a $\mathbb{R}$-smooth Banach space $X$
        \begin{equation}
            \label{eqell}
            \begin{aligned}
                &A(y)=0\\
                &B(y)=0,
            \end{aligned}
        \end{equation}
        where $A:\mathcal{D}(A)\to X$ and $B:\mathcal{D}(B)\to Y$ are nonlinear operators, $\mathcal{D}(A)\subset \mathcal{D}(B)\subset X$ are subspaces, $Y$ is a Banach space. 

        Let us consider a neural network $y_\theta$ with parameter $\theta$, approximating solution $y$ of (\ref{eqell}) and the PINN residuals
        \begin{equation}
            \label{resell}
            \begin{aligned}
                &\mathcal{R}_{eq}=A(y_\theta),\\
                &\mathcal{R}_{bn}=B(y_\theta).
            \end{aligned}
        \end{equation}
        Total error is
        \begin{equation}
            \label{totalerrell}
            \mathcal{E}:=\|y_\theta-y\|^q,~q>1.
        \end{equation}
        \begin{theorem}
            \label{theq}
            Let us given a solution $y$ of problem (\ref{eqell}), and a neural network $y_\theta$ with residuals (\ref{resell}) and total error (\ref{totalerrell}). Let also $A$ be $(p,\psi)$-coercive, with $\psi$ subordinate to $B$. Then we have an estimate
            \begin{equation*}
                \mathcal{E}\leq p^{\frac{q}{p}}\left(\gamma(y_\theta,y)\rho(\mathcal{R}_{bn})+\frac{1}{p}\Lambda^p(y_\theta,y)\|\mathcal{R}_{eq}\|^p\right)^{\frac{q}{p}}
            \end{equation*}
        \end{theorem} 
        The proof follows standard procedures. It relies on conditions of the theorem and the Young inequality.

\section{PINN's residuals upper bound}\label{residss}
    In \cite{deryckjagtapmishra}, the authors proved the existence of a two-layer tahn neural network approximation for a sufficiently regular solution of the Navier-Stokes equation. Also, they obtained $L^2$ upper bounds for the residuals depending on the neural network width. The core theorem underlying this result claims the existence of a tahn neural network, approximating a function from $H^\sigma(\Omega)$, where $\Omega\subset\mathbb{R}^m$ is an integer right parallelepiped and $\sigma\geq 3$. We obtain a similar theorem in $W^{\sigma,p}(\Omega)$. 

    Let $\Omega\subset\mathbb{R}^m$ be a convex bounded domain, and let $p\geq 1$. First, we need to obtain the Bramble-Hilbert lemma in $W^{\sigma,p}(\Omega)$, similar to the result of \cite{verfurth}. R. Verf\"urth in \cite{verfurth} noted that the core point for such extension is a Poincar\'e inequality with a ''good'' constant in $W^{1,p}(\Omega)$ for zero mean value functions. However, for such functions (i.e. $\frac{1}{|\Omega|}\int_{\Omega}ydx=0$), a constant is known only for one-dimensional cases \cite{gerasimovnazarov}. We provide slightly another idea: the mean value in $L^2(\Omega)$ is a projection to the set of constant functions, and for a function with vanishing projection, there is a ''good'' Poincar\'e inequality. The projection operator is defined as follows.
    \begin{statement}
        Let $p\geq 1$, and $\Omega\subset\mathbb{R}^m$ be a convex bounded domain. Then there exists an operator $\mathcal{J}_p: L^p(\Omega)\to \mathbb{R}$,
        \begin{equation*}
            \mathcal{J}_p(y):=\begin{cases}
                \frac{1}{|\Omega|}\int_{\Omega}ydx,~&1\leq p<2,\\
                {\rm arg}\min_{s\in\mathbb{R}}\|y-s\|_{L^p(\Omega)},~&p\geq 2
            \end{cases}
        \end{equation*}
        with the following properties:
        \begin{equation*}
            \begin{aligned}
                &\mathcal{J}_p(-y)=-\mathcal{J}_p(y),\\
                &\mathcal{J}_p(y-\mathcal{J}_p(y))=0.
            \end{aligned}
        \end{equation*}
    \end{statement}
    \begin{proof}
        If $1\leq p<2$, the result is clear. Let $p\geq 2$. Given fixed element $y\in L^p(\Omega)$, let us consider the following function 
        \begin{equation*}
            h:\mathbb{R}\to \mathbb{R},~h(s):=\|y-s\|_{L^p(\Omega)}
        \end{equation*}
        Then $h$ is a differentiable convex function. Since
        \begin{equation*}
            h(s)\geq |s||\Omega|^{\frac{1}{p}}-\|y\|
        \end{equation*}
        we have that $h(+\infty)=h(-\infty)=+\infty$, and, therefore, $h$ attends a global minimum. We have two possible cases. 

        If $y\equiv s_0$ for some $s_0\in\mathbb{R}$, then we have an exactly one minimum $h(s_0)=0$. If $y\not\equiv const$, then the equality condition in the Minkowski inequality implies that $h$ is strictly convex. Hence, $h$ attends its global minimum exactly at one point. 

        Thus, an operator
        \begin{equation*}
            \mathcal{J}_p(y):={\rm arg}\min_{s\in\mathbb{R}}\|y-s\|_{L^p(\Omega)}
        \end{equation*}
        is correctly defined.

        Also, we have
        \begin{equation*}
            \|-y-(-\mathcal{J}_p(y))\|_{L^p(\Omega)}=\|y-\mathcal{J}_p(y)\|_{L^p(\Omega)}\leq \|-y-s\|_{L^p(\Omega)}
        \end{equation*}
        and
        \begin{equation*}
            \|y-\mathcal{J}_p(y)\|_{L^p(\Omega)}\leq \|y-\mathcal{J}_p(y)-s\|_{L^p(\Omega)}
        \end{equation*}
        for every $s\in\mathbb{R}$. Hence, $\mathcal{J}_p$ has the desired properties. 
    \end{proof}

    We turn to the Poincar\'e inequality.
    \begin{lemma}(Poincar\'e inequality)
        For every $y\in W^{1,p}(\Omega)$ with $\mathcal{J}_p(y)=0$, the following inequality holds
        \begin{equation*}
            \|y\|_{L^p(\Omega)}\leq \pi_p{\rm diam}(\Omega)\|\nabla y\|_{L^p(\Omega)},
        \end{equation*} 
        where
        \begin{equation}    
            \label{poincare}
            \pi_p=\begin{cases}
                \frac{1}{2},~&p=1\\
                \pi^{\frac{2}{p}-2}2^{1-\frac{2}{p}},~&1<p<2\\
                \frac{p\sin\left(\frac{\pi}{p}\right)}{2\pi(p-1)^{\frac{1}{p}}},~&p\geq 2
            \end{cases}
        \end{equation}
    \end{lemma} 
    \begin{proof}
        For $p=2$, the inequality is a result of \cite{payneweinberger}, which was correctly proved in \cite{bebendorf}. For $p=1$, it was proved in \cite{acostaduran}. In \cite{bebendorf}, there is a remark that the proof in \cite{acostaduran} also contains a similar mistake, but the authors of \cite{acostaduran} corrected their text after the paper was published. 

        If $p\geq 2$, then $\mathcal{J}_p(y)=0$ implies $\int_{\Omega}y|y|^{p-2}dx=0$. Really, since $s=\mathcal{J}_p(y)=0$ is a stationary point of $h$, we have 
        \begin{equation*}
            0=\frac{dh}{ds}\Big|_{s=0}=-p\int_{\Omega}y|y|^{p-2}dx
        \end{equation*}
        For such functions, the inequality was proved in \cite{espositonitschtrombetti}.
        
        Finally, in case $1<p<2$, the constant $\pi_p$ can be obtained by the Riecz-Thorin interpolation theorem between $L^1$ and $L^2$. 
    \end{proof}

    As a result, we have
    \begin{statement}(Bramble-Hilbert lemma in $W^{\sigma+1,p}(\Omega)$)
        Let $\Omega\subset\mathbb{R}^m$ be a convex bounded domain, $p\geq1$, $y\in W^{\sigma+1,p}(\Omega)$. Then there exists a polynomial $P_y$ in $m$ variables of degree at most $\sigma$ such that
        \begin{equation*}
            |y-P_y|_{W^{\nu,p}(\Omega)}\leq c_{\sigma+1,\nu,p}\left({\rm diam}(\Omega)\right)^{\sigma+1-\nu}|y|_{W^{\sigma+1,p}(\Omega)},~\forall 0\leq \nu\leq \sigma,
        \end{equation*}
        where
        \begin{equation}
            \label{bramblehilbert}
            c_{\sigma,\nu,p}\leq \pi^{\sigma-\nu}_{p}\binom{m+\nu-1}{\nu}^{\frac{1}{p}}\frac{\left((\sigma-\nu)!\right)^{\frac{1}{p}}}{\left(\left\lceil\frac{\sigma-\nu}{m}\right\rceil!\right)^{\frac{m}{p}}}
        \end{equation}
        and $\pi_p$ is defined in (\ref{poincare}).
    \end{statement}
    \begin{proof}
        As in \cite{verfurth}, we define polynomials $P_{y,\sigma}$,\dots,$P_{y,0}=:P_y$ recursively
        \begin{equation*}
            P_{y,\sigma}:=\sum_{\iota\in\mathbb{N}^m,|\iota|=\sigma}\frac{1}{\iota!}x^{\iota}\mathcal{J}_{p}(\partial^{\iota} y)
        \end{equation*}
        and
        \begin{equation*}
            P_{y,\nu-1}:=P_{y,\nu}+\sum_{\iota\in\mathbb{N}^m,|\iota|=\nu-1}\frac{1}{\iota!}x^{\iota}\mathcal{J}_{p}\left(\partial^{\iota} (y-\mathcal{J}_{p}(y))\right),~\sigma\geq \nu\geq 1.
        \end{equation*}
        As in \cite{verfurth}, with properties of $\mathcal{J}_{p}$, one can show that 
        \begin{equation*}
            \partial^\omega(P_{y,\sigma-\tilde{\sigma}})=P_{y,\sigma-\tilde{\sigma}-\nu}(\partial^\omega y)
        \end{equation*}
        and 
        \begin{equation*}
            \mathcal{J}_p\left(\partial^\omega(y-P_{y,\sigma-\tilde{\sigma}}(y))\right)=0,
        \end{equation*}
        for all $y\in W^{\tilde{\sigma},p}(\Omega)$, $0\leq \nu\leq \tilde{\sigma}$, and $\omega\in\mathbb{N}^m$ with $|\omega|:=\omega_1+\dots+\omega_m=\nu$. 

        The rest of the proof is similar to the one in \cite{verfurth}.
    \end{proof}

    As a corollary, arguing exactly as in \cite{deryckjagtapmishra}, one can obtain a theorem on the neural network approximation existence in $W^{\sigma,p}(\Omega)$.
    \begin{theorem}
        \label{thresid}
        Let $p\geq1$, $m\geq 2$, $\sigma\geq 3$, $\delta>0$, $a_j,b_j\in\mathbb{Z}$, $a_j<b_j$, for $1\leq j\leq m$, $\Omega=\prod_{j=1}^{m}[a_j,b_j]$ and $y\in W^{\sigma,p}(\Omega)$. Then for every $N\in\mathbb{N}$ with $N>5$ there exists a tanh neural network $y_{\theta,N}$ with two hidden layers, one of width at most $3\left\lceil\frac{\sigma}{2}\right\rceil\binom{\sigma+m-1}{m+1}+\sum_{j=1}^{m}(b_j-a_j)(N-1)$ and another of width at most $3\left\lceil\frac{m+2}{2}\right\rceil\binom{2m+1}{m+1}N^{m}\prod_{i=1}^{m}(b_j-a_j)$, such that for $\nu\in\{0,1,2\}$ it holds that,
        \begin{equation*}
            \|y-y_{\theta,N}\|_{W^{\nu,p}(\Omega)}\leq 2^\nu3^m\mathcal{A}_{\nu,\sigma,m,y}(1+\delta)\ln^\nu\left(\mathcal{B}_{\nu,\delta,m,y}N^{m+\sigma+2}\right)N^{-\sigma+\nu},
        \end{equation*}
        where
        \begin{equation*}
            \mathcal{B}_{\nu,\delta,m,y}=\frac{5\cdot 2^{\nu m}\max\left\{\prod_{j=1}^{m}(b_j-a_j),m\right\}\max\{\|y\|_{W^{\nu,\infty}(\Omega)},1\}}{3^m\delta\min\{1,\mathcal{A}_{\nu,\sigma,m,y}\}},
        \end{equation*}
        and
        \begin{equation*}
            \mathcal{A}_{\nu,\sigma,m,y}=\max_{0\leq l\leq \nu}c_{\sigma,l,p}\left(3\sqrt{m}\right)^{\sigma-l}|y|_{W^{\sigma,p}(\Omega)},
        \end{equation*}
        where $c_{\sigma,l,p}$ is defined in (\ref{bramblehilbert}). 

        Furthermore, the weights of $y_{\theta,N}$ scale as $O\left(N^{\max\left\{\frac{\sigma^2}{2},m\left(1+\frac{\sigma}{2}+\frac{m}{2}\right)\right\}}\right)$.
    \end{theorem}
    \begin{remark}
        One can also obtain the bound for $\nu>2$; however, it will be different, as $\|{\rm tanh}\|_{C^\nu}>1$ if $\nu>2$. See \cite{derycklanthalermishra}. 
    \end{remark}

\section{Numerical experiments}\label{numerical}
    We demonstrate how theorems, statements, and corollaries work. Corollary \ref{corpar} (and similar that one can obtain for other types) provides us with two estimates: exact (\ref{parexact}) and asymptotic (\ref{parasymp}). The exact estimate is an explicit expression where each component ($\beta^{*}(w_\theta,w)$, $\gamma(w_\theta,w)$, $\Lambda(w_\theta,w)$ etc.) is known. The asymptotic estimate does not include these components. It simply demonstrates asymptotic dependence between training and total error. Of course, the exact estimate is more desirable; however, it usually overestimates total error and requires additional assumptions on the existence of sufficiently smooth solutions, not to mention that it involves explicit estimates of the solution norms. Thus, the exact estimate serves more for qualitative analysis (e.g., to show the error's exponential dependence on $T$), while the asymptotic estimate, being easy to calculate, allows tracking total error convergence. It also plays an essential role in qualitative analysis since it demonstrates how different training errors influence total error. 

    In this paper, we don't pay much attention to the global well-posedness of the equations. We assume instead that a solution exists and is sufficiently regular in all cases where it's needed. 

    We use the midpoint rule as a method of the approximation of norms (Lemma \ref{midpointrule}).  Also, in some cases, we deal with $L^p$ for different $p$. 

    Also, we obtain estimates with constants that depend on $C^2$-norms of residuals. One can also use instead constants that depend on the $C^{\nu}$-norm of the neural network, as in \cite{deryckjagtapmishra}. However, such an approach significantly overestimates total error and is computationally complex (it requires computing $\binom{\nu+m}{\nu}$ derivatives). It is still advantageous as it allows applying lemma C.1 from \cite{deryckjagtapmishra} to demonstrate that an increase in neural network width is much more preferable than an increase in depth. 

    As for the Theorem \ref{thresid}, we use it only once in subsection \ref{heatss}. The main aim is to demonstrate how the residuals depend on the smoothness order of the equation solution.

    Equations of subsections \ref{heatss} and \ref{kdvss} are both of type (\ref{partype}). First is a linear 1D-heat equation in $L^p$ to show that the method works for time-dependent operators. Second is a (nonlinear) Korteweg–De Vries's equation in $L^p$. 

    Subsection \ref{maxwellss} deals with the 2D-Maxwell equation in $L^2$, which exemplifies generalized parabolic-type equation (\ref{genpartype}). We don't provide in the text the Camassa-Holm equation of higher order (see, for instance, \cite{constantinkolev}); however, it also can be considered as an example of the generalized parabolic-type equation. 

    Subsections \ref{boussinesqss} and \ref{rayleighss} are devoted to hyperbolic type (\ref{hyptype}) equations. The first is the so-called "good" Boussinesq fourth-order equation with two $(2,\epsilon)$-powered operators and operator $F$, but not with $A$. The second equation is the Rayleigh wave equation with one $(2,\epsilon)$-powered operator, and it involves operator $A$ but not $F$. 

    Finally, subsection \ref{ellipticss} covers a 2D-Poisson equation in $L^p$ with a piecewise continuous forcing function. It demonstrates how $(p,\psi)$-coercivity extends to the closure of the operator (properties of coercive operator similar to Statement \ref{closure1}, see Remark \ref{remclos}). Surprisingly, in this case, we can obtain not only an estimate in terms of residuals but also an estimate in terms of training errors. This capability arises because residuals are piecewise smooth functions, and we can consider them almost everywhere. Of course, such a technique applies to any of the types (\ref{partype})--(\ref{elltype}).

    \subsection{1D-Heat equation with non-uniform time-dependent thermal diffusivity}\label{heatss}
        We start with the heat equation with non-uniform time-dependent thermal diffusivity. Let $\Omega=[a;b]$, $X=L^p(\Omega)$, $p\geq2$.  

        \begin{equation}\label{heateq}
            \begin{aligned}
                &\partial_t u = \partial_x\left[\phi_1\partial_x u\right]+\phi_2,~\phi_1\geq 0\\
                &u(0,x)=u_0(x)\\
                &u(t,a)=0,~u(t,b)=0
            \end{aligned}
        \end{equation}

        Since the equation is linear, for every $T>0$, its solution satisfies $u\in C(\overline I;C^{\sigma}(\Omega))$, for a sufficiently regular $u_0$ and some admissible $\phi_1=\phi_1(t,x)$, $\phi_2=\phi_2(t,x)$. 

        Firstly, using Theorem \ref{thresid}, one can formulate the existence of a tanh neural network that approximates the solution and the bound on residuals, which depend on the regularity of the solution. Namely, 
        \begin{statement}
            \label{heatresids}
            Let $u\in W^{\sigma,p}(I\times\Omega)$ be a solution of problem (\ref{heateq}), $\sigma\geq 3$, $\delta>0$, $a,b,T\in\mathbb{Z}$. Then for every $N\in\mathbb{N}$ with $N>5$, there exists a tanh neural network $u_{\theta,N}$ with two hidden layers, one of width at most $3\left\lceil\frac{\sigma}{2}\right\rceil\binom{\sigma+1}{2}+(b-a+T)(N-1)$ and another of width at most $60N^{3}(b-a)T$, such that 
            \begin{equation*}   
                \begin{aligned}
                    &\|\mathcal{R}_{eq}\|_{L^p(I\times\Omega)}\leq \mathcal{A}_1 {\ln^2}(\mathcal{B}_1N)N^{-\sigma+2}\\
                    &\|\mathcal{R}_{in}\|_{L^p(\Omega)}\leq \mathcal{A}_2\ln \left(\mathcal{B}_2N\right)N^{-\sigma+1}\\
                    &\|\mathcal{R}_{bn}\|_{L^p(I;\mathbb{R}^2)}\leq \mathcal{A}_3\ln \left(\mathcal{B}_2N\right)N^{-\sigma+1}
                \end{aligned}
            \end{equation*}
            for some constants $\mathcal{A}_1$--$\mathcal{A}_3$, $\mathcal{B}_1$, $\mathcal{B}_2$ that don't depend on $N$.
        \end{statement}
        \begin{proof}
            For $\mathcal{R}_{eq}$, we have
            \begin{equation*}
                \|\mathcal{R}_{eq}\|_{L^p(I\times\Omega)}\leq \|\partial_t \hat{u}\|_{L^p(I\times\Omega)}+\|\partial_x\phi_1\partial_x \hat{u}\|_{L^p(I\times\Omega)}+\|\phi_1\partial^2_{x}\hat{u}\|_{L^p(I\times\Omega)}\lesssim\|\hat{u}\|_{W^{2,p}(I\times\Omega)}
            \end{equation*}
            Applying Lemma \ref{traceinequality} we have
            \begin{equation*}
                \|\mathcal{R}_{in}\|_{L^p(\Omega)}\leq \left\|\hat{u}\Big|_{t=0}\right\|_{L^p(I\times\Omega)}\lesssim\|\hat{u}\|_{W^{1,p}(I\times\Omega)}
            \end{equation*}
            Similarly, one can obtain the bound on $\mathcal{R}_{bn}$. It remains to apply Theorem \ref{thresid}.
        \end{proof}

        As for the estimates, one can formulate the following result.
        \begin{statement}
            Let $u$ be a solution to problem (\ref{heateq}) and $u_\theta$ be its neural network approximation. Then, one can estimate the total error in terms of residuals:
            \begin{equation*}
                \mathcal{E}\leq \mathcal{C}^{\frac{q}{p}}\left(\frac{p\left(e^{\frac{q(p-1)T}{p}}-1\right)}{q(p-1)}\right),
            \end{equation*}
            where
            \begin{equation*}
                \mathcal{C}=\|\mathcal{R}_{eq}\|^p_{L^p(I\times \Omega)}+\|\mathcal{R}_{in}\|_{L^p(\Omega)}^p+2p(2T)^{\frac{1}{p}}\|\phi_1\|_{C(\overline{I}\times\Gamma)}\max\{\|\partial _xu\|_{C(\overline{I}\times\Gamma)},\|\partial _xu_{\theta}\|_{C(\overline{I}\times\Gamma)}\}\|\mathcal{R}_{bn}\|^{p-1}_{L^p(I;\mathbb{R}^2)}
            \end{equation*}
            Furthermore, if residuals are at least of $C^2$-class, then one can estimate the total error in terms of training error:
            \begin{equation*}
                \mathcal{E}\leq {\tilde{\mathcal{C}}}^{\frac{q}{p}}\left(\frac{p\left(e^{\frac{q(p-1)T}{p}}\right)}{q(p-1)}-1\right),
            \end{equation*}
            where
            \begin{equation*}
                \begin{aligned}
                    \tilde{\mathcal{C}}&= \mathcal{E}_{T,eq}+\frac{p^2(b-a)T}{24}\left((b-a)^2+T^2\right)\|\mathcal{R}_{eq}\|^p_{C^2(I\times\Omega)}M_{eq}^{-1}+\\
                    &+\mathcal{E}_{T,in}+\frac{p^2(b-a)^3}{24}\|\mathcal{R}_{in}\|_{C^2(\Omega)}^pM_{in}^{-2}+\\
                    &+2p(2T)^{\frac{1}{p}}\|\phi_1\|_{C(\overline{I}\times\Gamma)}\max\{\|\partial _xu\|_{C(\overline{I}\times\Gamma)},\|\partial _xu_{\theta}\|_{C(\overline{I}\times\Gamma)}\}\left(\mathcal{E}_{T,bn}+\frac{p^2T^3}{12}\|\mathcal{R}_{bn}\|_{C^2(I;\mathbb{R}^2)}^pM_{bn}^{-2}\right)^{\frac{p-1}{p}}
                \end{aligned}
            \end{equation*}
            That is asymptotically
            \begin{equation*}
                \mathcal{E}=\mathcal{O}\left(\mathcal{E}_{T,eq}+M_{eq}^{-1}+\mathcal{E}_{T,in}+M_{in}^{-2}+\left(\mathcal{E}_{T,bn}+M_{bn}^{-2}\right)^{\frac{p-1}{p}}\right)
            \end{equation*}
        \end{statement}
        \begin{proof}
            For the operator $A(t)=(\phi_1 y')'$, we have
            \begin{equation*}
                \begin{aligned}
                    \langle A(t)(y), y\rangle_p = \int_{a}^{b} (\phi_1 y')'|y|^{p-2}ydx = \phi_1 y'|y|^{p-2}y\Big|_{a}^{b}-(p-1)\int_{a}^{b}\phi_1(y')^2|y|^{p-2}dx\leq \phi_1 y'|y|^{p-2}y\Big|_{a}^{b}
                \end{aligned}
            \end{equation*}
            Hence, $\Lambda\equiv0$ and if we put $Y=\mathbb{R}^2$ with $p$-norm, then
            \begin{equation*}
                |\psi(\chi, y)| = \left|\phi_1 \hat{y}'|\hat{y}|^{p-2}\hat{y}\Big|_{a}^{b}\right|\leq 2^{\frac{p+1}{p}}\|\phi_1\|_{C(\overline{I}\times\Gamma)}\max\{\|y'\|_{C(\Gamma)},\|\chi'\|_{C(\Gamma)}\}|B\hat{y}|^{p-1}_{p},~\hat y=\chi-y
            \end{equation*}
            Hence, with Theorem \ref{th2}, we have the first estimate. For the exact and asymptotic estimates, it remains to apply Corollary \ref{corpar} and Lemmas \ref{midpointrule}, \ref{pc2norm}. 
        \end{proof}

        Now, we turn to the exact problem. Let $a=1$, $b=2$, $T=1$, $\phi_2\equiv0$, $\phi_1=x^2\sin t$ and 
        \begin{equation*}
            u_0(x)=\frac{\sin\left(\frac{\pi}{\ln(2)}\ln x\right)}{\sqrt{x}}
        \end{equation*}
        Then the exact solution is 
        \begin{equation*}
            u(t,x)=\frac{\sin\left(\frac{\pi}{\ln(2)}\ln x\right)}{\sqrt{x}}e^{-\left(\frac{\pi^2}{\ln^2(2)}+\frac{1}{4}\right)(1-\cos t)}
        \end{equation*}
        We take 2 hidden layers of width 80, 20000 epochs, and training samples of size $M_{eq}=2^{18}$ and $M_{in}=M_{bn}=2^{12}$. Also, we make experiments for the cases $q=p\in\{2,3,4,5\}$. 
        \begin{figure}[t]
            \centering
            \includegraphics[scale=0.4]{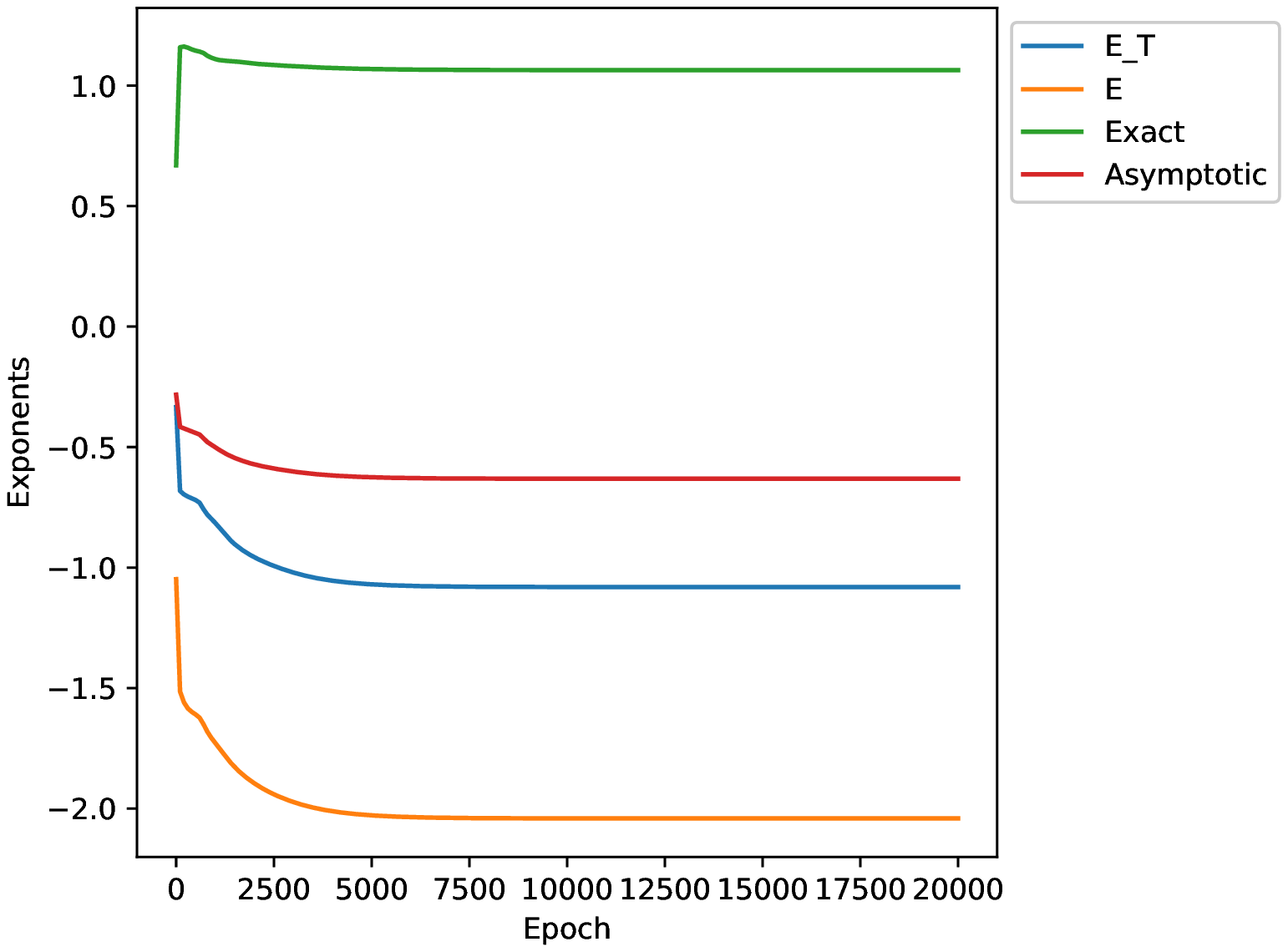}
            \includegraphics[scale=0.4]{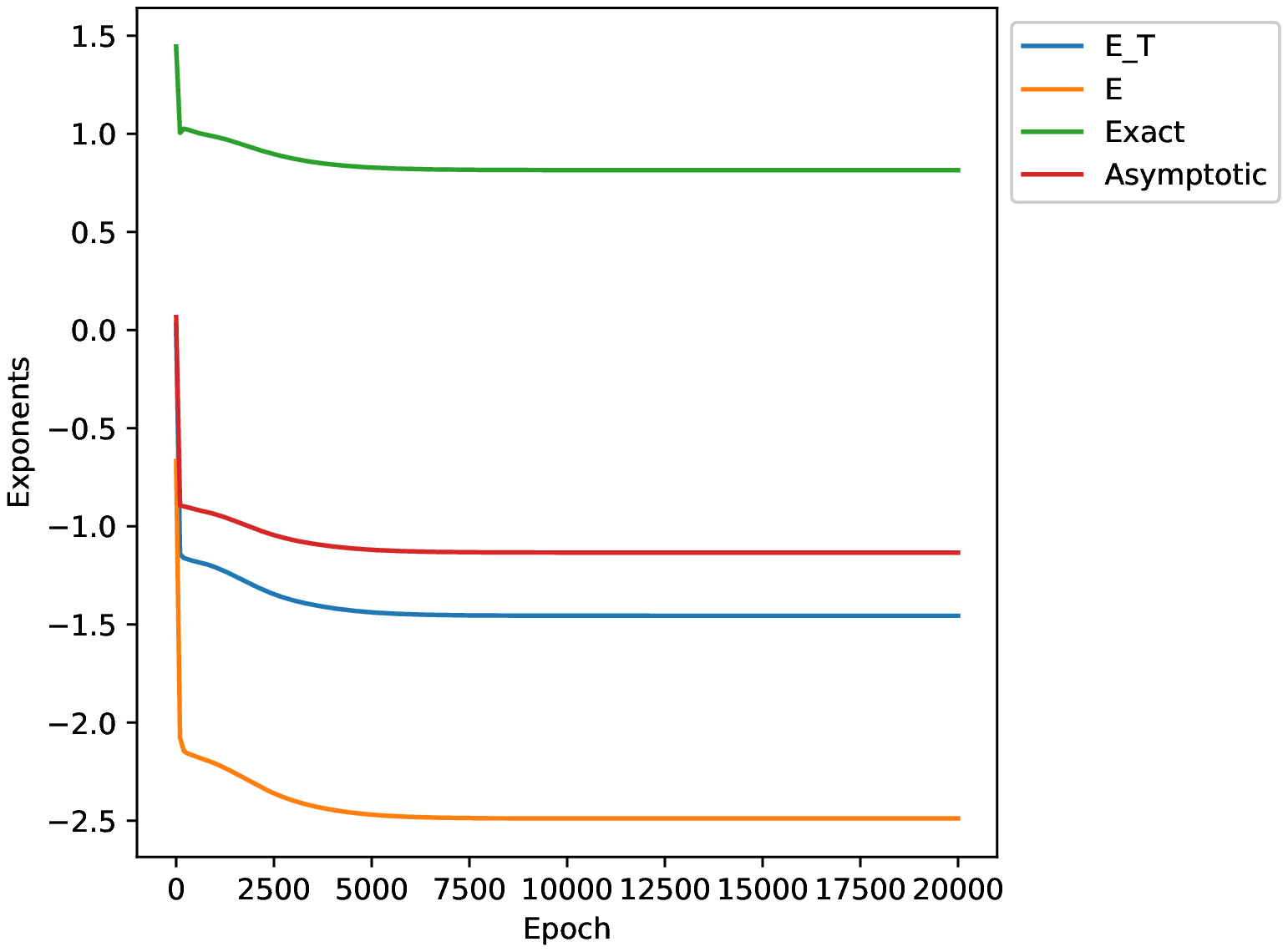}
            \includegraphics[scale=0.4]{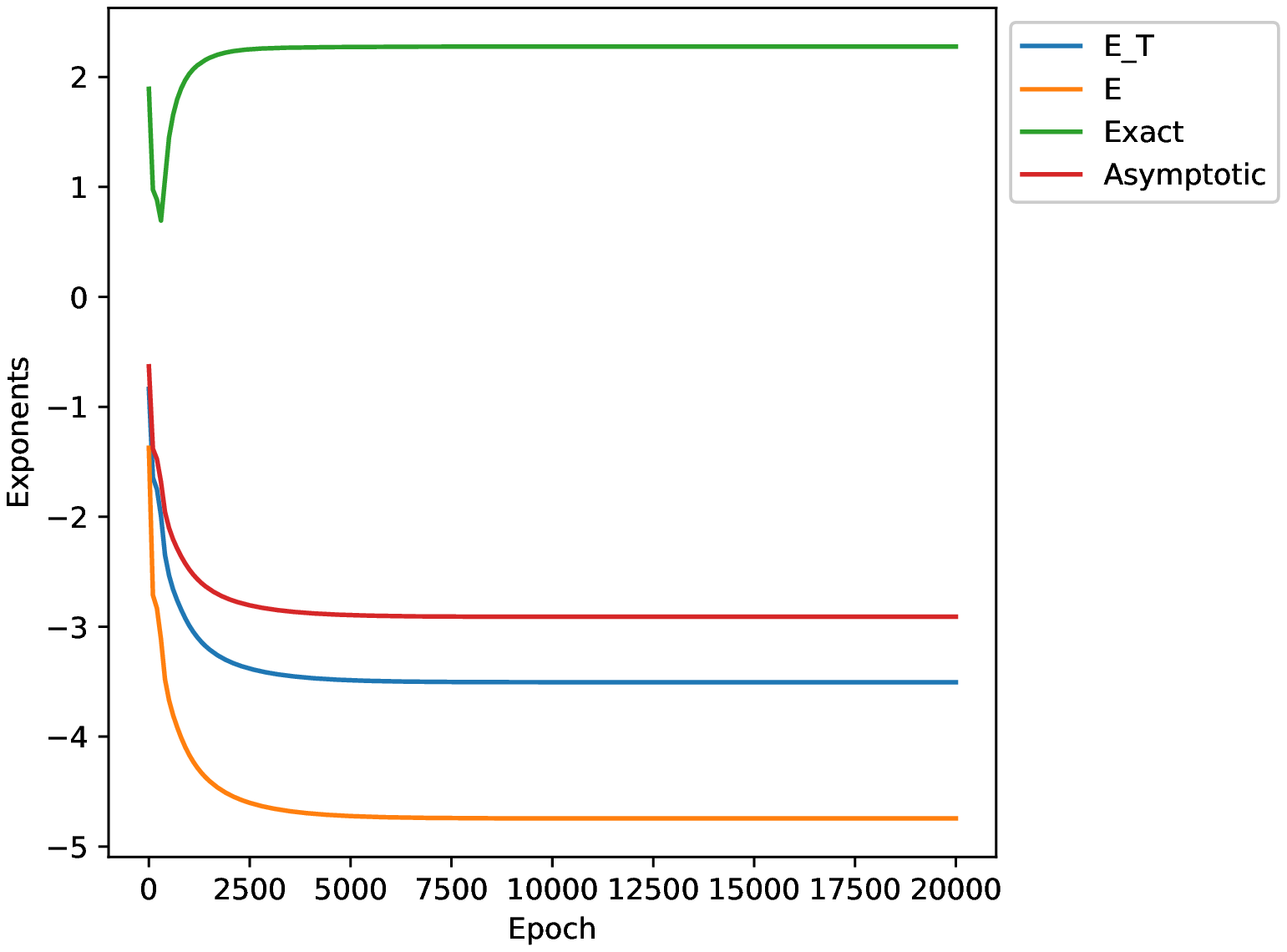}
            \includegraphics[scale=0.4]{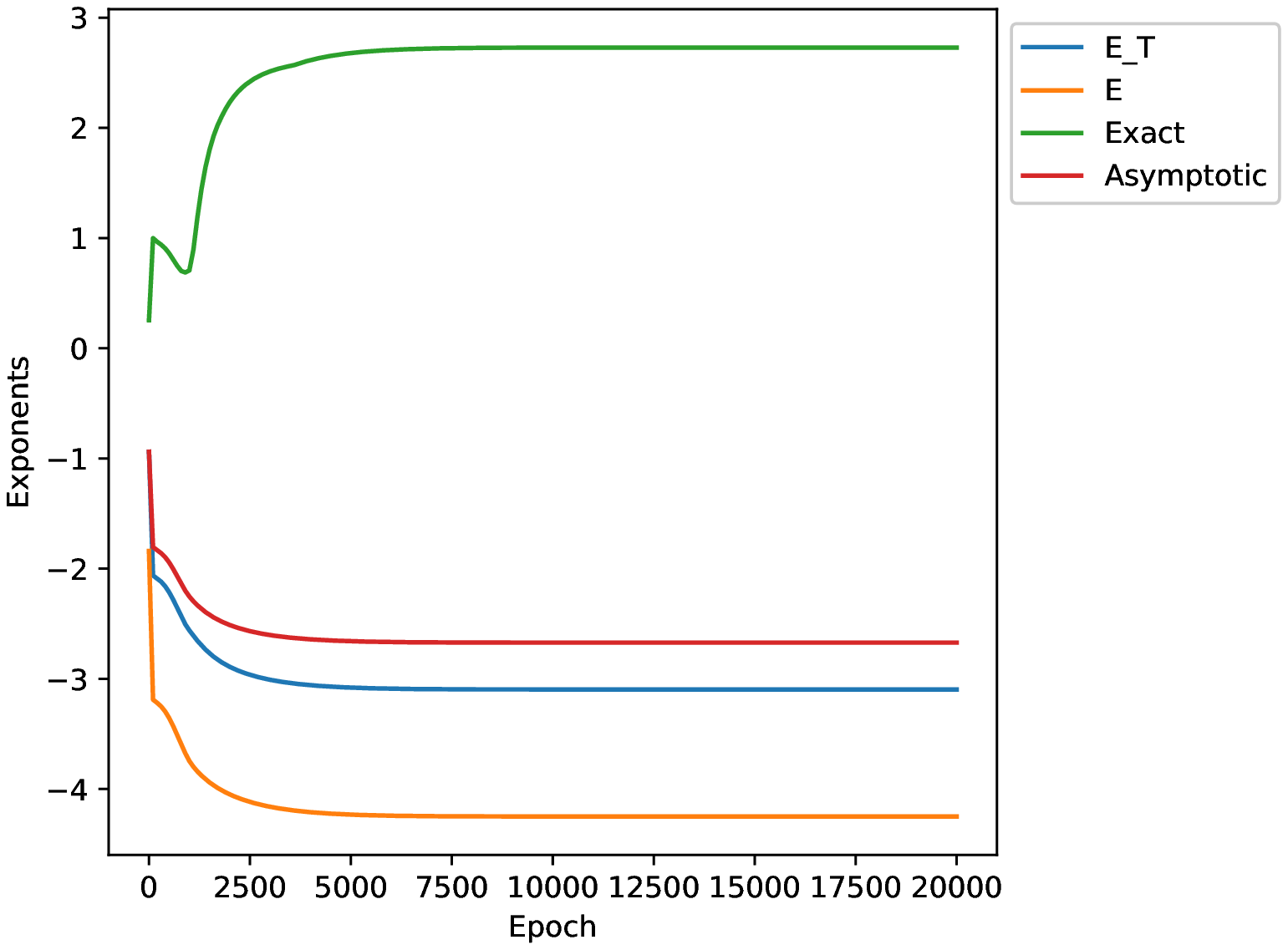}
            \caption{1D-Heat equation: $\lg(\mathcal{E})$, $\lg(\mathcal{E}_T)$, $\lg(\mathcal{E}_{exact})$, $\lg(\mathcal{E}_{asymp})$ for $L^p$, $p\in\{2,3,4,5\}$}\label{heat1}
        \end{figure} 
        Figure \ref{heat1} shows that the exact estimate can significantly overestimate total error. Also, starting from some epoch, the behavior of total error and estimates is similar. To illustrate it further, let us build a correlation plot between $\mathcal{E}_T$ and $\mathcal{E}_{asymp}$ (Figure \ref{heat2}).
        \begin{figure}[t]
            \centering
            \includegraphics[scale=0.4]{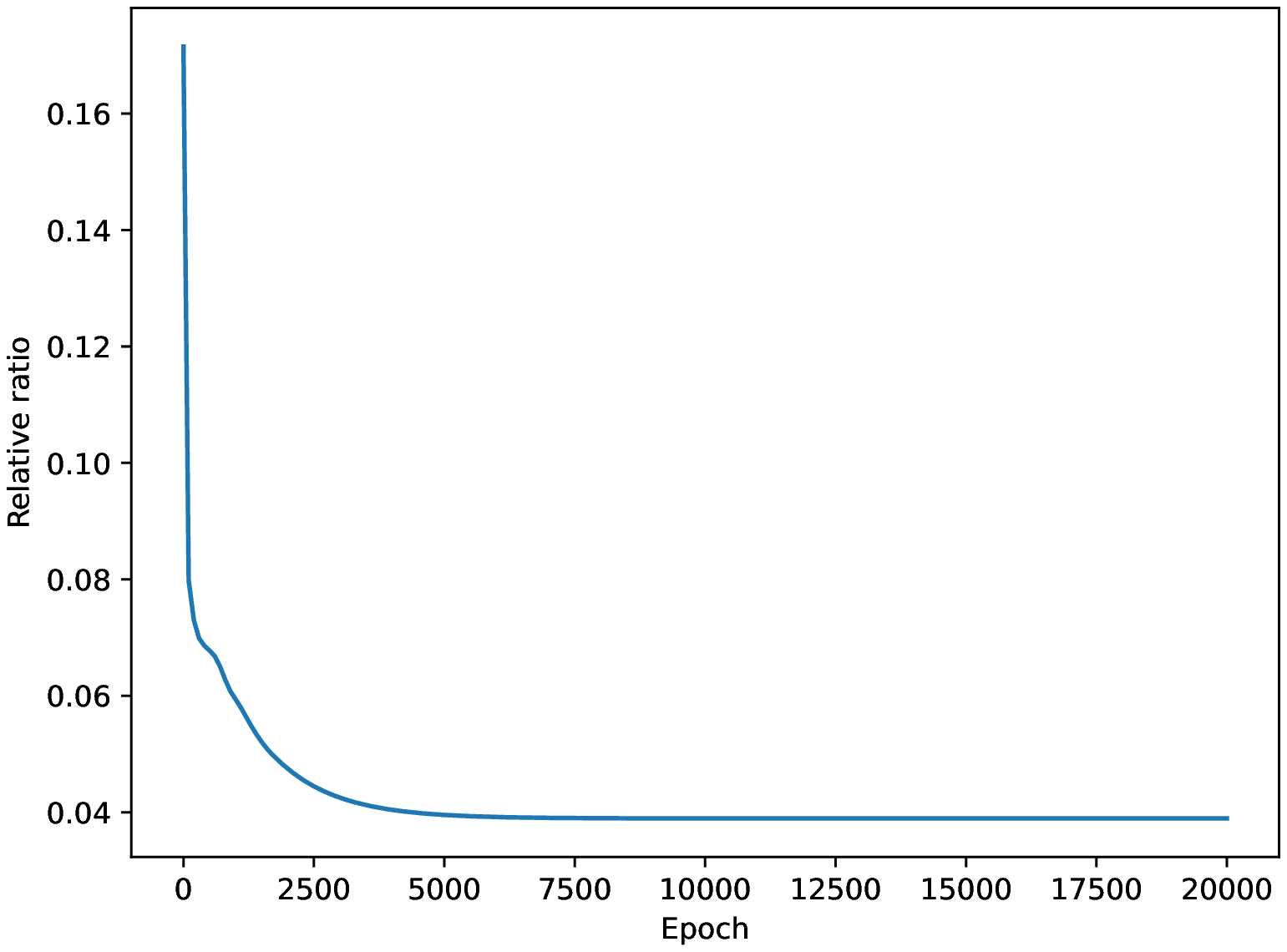}
            \includegraphics[scale=0.4]{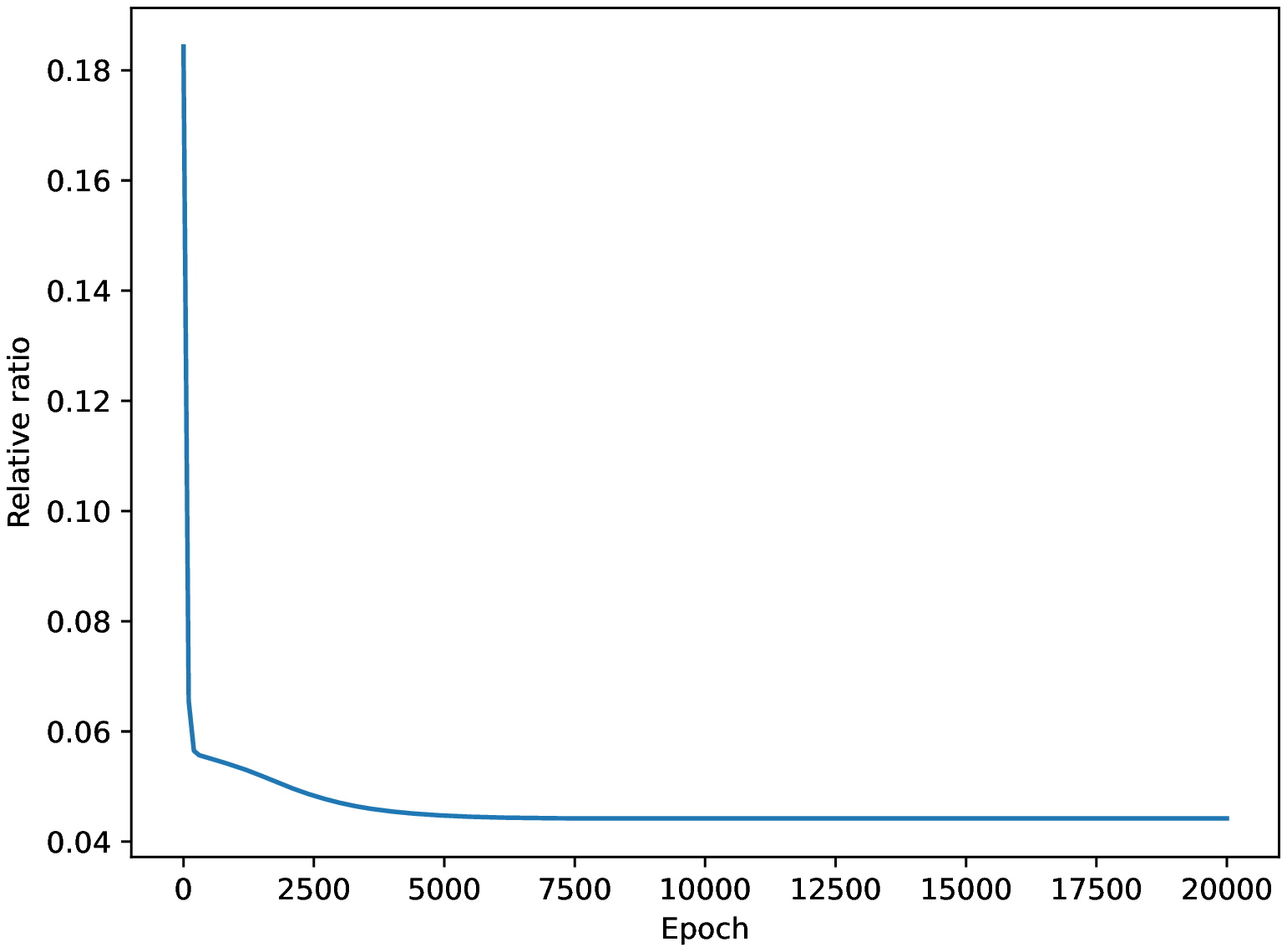}
            \includegraphics[scale=0.4]{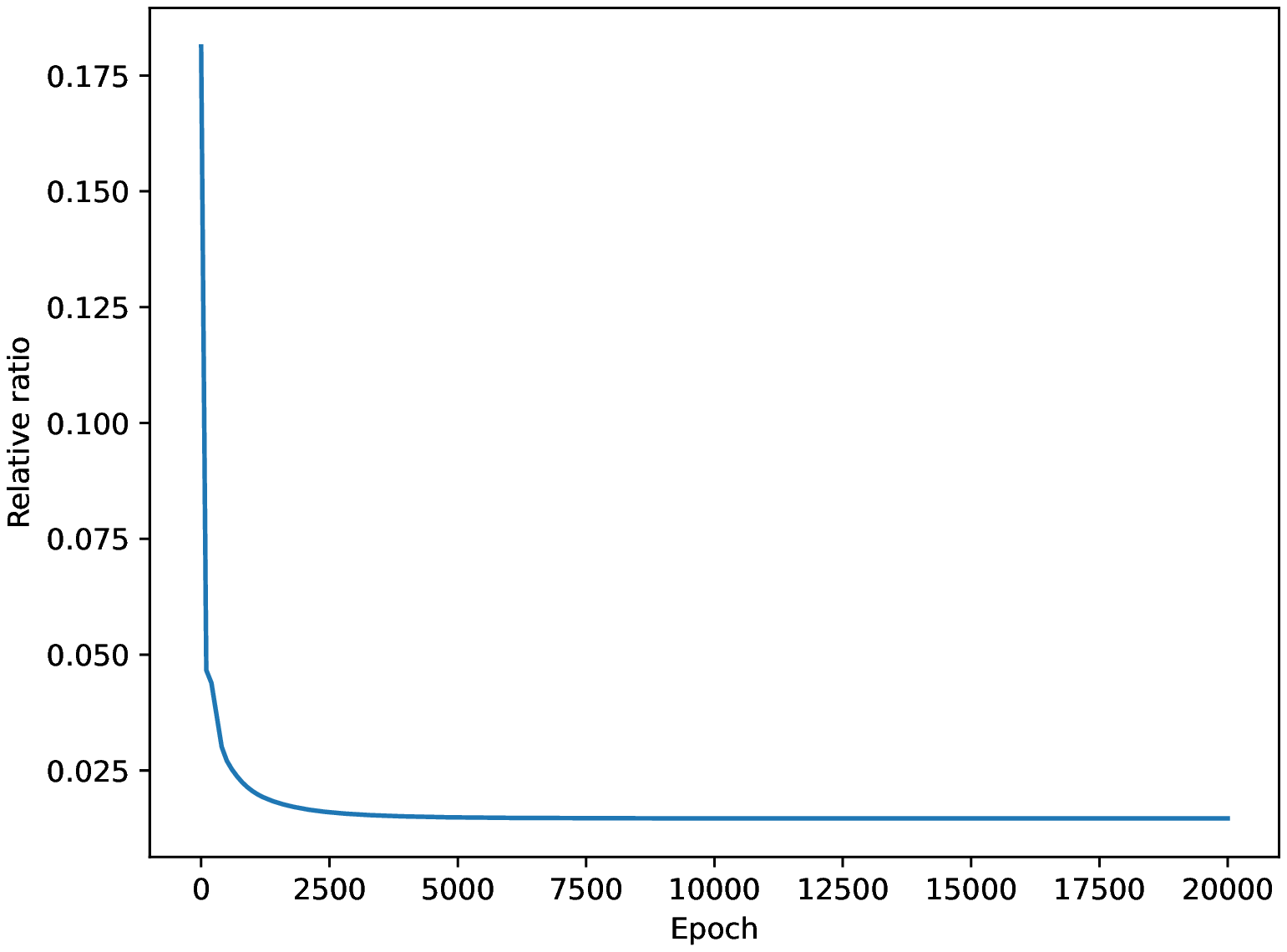}
            \includegraphics[scale=0.4]{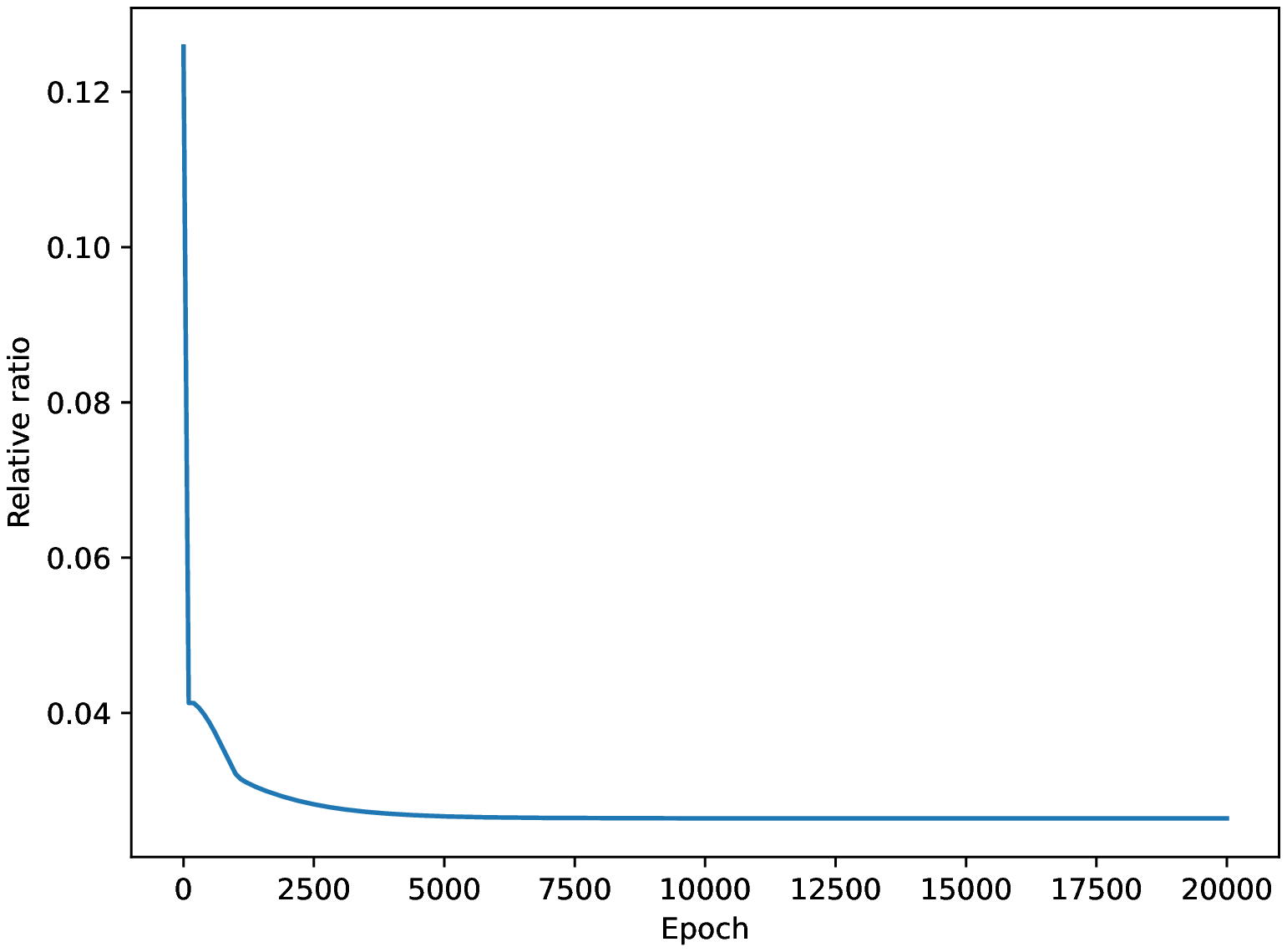}
            \caption{1D-Heat equation: $\frac{\mathcal{E}_T}{\mathcal{E}_{asymp}}$ for $L^p$, $p\in\{2,3,4,5\}$}\label{heat2}
        \end{figure} 

    \subsection{KdV equation}\label{kdvss}
        Another equation we consider is Korteweg–De Vries's equation. Let $\Omega=[a;b]$, $X=L^p(\Omega)$, and $p=2$ or $p>3$.   
        \begin{equation}\label{kdveq}
            \begin{aligned}
                &\partial_t u=6u\partial_xu-\partial^3_{x}u\\
                &u(0,x)=u_0(x)\\
                &u(t,a)=g_1(t),~u(t,b)=g_2(t),~\partial_xu(t,b)=g_3(t)
            \end{aligned}
        \end{equation}

        For a sufficiently regular $u_0$ and admissible $g_1$, $g_2$ and $g_3$, for every $T>0$, its solution satisfies $u\in C^\sigma(I\times[a;b])$. See, for intance, \cite{godlewskiraviart}. 

        Then we have the following result
        \begin{statement}
            Let $u$ be a solution to problem (\ref{kdveq}) and $u_\theta$ be its neural network approximation. Then, one can estimate the total error in terms of residuals:
            \begin{equation*}
                \mathcal{E}\leq \mathcal{C}^{\frac{q}{p}}\left(\frac{p(e^{\frac{q(p-1)T}{p}}-1)}{q(p-1)}\right)e^{qT\max\left\{6\left(1+\frac{1}{p}\right)\|\partial_xu\|_{C(\overline{I}\times\Omega)}-\frac{27(p-1)(p-2)}{4p^3(b-a)^3}, 0\right\}},
            \end{equation*}
            where
            \begin{equation*}
                \begin{aligned}
                    &\mathcal{C}= \|\mathcal{R}_{eq}\|^p_{L^p(I\times\Omega)}+\|\mathcal{R}_{in}\|_{L^p(\Omega)}^p+(2T)^{\frac{1}{p}}\mu p\|\mathcal{R}_{bn}\|^{p-1}_{L^p(I;\mathbb{R}^3)}+\\
                    &+\left(\frac{6\mu (2p+1)}{p+1}+\max\left\{\frac{(p-1)(p-2)}{2},p-1\right\}+\frac{27(p-1)(p-2)}{2p^2(b-a)^2}\right)\|\mathcal{R}_{bn}\|^p_{L^p(I;\mathbb{R}^3)},\\
                    &\mu =\max\{\|u\|_{C(\overline{I}\times\Gamma)},\|u_\theta\|_{C(\overline{I}\times\Gamma)},\|\partial^2_{x}u\|_{C(\overline{I}\times\Gamma)},\|\partial^2_{x}u_{\theta}\|_{C(\overline{I}\times\Gamma)}\}
                \end{aligned}
            \end{equation*}

            Furthermore, if residuals are at least of $C^2$-class, then one can estimate the total error in terms of training error:
            \begin{equation*}
                \mathcal{E}\leq {\tilde{\mathcal{C}}}^{\frac{q}{p}}\left(\frac{p(e^{\frac{q(p-1)T}{p}}-1)}{q(p-1)}\right)e^{qT\max\left\{6\left(1+\frac{1}{p}\right)\|\partial_x u\|_{C(\overline{I}\times\Omega)}-\frac{27(p-1)(p-2)}{4p^3(b-a)^3}, 0\right\}},
            \end{equation*}
            where
            \begin{equation*}
                \begin{aligned}
                    \mathcal{C}&=\mathcal{E}_{T,eq}+p^2\frac{(b-a)T}{24}\left((b-a)^2+T^2\right)\|\mathcal{R}_{eq}\|^p_{C^2(I\times\Omega)}M^{-1}_{eq}+\\
                    &+\mathcal{E}_{T,in}+p^2\frac{(b-a)^3}{24}\|\mathcal{R}_{in}\|^p_{C^2(I\times\Omega)}M^{-2}_{in}+\\
                    &+(2T)^{\frac{1}{p}}\mu p\left(\mathcal{E}_{T,bn}+p^2\frac{3^p(b-a)^3}{24}\|\mathcal{R}_{bn}\|^p_{C^2(I;\mathbb{R}^3)}M^{-2}_{bn}\right)^{1-\frac{1}{p}}+\\
                    &+\left(\frac{6\mu (2p+1)}{p+1}+\max\left\{\frac{(p-1)(p-2)}{2},p-1\right\}+\right.\\
                    &\left.+\frac{27(p-1)(p-2)}{2p^2(b-a)^2}\right)\left(\mathcal{E}_{T,bn}+p^2\frac{3^p(b-a)^3}{24}\|\mathcal{R}_{bn}\|^p_{C^2(I;\mathbb{R}^3)}M^{-2}_{bn}\right)\\
                    &\mu =\max\{\|u\|_{C(\overline{I}\times\Gamma)},\|u_\theta\|_{C(\overline{I}\times\Gamma)},\|\partial^2_{x}u\|_{C(\overline{I}\times\Gamma)},\|\partial^2_{x}u_{\theta}\|_{C(\overline{I}\times\Gamma)}\}
                \end{aligned}
            \end{equation*}
            That is asymptotically
            \begin{equation*}
                \mathcal{E}=\mathcal{O}\left(\mathcal{E}_{T,eq}+M_{eq}^{-1}+\mathcal{E}_{T,in}+M_{in}^{-2}+\left(\mathcal{E}_{T,bn}+M_{bn}^{-2}\right)^{\frac{p-1}{p}}\right)
            \end{equation*}
        \end{statement}
        \begin{proof}
            With Examples \ref{thirdderiv} and \ref{yyderiv}, we have 
            \begin{equation*}
                \begin{aligned}
                &\psi(y_\theta,y)=\left(\frac{6}{p+1}|\hat{y}|^{p+1}+\frac{6}{p}|\hat{y}|^py_2-\hat{y}''|\hat{y}|^{p-2}\hat{y}\right)\Big|_{a}^{b}+\frac{(p-1)}{2}(\hat{y}'(b))^{2}|\hat{y}(b)|^{p-2}+\\
                &+\frac{27(p-1)(p-2)}{8p^3(b-a)^2}\left(|\hat{y}(a)|^{\frac{p}{3}}+|\hat{y}(b)|^{\frac{p}{3}}\right)^3\leq 2^{\frac{1}{p}}\mu |\mathcal{R}_{bn}|_p^{p-1}+\\
                &+\left(\frac{6\mu (2p+1)}{p(p+1)}+\max\left\{\frac{(p-1)(p-2)}{2p},\frac{p-1}{p}\right\}+\frac{27(p-1)(p-2)}{2p^3(b-a)^2}\right)|\mathcal{R}_{bn}|_p^{p}
                \end{aligned}
            \end{equation*}
            It remains to apply Theorem \ref{th2}, Corollary \ref{corpar} and Lemmas \ref{midpointrule}, \ref{pc2norm}. 
        \end{proof}

        Now, we turn to the exact problem. Let $a=-1$, $b=1$, $T=1$, and 
        \begin{equation*}   
            \begin{aligned}
                &u_0(x)=\frac{1}{2}\left({\rm tanh}^2\left(\frac{1}{2}x\right)-1\right)\\
                &g_1(t)=\frac{1}{2}\left({\rm tanh}^2\left(\frac{1}{2}(1+t)\right)-1\right)\\
                &g_2(t)=\frac{1}{2}\left({\rm tanh}^2\left(\frac{1}{2}(1-t)\right)-1\right)\\
                &g_3(t)=\frac{1}{2}{\rm tanh}\left(\frac{1}{2}(1-t)\right)\left(1-{\rm tanh}^2\left(\frac{1}{2}(1-t)\right)\right)
            \end{aligned}
        \end{equation*}

        Then the exact solution is 
        \begin{equation*}
            u(t,x)=\frac{1}{2}\left({\rm tanh}^2\left(\frac{1}{2}(x-t)\right)-1\right)
        \end{equation*}
        We take 2 hidden layers of width 80, 20000 epochs, and training samples of size $M_{eq}=2^{18}$ and $M_{in}=M_{bn}=2^{12}$. Also, we make experiments for the cases $q=p\in\{2,3.5,4,5\}$. Again, we see the correlation between asymptotic estimate and total error, and that exact estimate overestimates total error (Figures \ref{kdv1},\ref{kdv2}).
        \begin{figure}[t]
            \centering
            \includegraphics[scale=0.4]{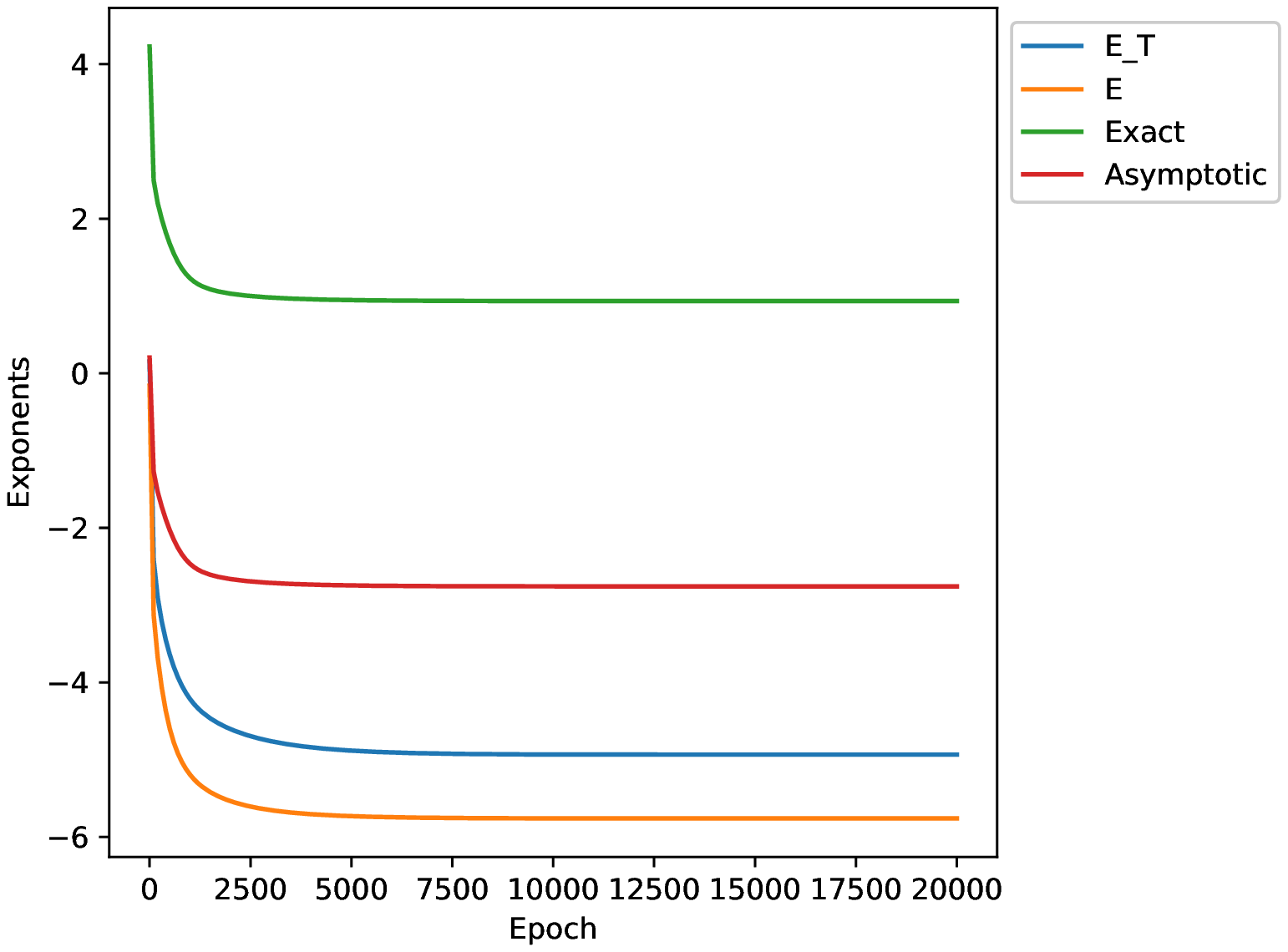}
            \includegraphics[scale=0.4]{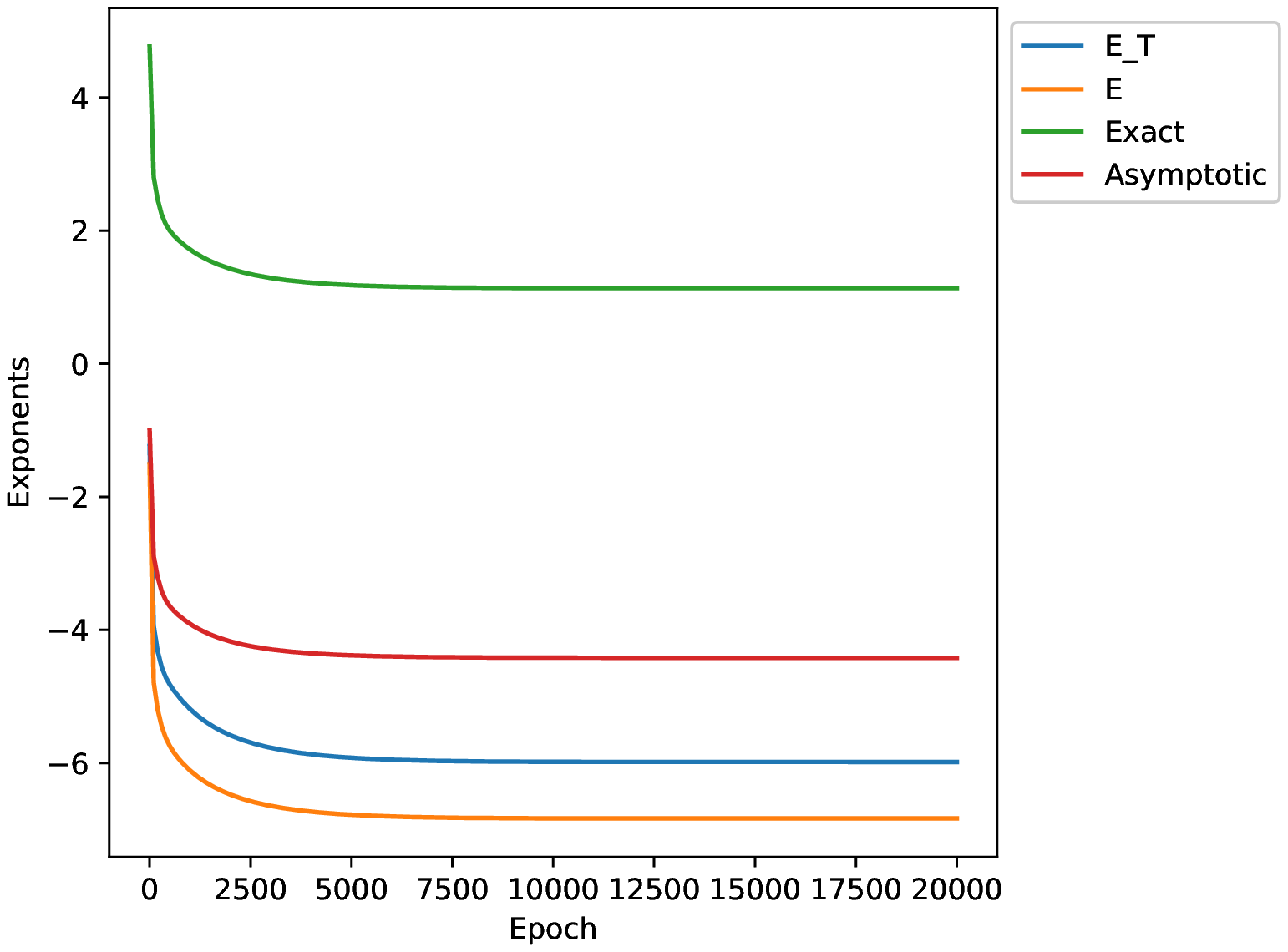}
            \includegraphics[scale=0.4]{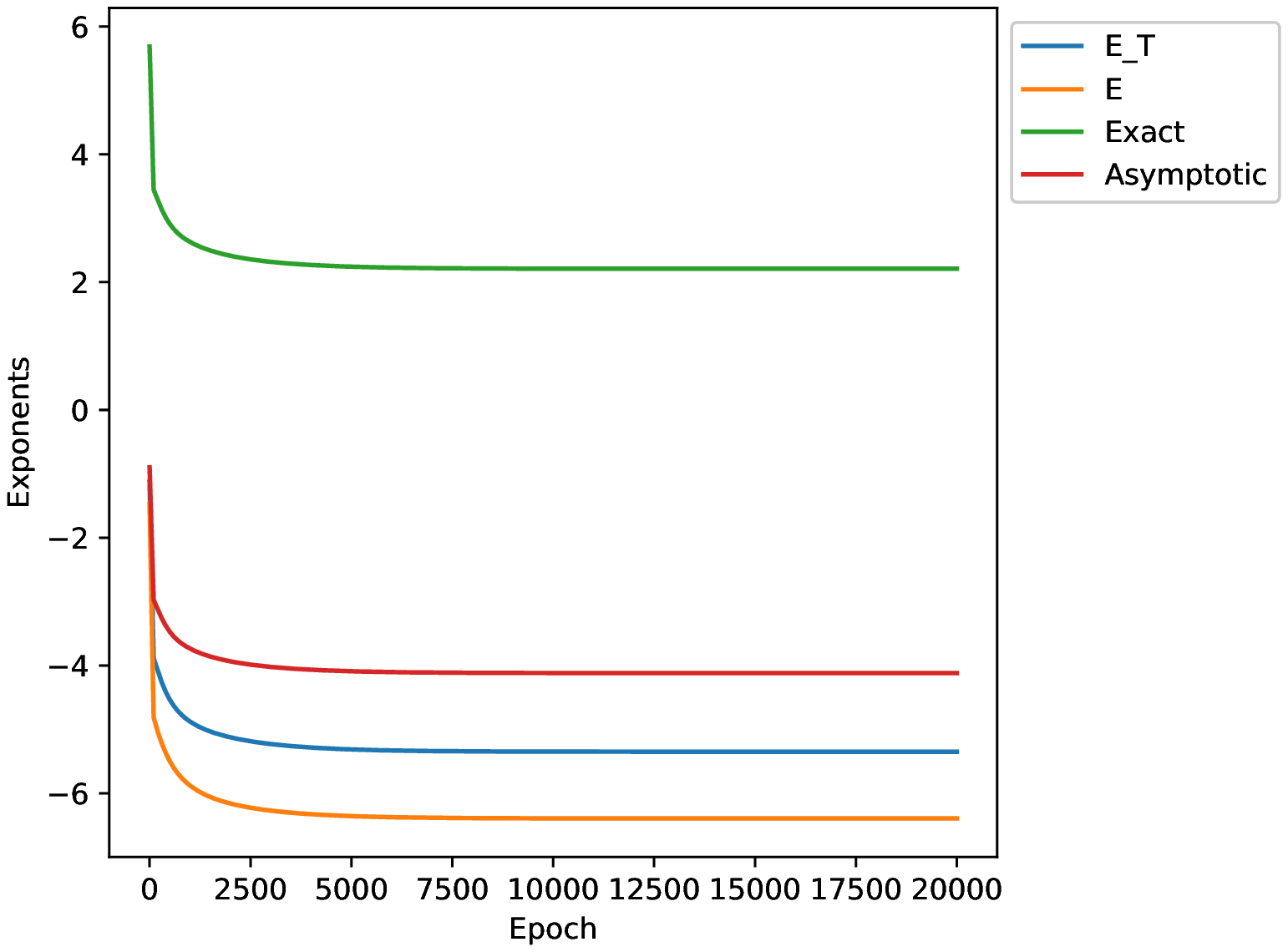}
            \includegraphics[scale=0.4]{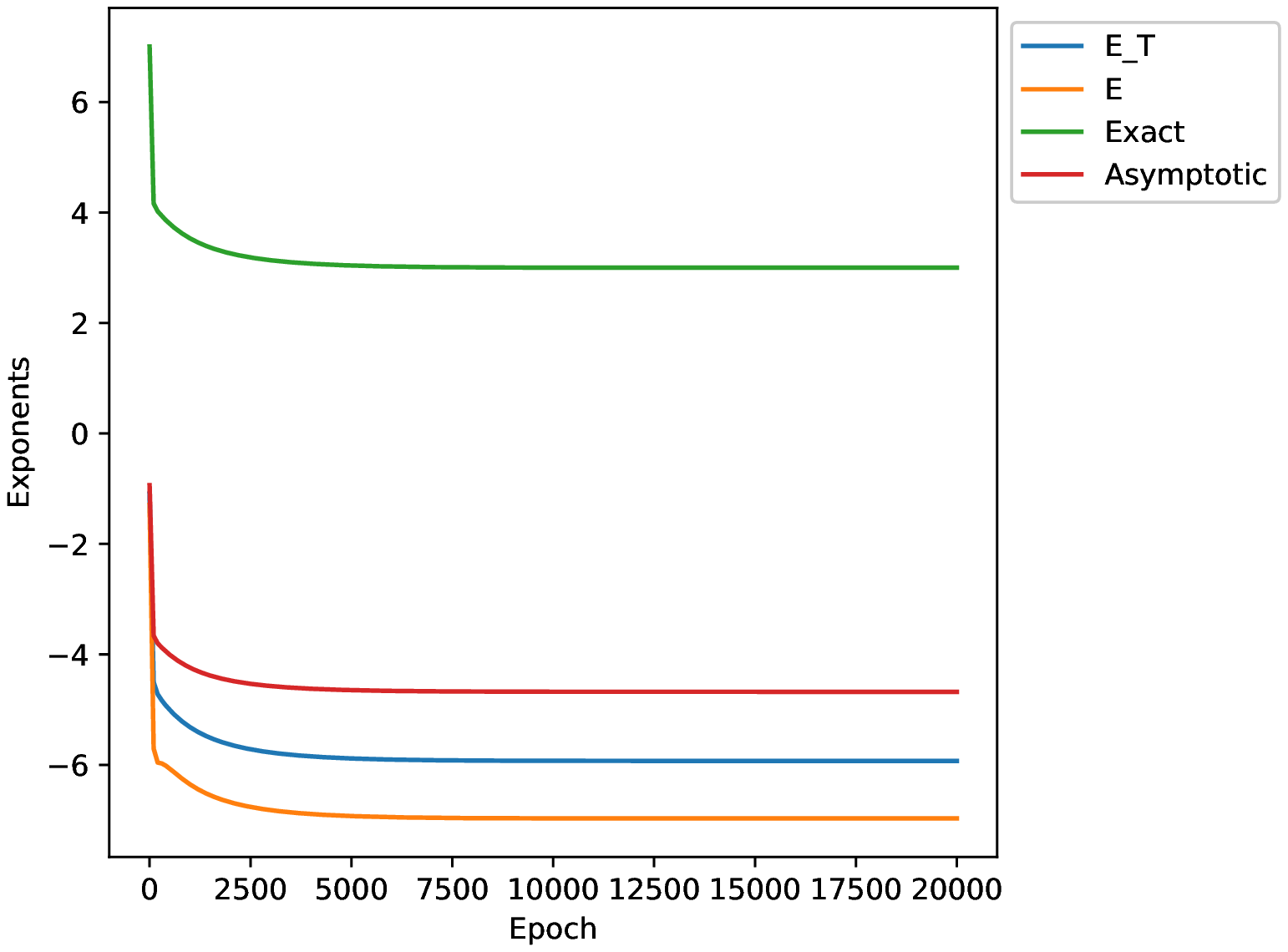}
            \caption{KDV equation: $\lg(\mathcal{E})$, $\lg(\mathcal{E}_T)$, $\lg(\mathcal{E}_{exact})$, $\lg(\mathcal{E}_{asymp})$ for $L^p$, $p\in\{2,3.5,4,5\}$}\label{kdv1}
        \end{figure} 
        \begin{figure}[t]
            \centering
            \includegraphics[scale=0.4]{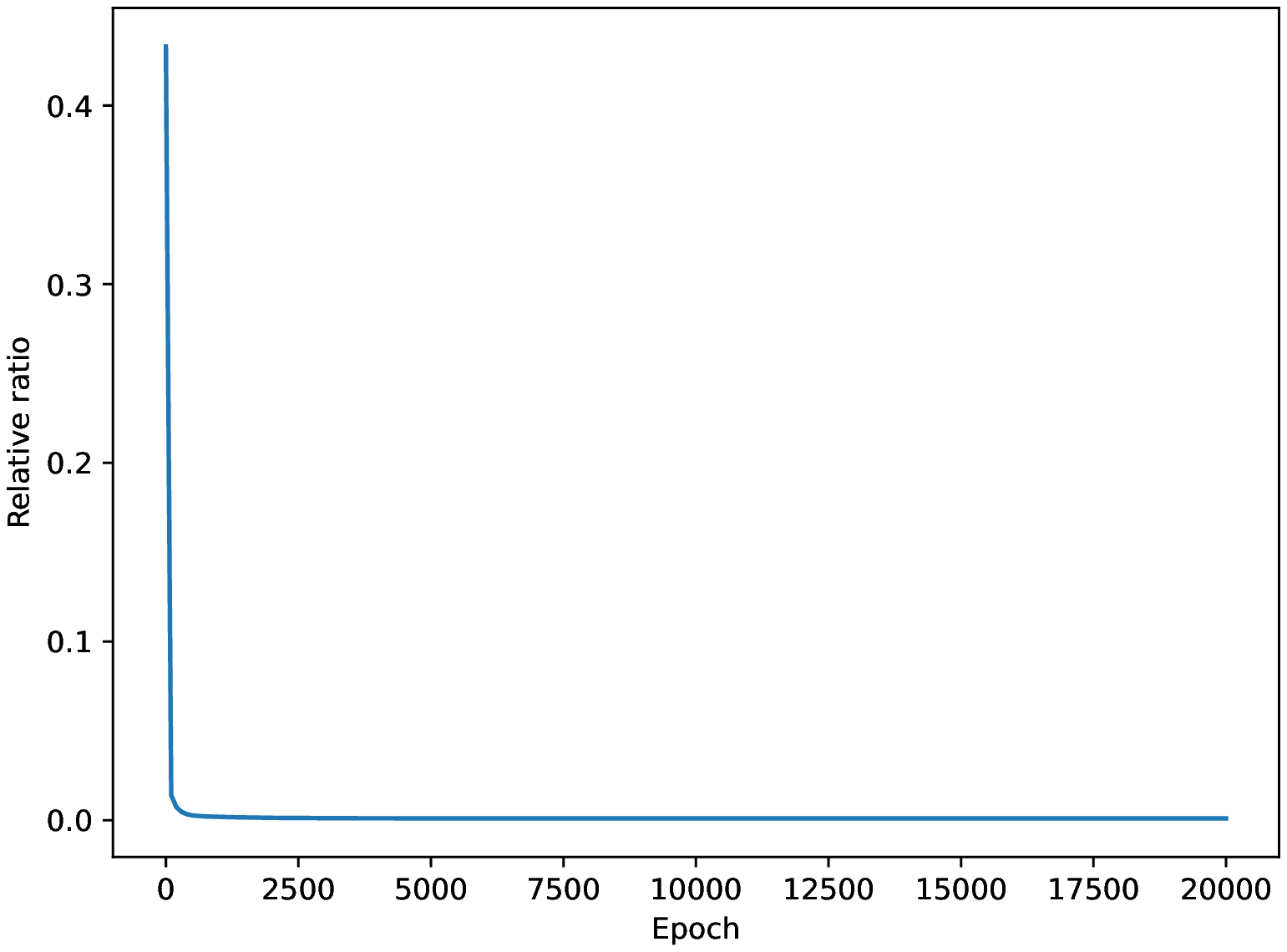}
            \includegraphics[scale=0.4]{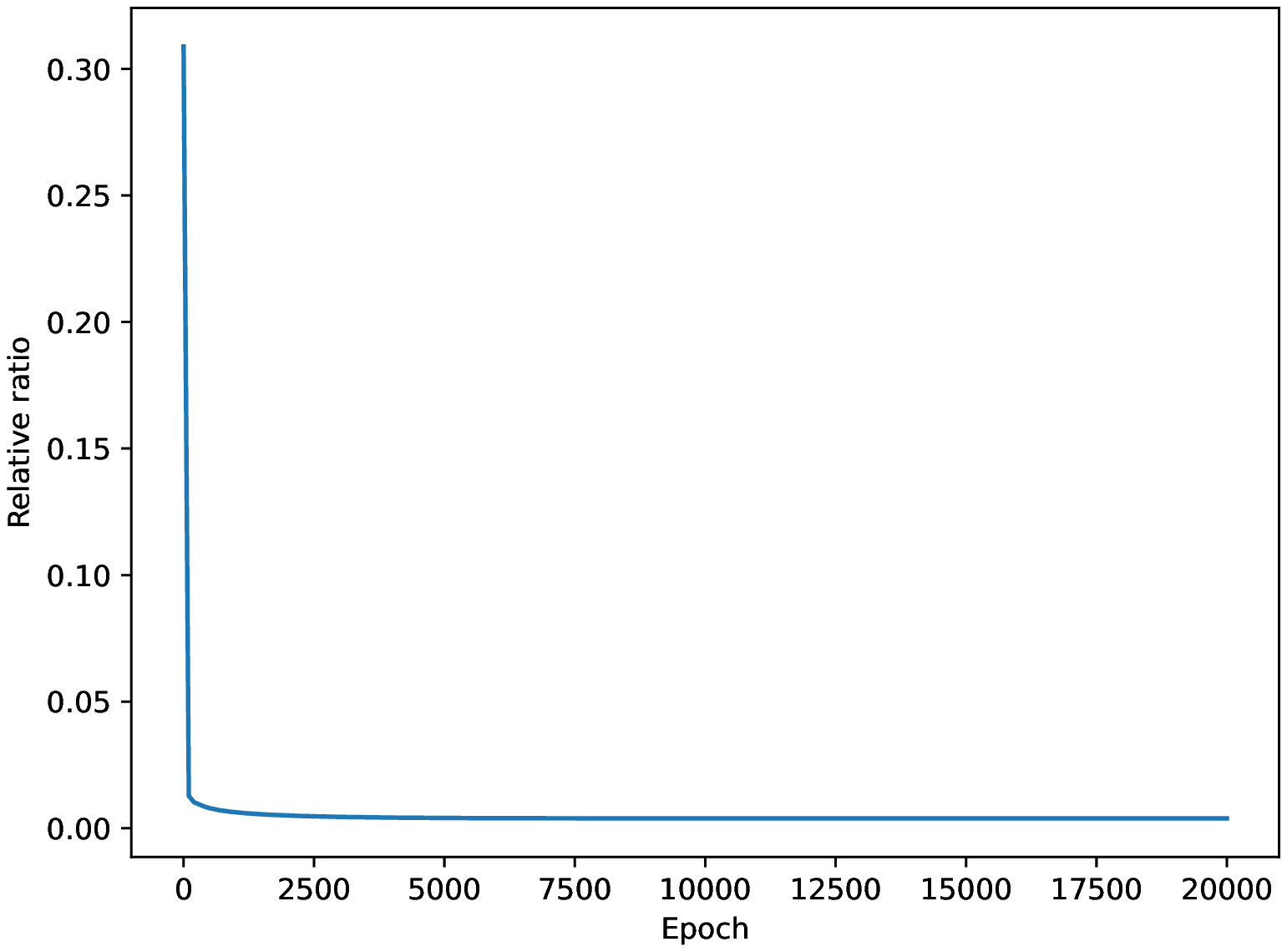}
            \includegraphics[scale=0.4]{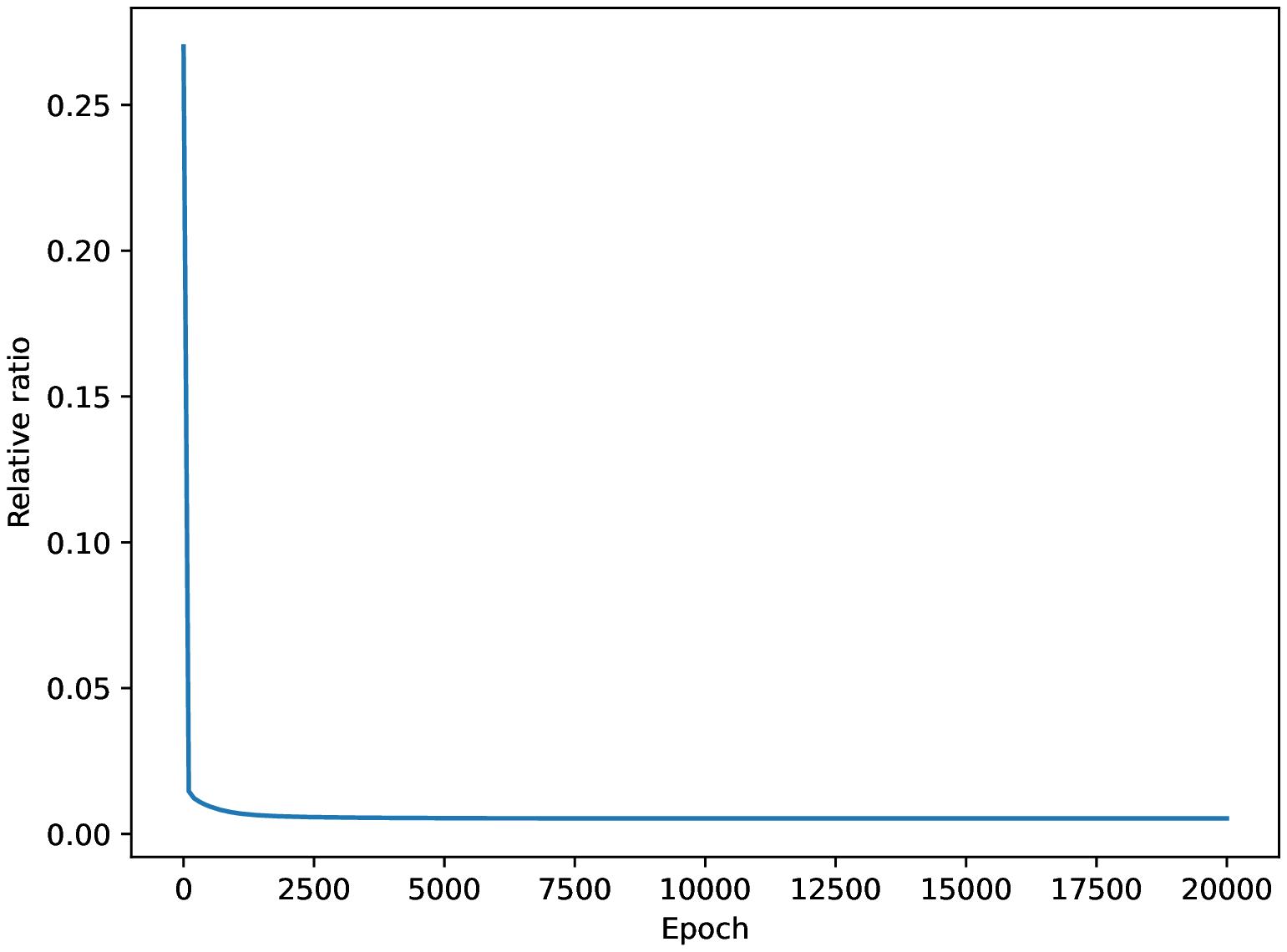}
            \includegraphics[scale=0.4]{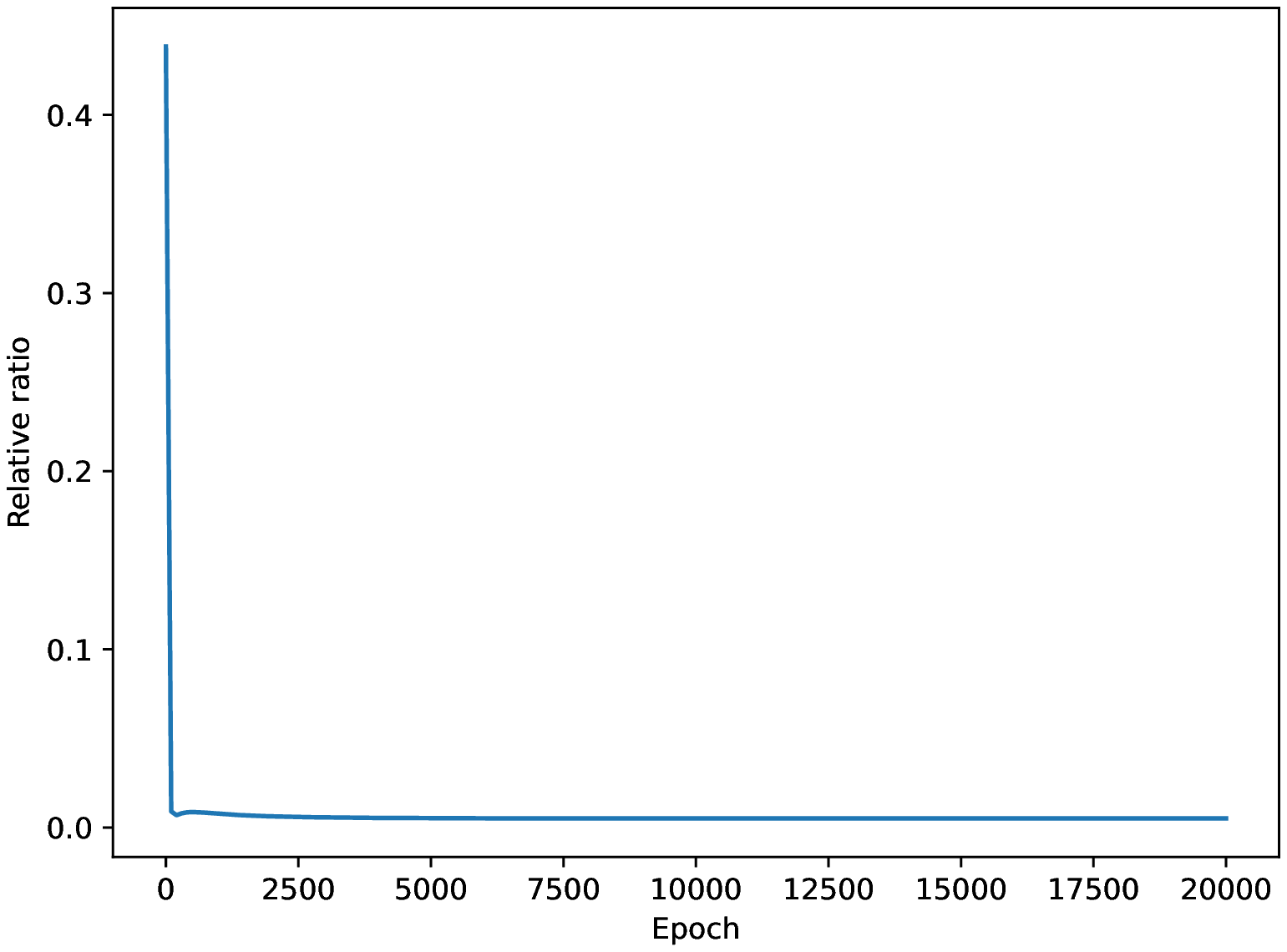}
            \caption{KDV equation: $\frac{\mathcal{E}_T}{\mathcal{E}_{asymp}}$ for $L^p$, $p\in\{2,3.5,4,5\}$}\label{kdv2}
        \end{figure} 

    \subsection{Maxwell equation}\label{maxwellss}
        Now we consider the 2D-Maxwell equation. Let $\Omega=[a_1;b_1]\times[a_2;b_2]$ and $X=L^2(\Omega;\mathbb{R}^3)$.

        \begin{equation}\label{maxwelleq}
            \begin{aligned}
                &\varepsilon_1 \partial_t u_{1} = \partial_{x_1}u_{3}-\partial_{x_2}u_{2}\\
                &\varepsilon_2 \partial_tu_{2} = -\partial_{x_2}u_{1}\\
                &\varepsilon_2 \partial_tu_{3} = \partial_{x_1}u_{1}\\
                &u_1(0,x_1,x_2)=u_{1,0}(x_1,x_2),~u_2(0,x_1,x_2)=u_{2,0}(x_1,x_2),~u_3(0,x_1,x_2)=u_{3,0}(x_1,x_2)\\
                &u_1(t,a_1,x_2)=u_1(t,b_1,x_2)=u_1(t,x_1,a_2)=u_1(t,x_1,b_2)\equiv0
            \end{aligned}
        \end{equation}

        Since the equation is linear, for every $T>0$, its solution satisfies $u\in C(\overline I;C^{\sigma}(\Omega))$ for a sufficiently regular $u_0$. 

        If $\Omega_1=[a_1;b_1]$, $\Omega_2=[a_2;b_2]$, $Y=L^2(\Omega_2)\times L^2(\Omega_2)\times L^2(\Omega_1)\times L^2(\Omega_1)$, $\Gamma_1=\partial\Omega_1$, $\Gamma_2=\partial\Omega_2$, one can formulate the following result.
        \begin{statement}
            Let $u$ be a solution to problem (\ref{maxwelleq}) and $u_\theta$ be its neural network approximation. Then, one can estimate the total error in terms of residuals:
            \begin{equation*}
                \mathcal{E}\leq \mathcal{C}\min\{\varepsilon_1,\varepsilon_2\}\left(e^{\frac{T}{\min\{\varepsilon_1,\varepsilon_2\}}}-1\right),
            \end{equation*}
            where
            \begin{equation*}
                \begin{aligned}
                    &\mathcal{C}= \|\mathcal{R}_{eq}\|_{L^2(I\times\Omega)}^2+\max\{\varepsilon_1,\varepsilon_2\}\|\mathcal{R}_{in}\|_{L^2(\Omega)}^2+\\
                    &+8\sqrt{T}\mu\|\mathcal{R}_{bn}\|_{L^2(I; Y)}\\
                    &\mu=\max\left\{\sqrt{b_1-a_1}\|u_2\|_{C(\overline{I}\times\Omega_1\times\Gamma_2)},\sqrt{b_2-a_2}\|u_3\|_{C(\overline{I}\times\Omega_2\times\Gamma_1)},\right.\\
                    &\left.\sqrt{b_1-a_1}\|u_{2,\theta}\|_{C(\overline{I}\times\Omega_1\times\Gamma_2)},\sqrt{b_2-a_2}\|u_{3,\theta}\|_{C(\overline{I}\times\Omega_2\times\Gamma_1)}\right\}
                \end{aligned}
            \end{equation*}
            Furthermore, if residuals are at least of $C^2$-class, then one can estimate the total error in terms of training error:
            \begin{equation*}
                \mathcal{E}\leq \tilde{\mathcal{C}}\min\{\varepsilon_1,\varepsilon_2\}\left(e^{\frac{T}{\min\{\varepsilon_1,\varepsilon_2\}}}-1\right),
            \end{equation*}
            where
            \begin{equation*}
                \begin{aligned}
                    &\tilde{\mathcal{C}}= \mathcal{E}_{T,eq}+\frac{T(b_2-a_2)(b_1-a_1)}{2}\left(T^2+(b_1-a_1)^2+(b_2-a_2)^2\right)\|\mathcal{R}_{eq}\|^2_{C^2(I\times\Omega)}M_{eq}^{-\frac{2}{3}}+\\
                    &+\max\{\varepsilon_1,\varepsilon_2\}\mathcal{E}_{in}+ \frac{(b_2-a_2)(b_1-a_1)}{2}((b_1-a_1)^2+(b_2-a_2)^2)\|\mathcal{R}_{in}\|_{C^2(\Omega)}^2M_{in}^{-1}+\\
                    &+8\sqrt{T}\mu\sqrt{\mathcal{E}_{T,bn}+\left[\frac{T(b_2-a_2)}{3}(T^2+(b_2-a_2)^2)+\frac{T(b_1-a_1)}{3}(T^2+(b_1-a_1)^2)\right]\|\mathcal{R}_{bn}\|_{C^2(I;Y)}^2M_{bn}^{-1}},\\
                    &\mu=\max\left\{\sqrt{b_1-a_1}\|u_2\|_{C(\overline{I}\times\Omega_1\times\Gamma_2)},\sqrt{b_2-a_2}\|u_3\|_{C(\overline{I}\times\Omega_2\times\Gamma_1)},\right.\\
                    &\left.\sqrt{b_1-a_1}\|u_{2,\theta}\|_{C(\overline{I}\times\Omega_1\times\Gamma_2)},\sqrt{b_2-a_2}\|u_{3,\theta}\|_{C(\overline{I}\times\Omega_2\times\Gamma_1)}\right\}
                \end{aligned}
            \end{equation*}
            That is asymptotically
            \begin{equation*}
                \mathcal{E}=\mathcal{O}\left(\mathcal{E}_{T,eq}+M_{eq}^{-\frac{2}{3}}+\mathcal{E}_{T,in}+M_{in}^{-1}+\sqrt{\mathcal{E}_{T,bn}+M_{bn}^{-1}}\right)
            \end{equation*}
        \end{statement}
        \begin{proof}
            For operator $A$, we refer to Example \ref{maxwellop}. Equation (\ref{maxwelleq}) exemplifies generalized parabolic type equation (\ref{genpartype}), where $U\mathbf{u}=(\varepsilon_1 u_1,\varepsilon_2 u_2, \varepsilon_2 u_3)$, and $U^{\frac{1}{2}}\mathbf{u}=(\sqrt{\varepsilon_1}u_1,\sqrt{\varepsilon_2}u_2,\sqrt{\varepsilon_2}u_3)$, and 
            \begin{equation*}
                \|\mathbf{u}\|_{U^{\frac{1}{2}}}^2=\varepsilon_1\|u_1\|^2_{L^2(\Omega)}+\varepsilon_2\|u_2\|^2_{L^2(\Omega)}+\varepsilon_2\|u_3\|^2_{L^2(\Omega)}\geq \min\{\varepsilon_1,\varepsilon_2\}\|\mathbf{u}\|^2_{L^2(\Omega;\mathbb{R}^3)}
            \end{equation*}
            Then we apply Theorem \ref{thgenpar}, ideas similar to Corollary \ref{corpar} and Lemmas \ref{midpointrule}, \ref{pc2norm}. 
        \end{proof}

        Now, we turn to the exact problem. Let $a_1=a_2=0$, $b_1=b_2=1$, $T=1$, $\varepsilon_2=3$, $\varepsilon_1=2$ and 
        \begin{equation*}
            \begin{aligned}
                &u_{1,0}(x_1,x_2)=\sin(\pi x_1)\sin(\pi x_2)\\
                &u_{2,0}(x_1,x_2)=u_{3,0}(x_1,x_2)=0
            \end{aligned}
        \end{equation*}
        Then the exact solution is 
        \begin{equation*}
            \begin{aligned}
                &u_1(t,x_1,x_2)=\sin(\pi x_1)\sin(\pi x_2)\cos\left(\sqrt{\frac{2}{\varepsilon_2\varepsilon_1}}\pi t\right)\\
                &u_2(t,x_1,x_2)=-\sqrt{\frac{\varepsilon_1}{2\varepsilon_2}}\sin(\pi x_1)\cos(\pi x_2)\sin\left(\sqrt{\frac{2}{\varepsilon_2\varepsilon_1}}\pi t\right)\\
                &u_3(t,x_1,x_2)=\sqrt{\frac{\varepsilon_1}{2\varepsilon_2}}\cos(\pi x_1)\sin(\pi x_2)\sin\left(\sqrt{\frac{2}{\varepsilon_2\varepsilon_1}}\pi t\right)
            \end{aligned}
        \end{equation*}
        We take 2 hidden layers of width 80, 20000 epochs, and training samples of size $M_{eq}=M_{in}=M_{bn}=2^{18}$. Also, we make experiments for the case $q=2$. 
        \begin{figure}[t]
            \centering
            \includegraphics[scale=0.4]{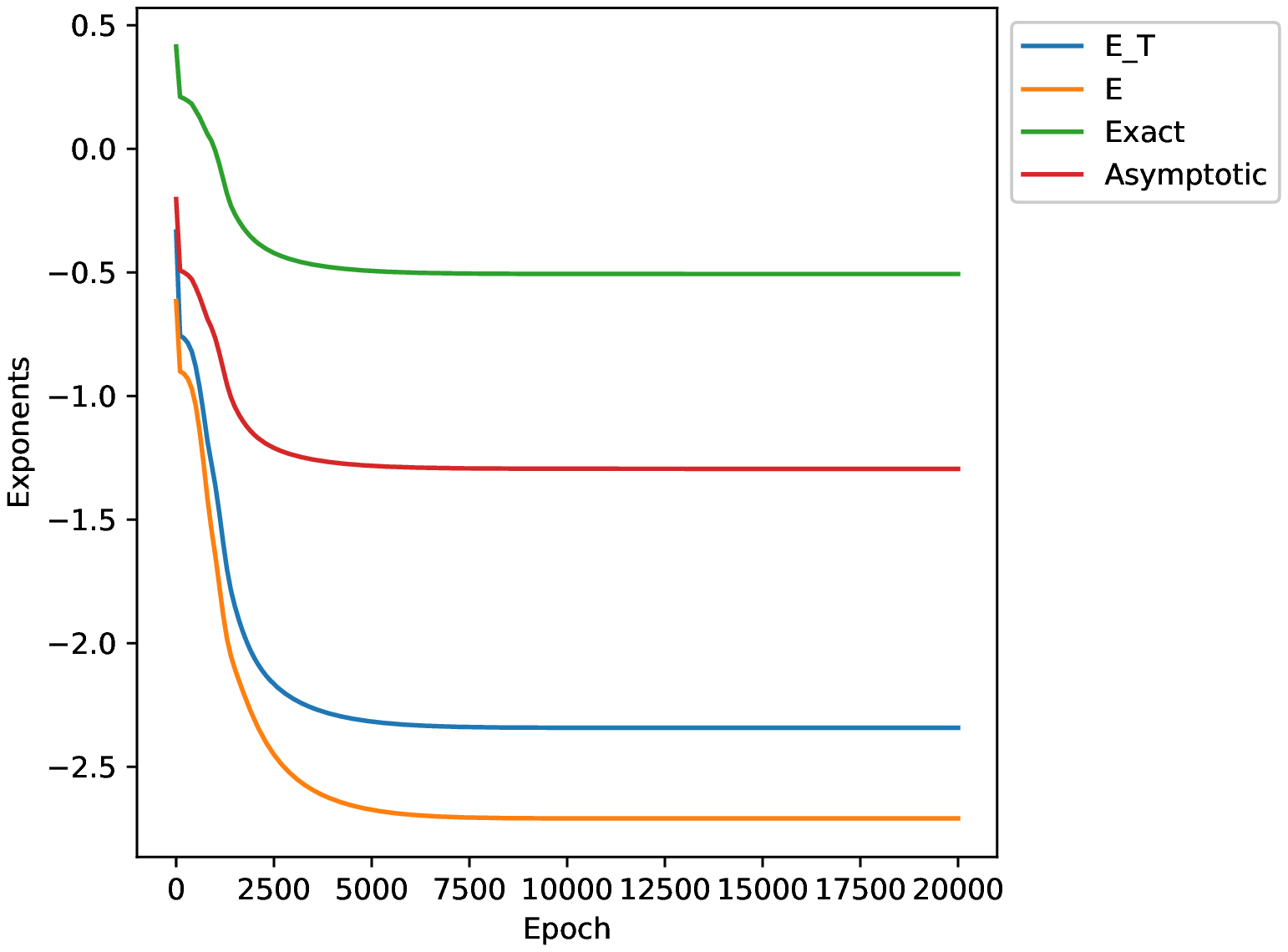}
            \includegraphics[scale=0.4]{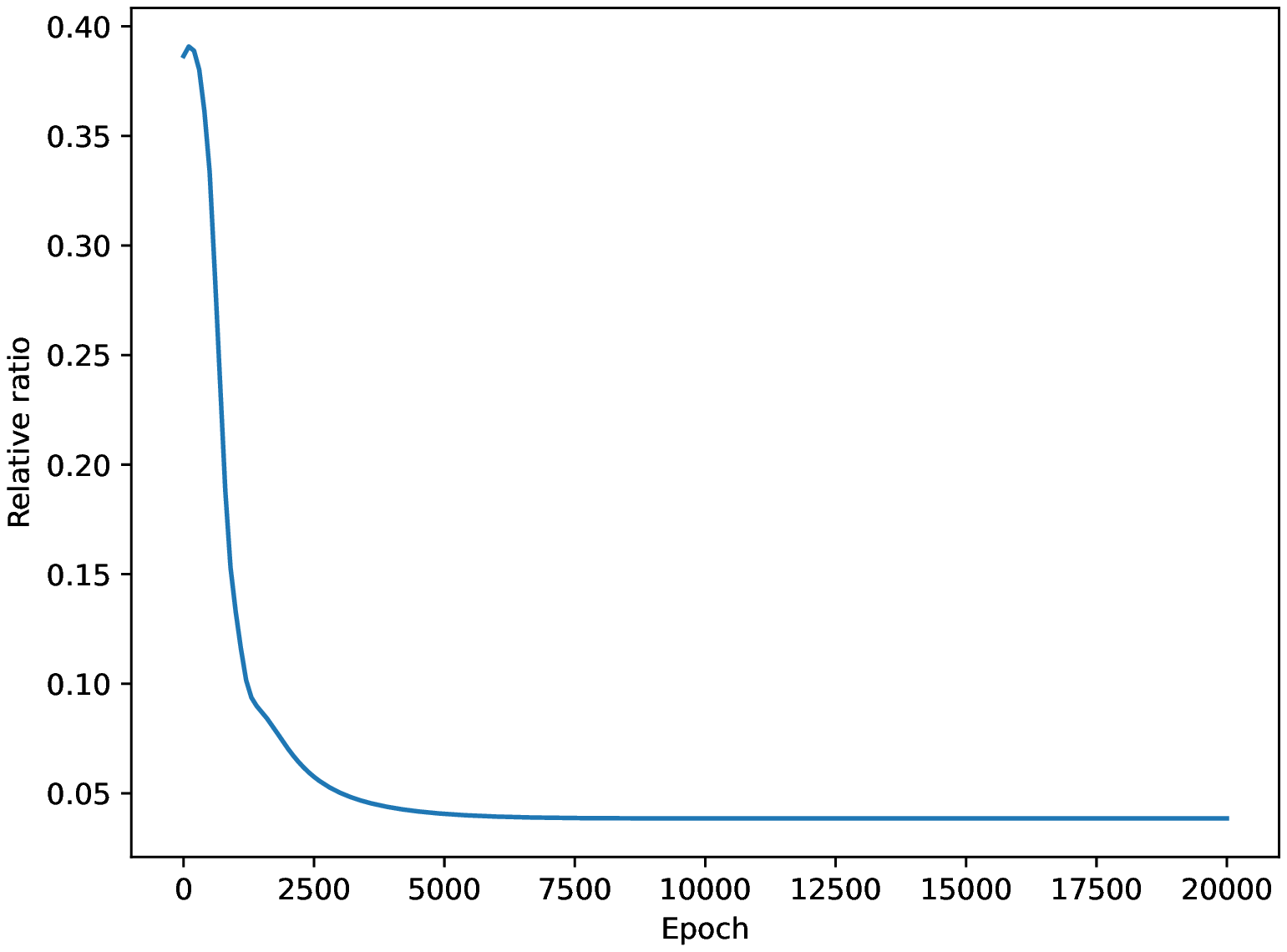}
            \caption{2D-Maxwell equation: $\lg(\mathcal{E})$, $\lg(\mathcal{E}_T)$, $\lg(\mathcal{E}_{exact})$, $\lg(\mathcal{E}_{asymp})$ and $\frac{\mathcal{E}_T}{\mathcal{E}_{asymp}}$}\label{maxwell1}
        \end{figure} 
        As we can see again, in the generalized case, the asymptotic estimate correlates with total error, while the exact estimate overestimates total error (Figure \ref{maxwell1}).

    \subsection{"Good" Boussinesq equation}\label{boussinesqss}
        The subsequent equation we consider is the so-called "good" Boussinesq equation. Let $\Omega=[a;b]$ and $X=L^2(\Omega)$.   
        \begin{equation}\label{boussinesqeq}
            \begin{aligned}
                &\partial^2_{t}u=\partial^2_{x}u-\partial^4_{x}u-\partial^2_{x}[u^2]\\
                &u(0,x)=u_0(x),~\partial_{t}u(0,x)=u_{t,0}(x)\\
                &\partial^2_{x}u(t,a)=g_1(t),~\partial^2_{x}u(t,b)=g_2(t)\\
                &\partial_{t}u(t,a)=g_3(t),~\partial_{t}u(t,b)=g_4(t)
            \end{aligned}
        \end{equation}

        One can formulate the following result.
        \begin{statement}
            Let $u$ be a solution to problem (\ref{boussinesqeq}) and $u_\theta$ be its neural network approximation. Then, one can estimate the total error in terms of residuals:
            \begin{equation*}
                \mathcal{E}\leq \mathcal{C}\frac{(e^{2T}-1)}{2}e^{36T\max\{\|u\|_{C^2(I\times\Omega)}^2,\|u_\theta\|_{C^2}^2(I\times\Omega)\}},
            \end{equation*}
            where
            \begin{equation*}
                \begin{aligned}
                    &\mathcal{C}=\|\mathcal{R}_{eq}\|_{L^2(I; X)}^2+\|\mathcal{R}_{in}\|_{H^2(\Omega)}^2+\|\mathcal{R}_{in,t}\|_{L^2(\Omega)}^2+\\
                    &+4\sqrt{2T}\max\{\|\partial_{t}\partial_xu_{\theta}\|_{C(\overline{I}\times\Gamma)},\|\partial_{t}\partial_xu\|_{C(\overline{I}\times\Gamma)}\}\|\mathcal{R}_{bn}\|_{L^2(I;\mathbb{R}^2)}+\\
                    &+8\sqrt{2T}\max\{\|\partial_{x}u_{\theta}\|_{C(\overline{I}\times\Gamma)},\|\partial_{x}u\|_{C(\overline{I}\times\Gamma)},\|\partial^3_{x}u_{\theta}\|_{C(\overline{I}\times\Gamma)},\|\partial^3_{x}u\|_{C(\overline{I}\times\Gamma)}\}\|\mathcal{R}_{bn,t}\|_{L^2(I;\mathbb{R}^2)}
                \end{aligned}
            \end{equation*}
            Furthermore, if residuals are at least of $C^2$-class, then one can estimate the total error in terms of training error:
            \begin{equation*}
                \mathcal{E}\leq \tilde{\mathcal{C}}\frac{(e^{2T}-1)}{2}e^{36T\max\{\|u\|_{C^2(I\times\Omega)}^2,\|u_\theta\|_{C^2}^2(I\times\Omega)\}},
            \end{equation*}
            where
            \begin{equation*}
                \begin{aligned}
                    &\tilde{\mathcal{C}}=\mathcal{E}_{T,eq}+\frac{(b-a)T}{6}\left((b-a)^2+T^2\right)\left\|\mathcal{R}_{eq}\right\|^2_{C^2(I\times \Omega)}M_{eq}^{-1}+\\
                    &+\mathcal{E}_{T,in}+\frac{(b-a)^3}{2}\left\|\mathcal{R}_{in}\right\|^2_{C^4(\Omega)}M_{in}^{-2}+\mathcal{E}_{T,in,t}+\frac{(b-a)^3}{6}\left\|\mathcal{R}_{in,t}\right\|^2_{C^2(\Omega)}M_{in,t}^{-2}+\\
                    &+4\sqrt{2T}\max\{\|\partial_{t}\partial_xu_{\theta}\|_{C(\overline{I}\times\Gamma)},\|\partial_{t}\partial_xu\|_{C(\overline{I}\times\Gamma)}\}\sqrt{\mathcal{E}_{T,bn}+\frac{T^3}{3}\|\mathcal{R}_{bn}\|_{C(\overline{I};\mathbb{R}^2)}^2M_{bn}^{-2}}+\\
                    &+8\sqrt{2T}\max\{\|\partial_{x}u_{\theta}\|_{C(\overline{I}\times\Gamma)},\|\partial_{x}u\|_{C(\overline{I}\times\Gamma)},\|\partial^3_{x}u_{\theta}\|_{C(\overline{I}\times\Gamma)},\|\partial^3_{x}u\|_{C(\overline{I}\times\Gamma)}\}\sqrt{\mathcal{E}_{T,bn,t}+\frac{T^3}{3}\|\mathcal{R}_{bn,t}\|_{C^2(I;\mathbb{R}^2)}^2M_{bn,t}^{-2}}
                \end{aligned}
            \end{equation*}
            That is asymptotically
            \begin{equation*}
                \mathcal{E}=\mathcal{O}\left(\mathcal{E}_{T,eq}+M_{eq}^{-1}+\mathcal{E}_{T,in}+M_{in}^{-2}+\mathcal{E}_{T,in,t}+M_{in,t}^{-2}+\sqrt{\mathcal{E}_{T,bn}+M_{bn}^{-2}}+\sqrt{\mathcal{E}_{T,bn,t}+M_{bn,t}^{-2}}\right)
            \end{equation*}
        \end{statement}
        \begin{proof}
            We have
            \begin{equation*}
                A(t)\equiv0,~U_1=-\partial^2_{x},~U_2=\partial^4_{x},~F(t)(u)\equiv -\partial^2_{x}[u^2],~\|\cdot\|_{{\rm Id}+U^{\frac{1}{2}}_1+U^{\frac{1}{2}}_2}=\|\cdot \|_{H^2(\Omega)}
            \end{equation*}
            For $F$, we have
            \begin{equation*}
                \begin{aligned}
                    &\|F(y_1)-F(y_2)\|_{L^2(\Omega)}^2\leq 36\max\{\|y_1\|_{C^2(\Omega)}^2,\|y_2\|_{C^2(\Omega)}^2)\}\|\hat{y}\|^2_{H^2(\Omega)}
                \end{aligned}
            \end{equation*}
            As for the $U_1$ and $U_2$,
            \begin{equation*}
                |\epsilon_1(\hat{\chi},\hat{y})|\leq 2\max\{\|\chi'_1\|_{C(\Gamma)},\|\chi'_2\|_{C(\Gamma)}\}\left(|B_2\hat{y}|_1+|B_2\hat{y}|_2\right)
            \end{equation*}
            and
            \begin{equation*}
                \begin{aligned}
                    &|\epsilon_2(\hat{\chi},\hat{y})|\leq 2\max\{\|\chi'''_1\|_{C(\Gamma)},\|\chi'''_2\|_{C(\Gamma)}\}\left(|B_2\hat{y}|_1+|B_2\hat{y}|_2\right)+2\max\{\|y'_1\|_{C(\Gamma)},\|y'_2\|_{C(\Gamma)}\}\left(|B_1\hat{\chi}|_1+|B_1\hat{\chi}|_2\right)
                \end{aligned}
            \end{equation*}
            Then we apply Theorem \ref{thhyp}, ideas similar to Corollary \ref{corpar} and Lemmas \ref{midpointrule}, \ref{pc2norm}. 
        \end{proof}

        Now, we turn to the exact problem. Let $a=-1$, $b=1$, $T=1$ and 
        \begin{equation*}
            \begin{aligned}
                &u_0(x)=\frac{9}{8}{\rm sech}^2\left(x\right)\\
                &u_{t,0}(x)=\frac{9\sqrt{3}}{32}{\rm tanh}\left(x\right){\rm sech}^2\left(x\right)\\
                &g_1(t)=-\frac{27}{64}{\rm sech}^2\left(1+\frac{t}{2}\right)\left(1-3{\rm tanh}^2\left(1+\frac{t}{2}\right)\right)\\
                &g_2(t)=-\frac{27}{64}{\rm sech}^2\left(1-\frac{t}{2}\right)\left(1-3{\rm tanh}^2\left(1-\frac{t}{2}\right)\right)\\
                &g_3(t)=-\frac{9\sqrt{3}}{32}{\rm tanh}\left(1+\frac{t}{2}\right){\rm sech}^2\left(1+\frac{t}{2}\right)\\
                &g_4(t)=\frac{9\sqrt{3}}{32}{\rm tanh}\left(1-\frac{t}{2}\right){\rm sech}^2\left(1-\frac{t}{2}\right)
            \end{aligned}
        \end{equation*}
        Then the exact solution is 
        \begin{equation*}
            u(t,x)=\frac{9}{8}{\rm sech}^2\left(x-\frac{t}{2}\right)
        \end{equation*}
        We take 2 hidden layers of width 80, 20000 epochs, and training samples of size $M_{eq}=2^{18}$ and $M_{in}=M_{in,t}=M_{bn}=M_{bn,t}=2^{12}$. 
        \begin{figure}[t]
            \centering
            \includegraphics[scale=0.4]{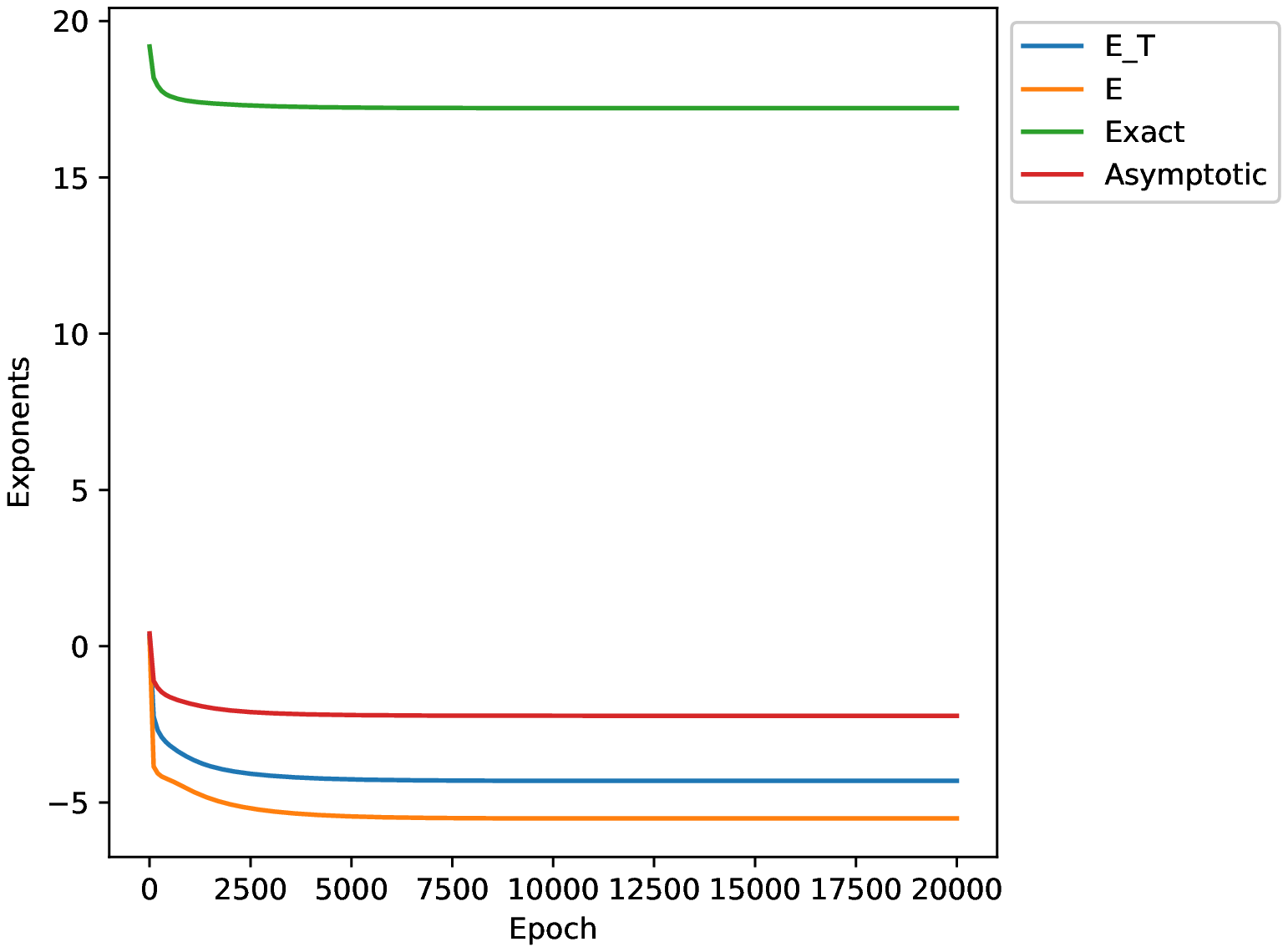}
            \includegraphics[scale=0.4]{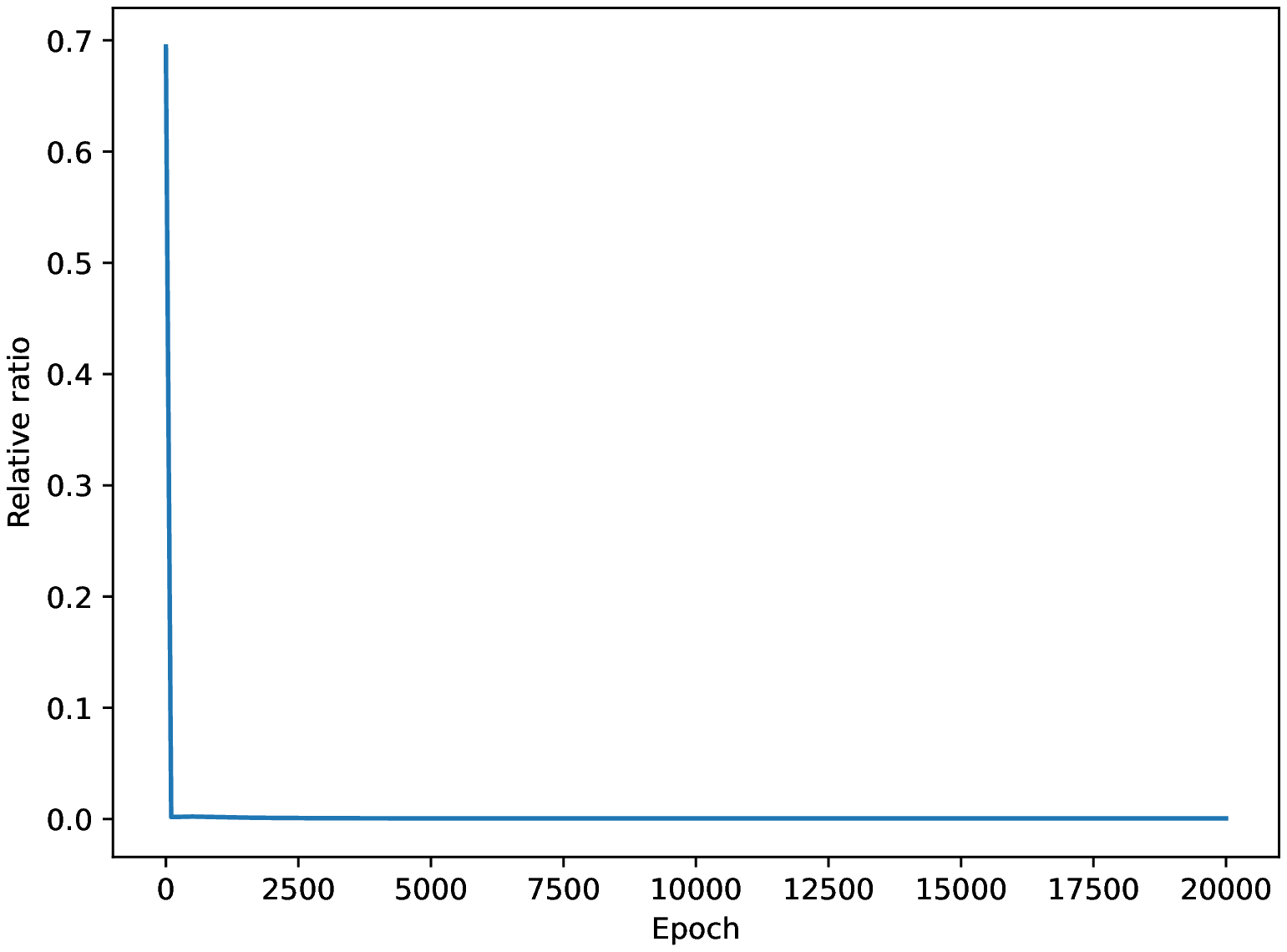}
            \caption{Boussinesq equation: $\lg(\mathcal{E})$, $\lg(\mathcal{E}_T)$, $\lg(\mathcal{E}_{exact})$, $\lg(\mathcal{E}_{asymp})$ and $\frac{\mathcal{E}_T}{\mathcal{E}_{asymp}}$}\label{boussinesq1}
        \end{figure} 
        For the hyperbolic case, again, we see the correlation between the asymptotic estimate and total error, and that exact estimate highly overestimates total error (Figure \ref{boussinesq1}). It is not surprising since the exact estimate involves $e^{36T}$ as a factor.  

    \subsection{Rayleigh wave equation}\label{rayleighss}
        The following equation we consider is the Rayleigh wave equation. Let $\Omega=[a;b]$ and $X=L^2(\Omega)$.     

        \begin{equation}\label{rayleigheq}
            \begin{aligned}
                &\partial^2_{t}u=\partial^2_{x}u+\varepsilon(\partial_{t}u-(\partial_{t}u)^3)+\phi\\
                &u(0,x)=u_0(x),~\partial_{t}u(0,x)=u_{0,t}(x)\\
                &\partial_{x}u(t,a)=\partial_{x}u(t,b)\\
                &\partial_{t}u(t,a)=\partial_{t}u(t,b)\\
            \end{aligned}
        \end{equation}

        One can formulate the following result.
        \begin{statement}
            Let $u$ be a solution to problem (\ref{rayleigheq}) and $u_\theta$ be its neural network approximation. Then, one can estimate the total error in terms of residuals:
            \begin{equation*}
                \mathcal{E}\leq \mathcal{C}\frac{(e^{2T}-1)}{2}e^{2\varepsilon T},
            \end{equation*}
            where
            \begin{equation*}
                \begin{aligned}
                    &\mathcal{C}=\|\mathcal{R}_{eq}\|_{L^2(I\times\Omega)}^2+\|\mathcal{R}_{in}\|_{H^1(\Omega)}^2+\|\mathcal{R}_{in,t}\|_{L^2(\Omega)}^2+\\
                    &+4\sqrt{T}\max\{\|\partial_{t}u_{\theta}\|_{C(\overline{I}\times\Gamma)},\|\partial_{t}u\|_{C(\overline{I}\times\Gamma)}\}\|\mathcal{R}_{bn}\|_{L^2(I)}+\\
                    &+4\sqrt{T}\max\{\|\partial_{x}u_{\theta}\|_{C(\overline{I}\times\Gamma)},\|\partial_{x}u\|_{C(\overline{I}\times\Gamma)}\}\|\mathcal{R}_{bn,t}\|_{L^2(I)}
                \end{aligned}
            \end{equation*}
            Furthermore, if residuals are at least of $C^2$-class, then one can estimate the total error in terms of training error:
            \begin{equation*}
                \mathcal{E}\leq \tilde{\mathcal{C}}\frac{(e^{2T}-1)}{2}e^{2\varepsilon T},
            \end{equation*}
            where
            \begin{equation*}
                \begin{aligned}
                    &\tilde{\mathcal{C}}=\mathcal{E}_{T,eq}+\frac{(b-a)T}{6}\left((b-a)^2+T^2\right)\left\|\mathcal{R}_{eq}\right\|^2_{C^2(I\times \Omega)}M_{eq}^{-1}+\\
                    &+\mathcal{E}_{T,in,\sigma}+\frac{(b-a)^3}{3}\left\|\mathcal{R}_{in}\right\|^2_{C^3(\Omega)}M_{in}^{-2}+\mathcal{E}_{T,in,t}+\frac{(b-a)^3}{6}\left\|\mathcal{R}_{in}\right\|^2_{C^2(\Omega)}M_{in,t}^{-2}+\\
                    &+4\sqrt{T}\max\{\|\partial_{t}u_{\theta}\|_{C(\overline{I}\times\Gamma)},\|\partial_{t}u\|_{C(\overline{I}\times\Gamma)}\}\sqrt{\mathcal{E}_{T,bn}+\frac{T^3}{6}\|\mathcal{R}_{bn}\|^2_{C^2(I)}M_{bn}^{-2}}+\\
                    &+4\sqrt{T}\max\{\|\partial_{x}u_{\theta}\|_{C(\overline{I}\times\Gamma)},\|\partial_{x}u\|_{C(\overline{I}\times\Gamma)}\}\sqrt{\mathcal{E}_{T,bn,t}+\frac{T^3}{6}\|\mathcal{R}_{bn,t}\|^2_{C^2(I)}M_{bn,t}^{-2}}
                \end{aligned}
            \end{equation*}
            That is asymptotically
            \begin{equation*}
                \mathcal{E}=\mathcal{O}\left(\mathcal{E}_{T,eq}+M_{eq}^{-1}+\mathcal{E}_{T,in}+M_{in}^{-2}+\mathcal{E}_{T,in,t}+M_{in,t}^{-2}+\sqrt{\mathcal{E}_{T,bn}+M_{bn}^{-2}}+\sqrt{\mathcal{E}_{T,bn,t}+M_{bn,t}^{-2}}\right)
            \end{equation*}
        \end{statement}
        \begin{proof}
            We have
            \begin{equation*}
                A(t)(y)=\varepsilon(y-y^3),~U=-\partial^2_{x},~F(t)\equiv0,~\|\cdot\|_{{\rm Id}+U^{\frac{1}{2}}}=\|\cdot \|_{H^1(\Omega)}
            \end{equation*}
            For $A$, we have
            \begin{equation*}
                \begin{aligned}
                    &\langle A(y_1)-A(y_2),\hat{y}\rangle_2=\varepsilon\|\hat{y}\|^2-\varepsilon\int_{a}^{b}(y^2_1+y_1y_2+y^2_2)\hat{y}^2dx\leq \varepsilon\|\hat{y}\|^2
                \end{aligned}
            \end{equation*}
            As for the operator $U$, 
            \begin{equation*}   
                \begin{aligned}
                    &|\epsilon_1(\hat{\chi},\hat{y})|\leq 2\max\{\|y_1\|_{C(\Gamma)},\|y_2\|_{C(\Gamma)}\}|B_1\hat{\chi}|+2\max\{\|\chi'_1\|_{C(\Gamma)},\|\chi'_2\|_{C(\Gamma)}\}|B_2\hat{y}| 
                \end{aligned}
            \end{equation*}
            Then we apply Theorem \ref{thhyp}, ideas similar to Corollary \ref{corpar} and Lemmas \ref{midpointrule}, \ref{pc2norm}.  
        \end{proof}

        Now, we turn to the exact problem. Let $a=0$, $b=2\pi$, $T=1$, $\varepsilon=1$, $\phi=\frac{\varepsilon}{2}\sin(x-t)\sin(2(x-t))$ and 
        \begin{equation*}
            \begin{aligned}
                &u_0(x)=\sin(x)\\
                &u_{0,t}(x)=-\cos(x)
            \end{aligned}
        \end{equation*}
        Then the exact solution is 
        \begin{equation*}
            u(t,x)=\sin(x-t)
        \end{equation*}
        We take 2 hidden layers of width 80, 20000 epochs, and training samples of size $M_{eq}=2^{18}$ and $M_{in}=M_{in,t}=M_{bn}=M_{bn,t}=2^{12}$. Also, we make experiments for the case $q=2$. 
        \begin{figure}[t]
            \centering
            \includegraphics[scale=0.4]{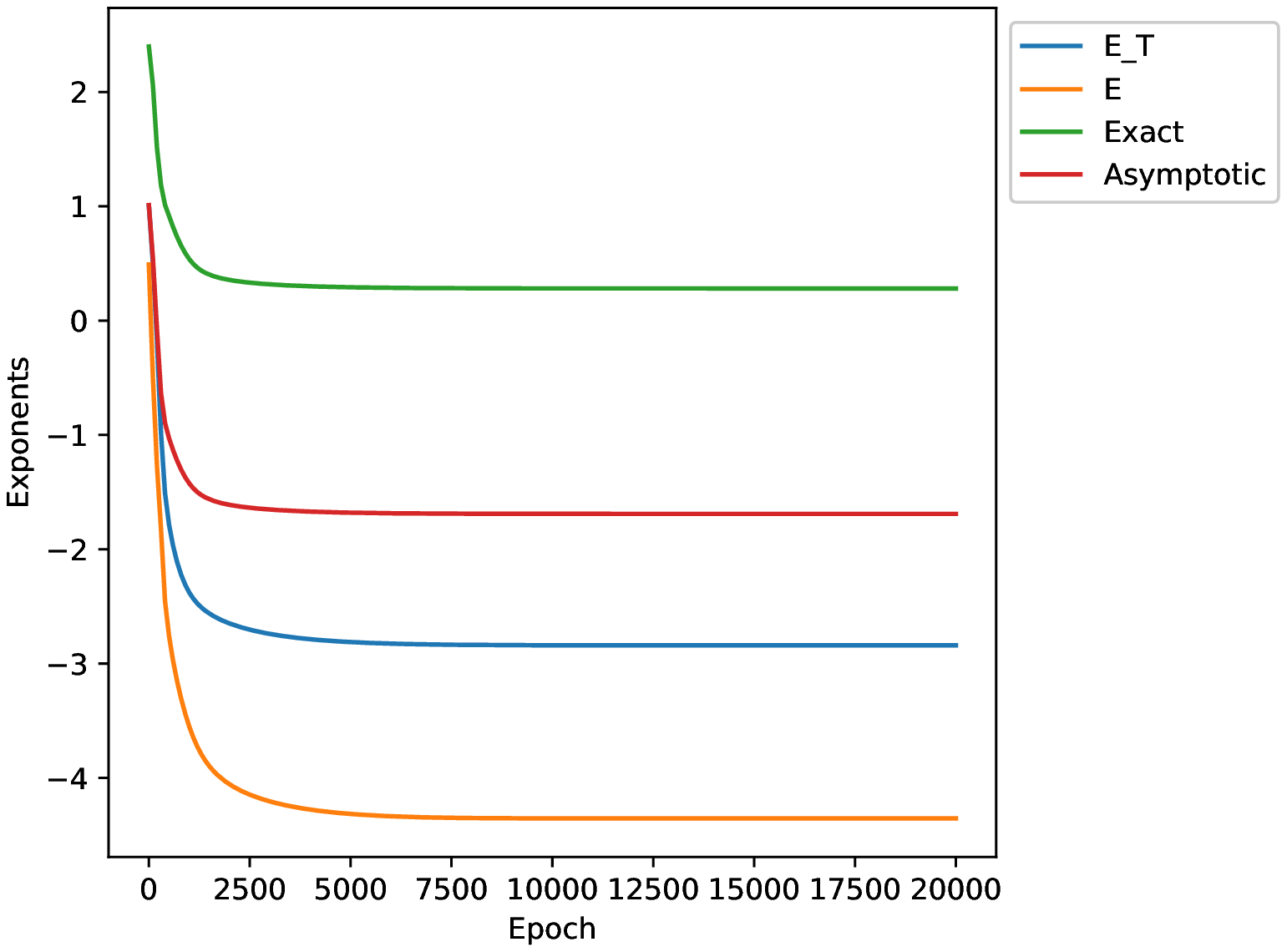}
            \includegraphics[scale=0.4]{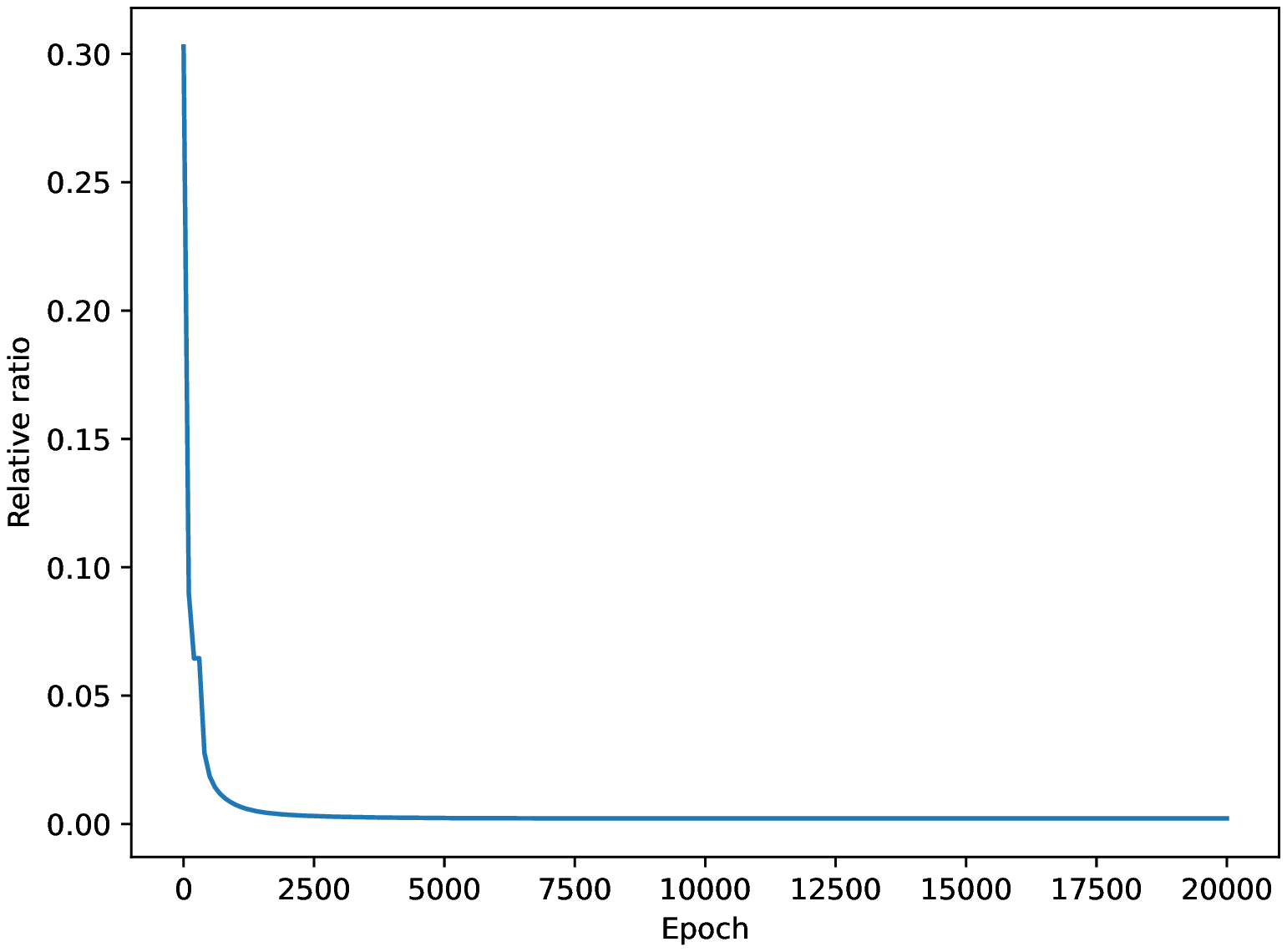}
            \caption{Rayleigh equation: $\lg(\mathcal{E})$, $\lg(\mathcal{E}_T)$, $\lg(\mathcal{E}_{exact})$, $\lg(\mathcal{E}_{asymp})$ and $\frac{\mathcal{E}_T}{\mathcal{E}_{asymp}}$}\label{rayleigh1}
        \end{figure} 
        Again, we see the correlation between the asymptotic estimate and total error, and that exact estimate overestimates total error (Figure \ref{rayleigh1}). Also, we observe that the exact estimate $<1$ since its exponent factor $e^{2\varepsilon T}=e^2$ is not substantial. 

    \subsection{Poisson equation with piecewise continuous forcing function}\label{ellipticss}
        The last equation we consider is the Poisson equation with piecewise continuous forcing function $\phi$. Let $\Omega=[-1;1]\times[-1;1]$, $X=L^p(\Omega)$ and $q=p\geq2$. 

        \begin{equation}\label{ellipticeq}
            \begin{aligned}
                &-\Delta y= \phi,\\
                &y(-1,x_2)=f_{1},~y(1,x_2)=f_{2}\\
                &y(x_1,-1)=f_{3},~y(x_1,1)=f_{4}
            \end{aligned}
        \end{equation} 

        For the estimates, one can formulate the following result.
        \begin{statement}
            Let $u$ be a solution to problem (\ref{ellipticeq}) and $u_\theta$ be its neural network approximation. Then, one can estimate the total error in terms of residuals:
            \begin{equation*}
                \begin{aligned}
                &\mathcal{E}\leq \frac{\pi_{2,tr}^pp^{2p}}{2^p(p-1)^p}\|\mathcal{R}_{eq}\|_{L^p(\Omega)}^p+2\pi_{2,tr}p\left(2^{p-1}\|\mathcal{R}_{bn}\|^p_{L^p(\Omega;\mathbb{R}^4)}+\right.\\
                &+\left. \frac{2^{\frac{(p+1)(p-1)^3}{p^3}-2}p}{(p-1)}\max\{\|\partial_{x_1}y\|_{C(\Omega)},\|\partial_{x_1}y_{\theta}\|_{C(\Omega)},\|\partial_{x_2}y\|_{C(\Omega)},\|\partial_{x_2}y_{\theta}\|_{C(\Omega)}\}\|\mathcal{R}_{bn}\|^{p-1}_{L^p((-1;1);\mathbb{R}^4)}\right),
                \end{aligned}
            \end{equation*}
            where $\pi_{2,tr}$ is a constant of $L^2$-Poincare inequality with trace term. 

            Furthermore, if residuals are piecewise at least of $C^2$-class, then one can estimate the total error in terms of training error:
            \begin{equation*}
                \begin{aligned}
                &\mathcal{E}\leq \frac{\pi_{2,tr}^pp^{2p}}{2^p(p-1)^p}\left(\mathcal{E}_{eq,T} + \frac{4p^2}{3}\left\|\mathcal{R}_{eq}\right\|^p_{W^{2,\infty}(\Omega)}M^{-1}_{eq}\right)+\\
                &+2\pi_{2,tr}p\left(2^{p-1}\left(\mathcal{E}_{bn,T} + \frac{4p^2}{3}\|\mathcal{R}_{bn}\|_{W^{2,\infty}((-1;1);\mathbb{R}^4)}^pM_{bn}^{-2}\right)+\right.\\
                &+\left. \frac{2^{\frac{(p+1)(p-1)^3}{p^3}-2}p}{(p-1)}\max\{\|\partial_{x_1}y\|_{C(\Omega)},\|\partial_{x_1}y_{\theta}\|_{C(\Omega)},\|\partial_{x_2}y\|_{C(\Omega)},\|\partial_{x_2}y_{\theta}\|_{C(\Omega)}\}\cdot\right.\\
                &\left.\cdot\left(\mathcal{E}_{bn,T} + \frac{4p^2}{3}\|\mathcal{R}_{bn}\|_{W^{2,\infty}((-1;1);\mathbb{R}^4)}^pM_{bn}^{-2}\right)^{\frac{p-1}{p}}\right)
                \end{aligned}
            \end{equation*}
            That is asymptotically
            \begin{equation*}
                \mathcal{E}=\mathcal{O}\left(\mathcal{E}_{T,eq}+M_{eq}^{-1}+\left(\mathcal{E}_{T,bn}+M_{bn}^{-2}\right)^{\frac{p-1}{p}}\right)
            \end{equation*}
        \end{statement}
        \begin{proof}
            We refer to Example \ref{ellipticop} and Remark \ref{remclos}. For $\psi$, we have
            \begin{equation*}
                \begin{aligned}
                    &\int_{\Gamma}\left(\nabla \hat{y}\cdot\mathbf{n}_\Gamma\right)\hat{y}|\hat{y}|^{p-2}d\Gamma + \int_{\Gamma}|\hat{y}|^pd\Gamma\leq 2^{p-1}\|\mathcal{R}_{bn}\|^p_{L^p(\Omega;\mathbb{R}^4)}+\\
                    & +2^{\frac{(p+1)(p-1)^3}{p^3}}\max\{\|y_{x_1}\|_{C(\Omega)},\|y_{\theta, x_1}\|_{C(\Omega)},\|y_{x_2}\|_{C(\Omega)},\|y_{\theta, x_2}\|_{C(\Omega)}\}\|\mathcal{R}_{bn}\|^{p-1}_{L^p((-1;1);\mathbb{R}^4)}
                \end{aligned}
            \end{equation*}
            Then, for estimation in terms of residuals, we apply Theorem \ref{thhyp}. We cannot explicitly apply ideas similar to Corollary \ref{corpar} and Lemmas \ref{midpointrule}, \ref{pc2norm} since residuals are not of $C^2$-class. However, we can consider $\Omega^{*}\subset\Omega$, where residuals are of $C^2$-class. Since integration doesn't depend on a null set, we can deal with residuals on $\Omega^{*}$, where we can apply Lemmas \ref{midpointrule}, \ref{pc2norm}. 
        \end{proof}
        Now, we turn to the exact problem. Let $\Omega^{*}=\left([-1;0)\cup(0;1]\right)\times\left([-1;0)\cup(0;1]\right)$, $\phi=2{\rm sgn}(x_1)+2{\rm sgn}(x_2)$ and 
        \begin{equation*}
            \begin{aligned}
                &f_{1}(x_2)=1-|x_2|x_2,~f_{2}(x_2)=-1-|x_2|x_2\\
                &f_{3}(x_1)=1-|x_1|x_1,~f_{4}(x_1)=-1-|x_1|x_1
            \end{aligned}
        \end{equation*}
        Then the exact solution is 
        \begin{equation*}
            y(x_1,x_2)=-|x_1|x_1-|x_2|x_2
        \end{equation*}
        
        Again, the neural network is of 2 hidden layers of width 80, taking $(x_1,x_2,|x_1|,|x_2|)$ as input. The count of epochs is 20000, and training samples are of size $M_{eq}=2^{18}$ and $M_{bn}=2^{12}$. Also, we make experiments for the cases $p\in\{2,3,4,5\}$. 
        \begin{figure}[t]
            \centering
            \includegraphics[scale=0.4]{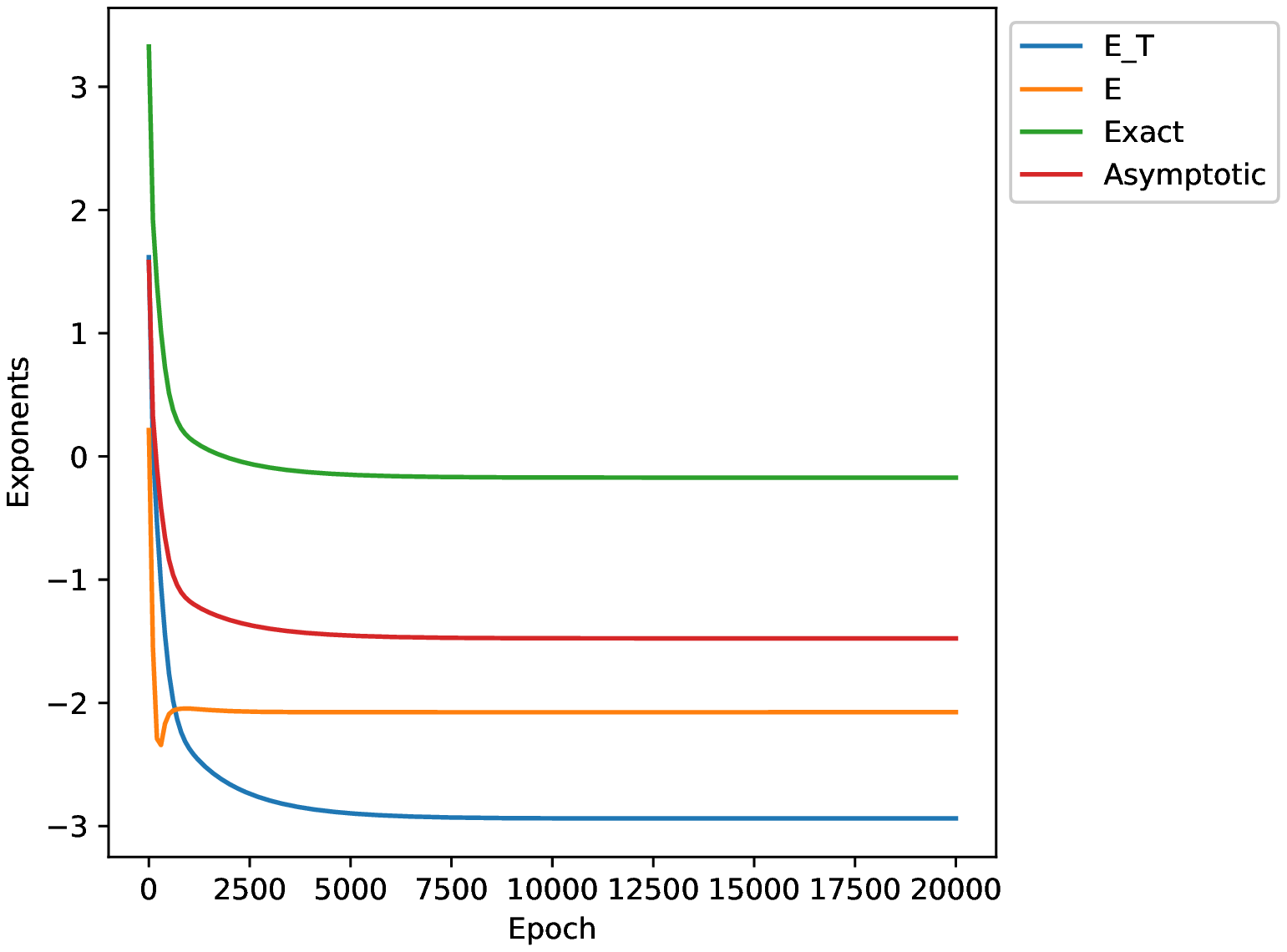}
            \includegraphics[scale=0.4]{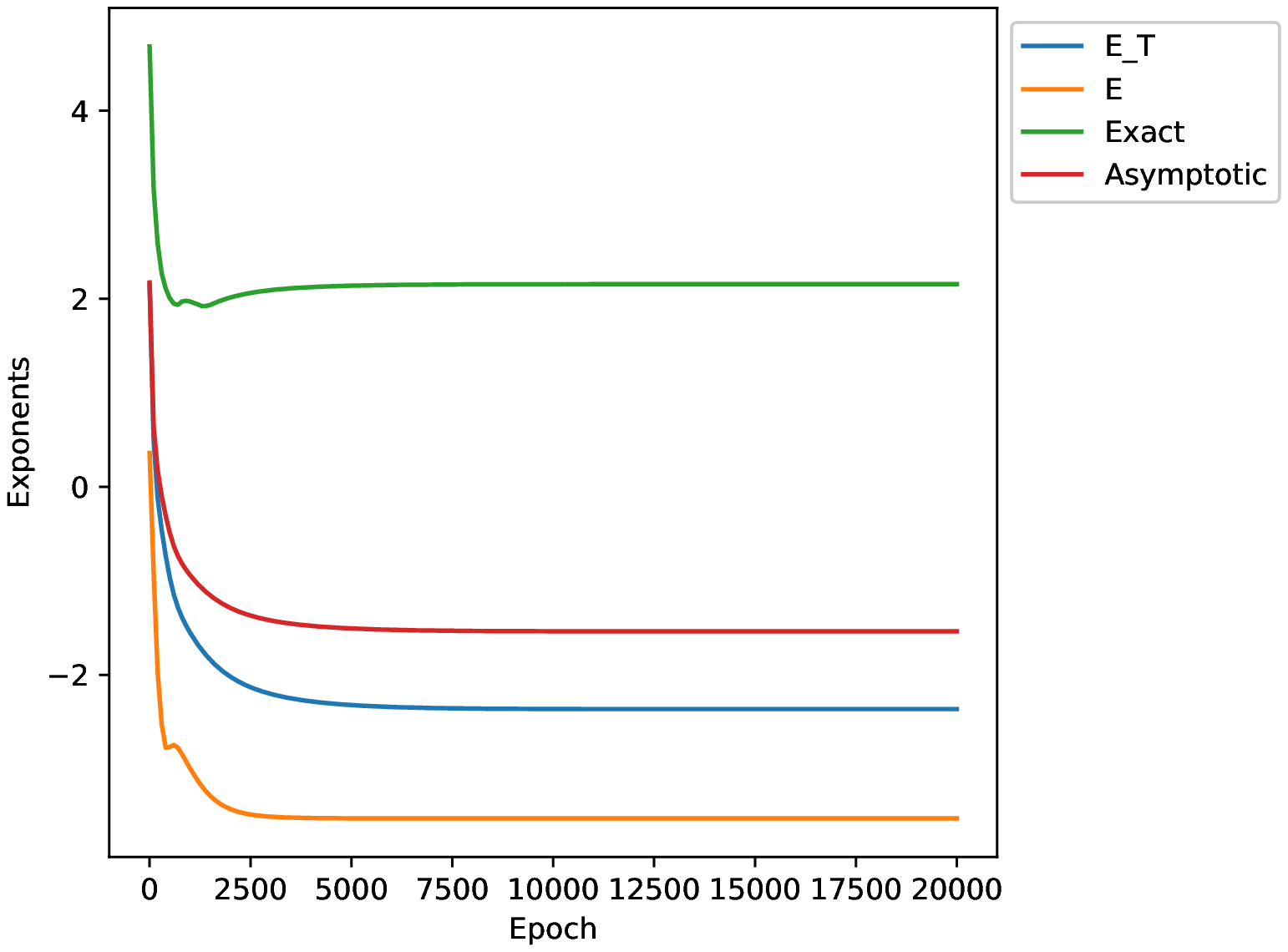}
            \includegraphics[scale=0.4]{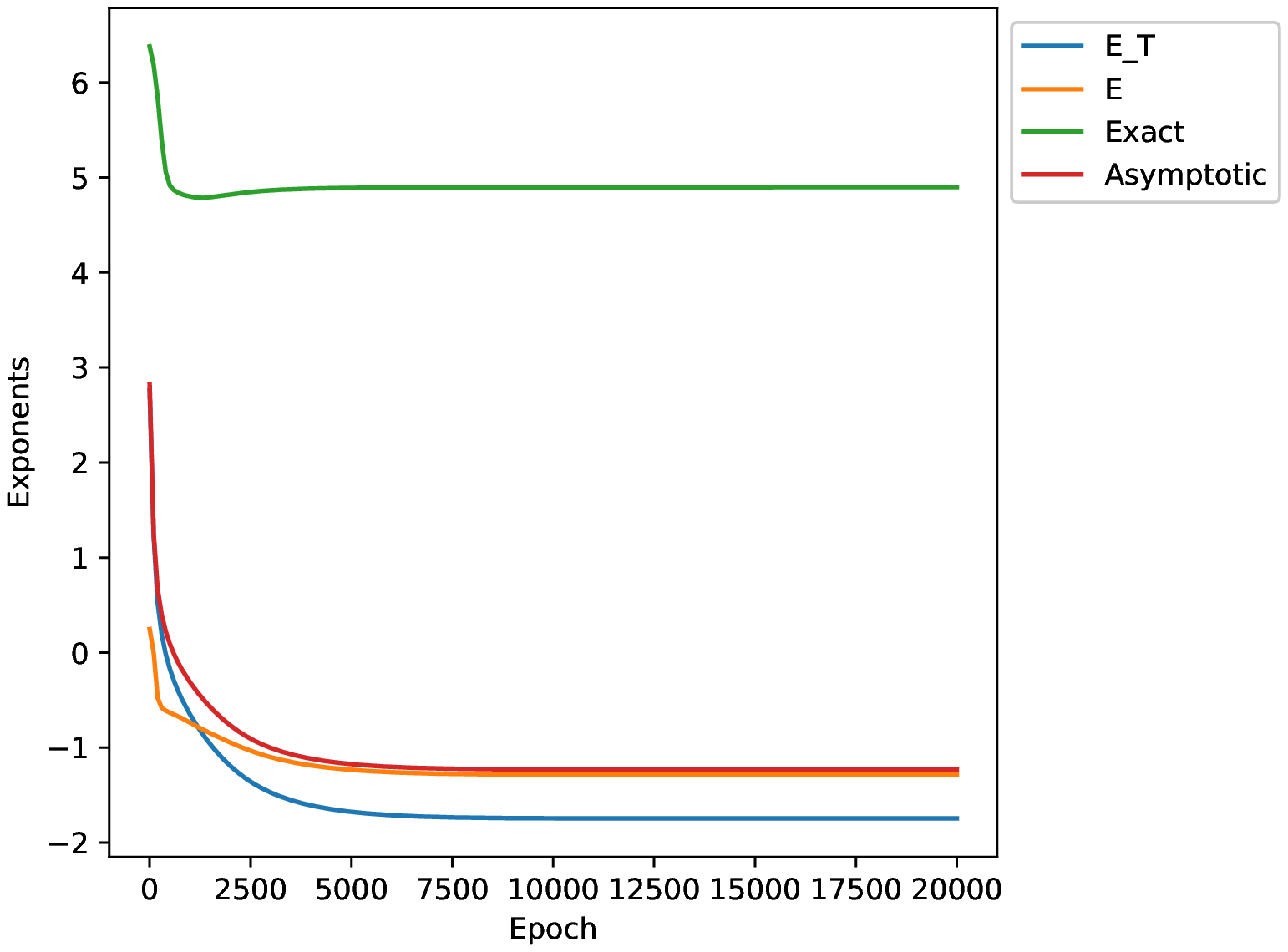}
            \includegraphics[scale=0.4]{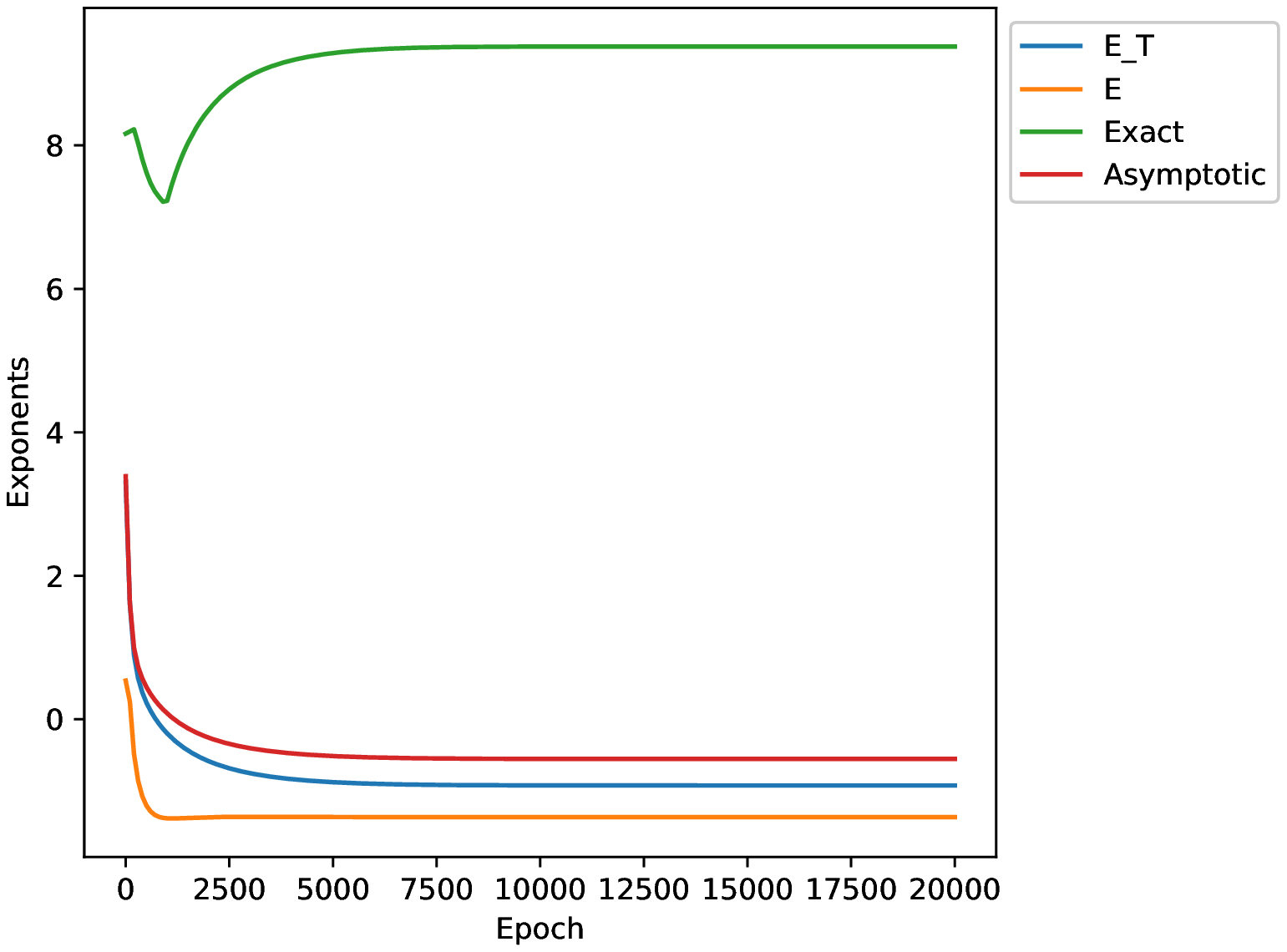}
            \caption{Poisson equation: $\lg(\mathcal{E})$, $\lg(\mathcal{E}_T)$, $\lg(\mathcal{E}_{exact})$, $\lg(\mathcal{E}_{asymp})$ for $L^p$, $p\in\{2,3,4,5\}$}\label{elliptic1}
        \end{figure} 
        \begin{figure}[t]
            \centering
            \includegraphics[scale=0.4]{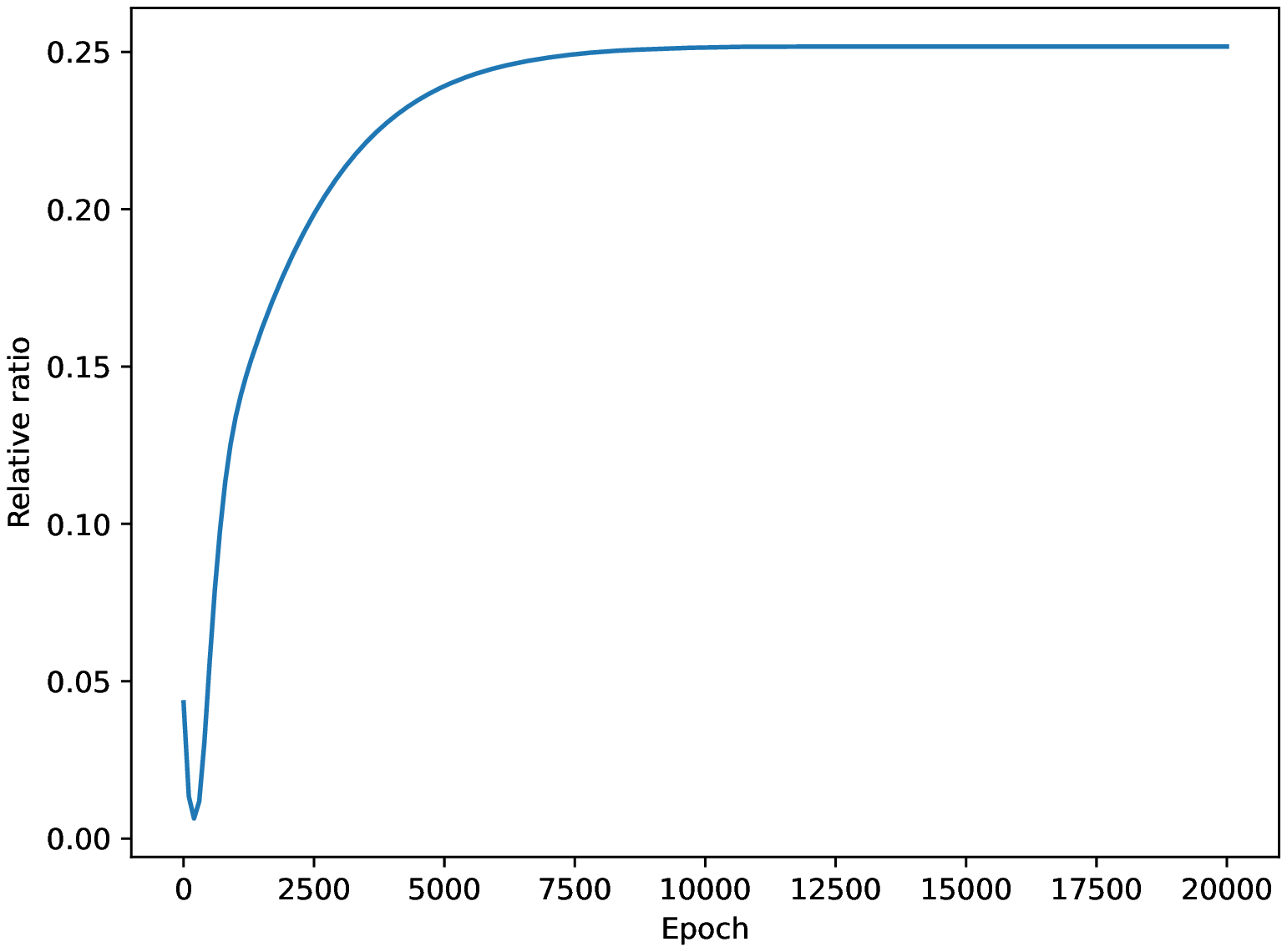}
            \includegraphics[scale=0.4]{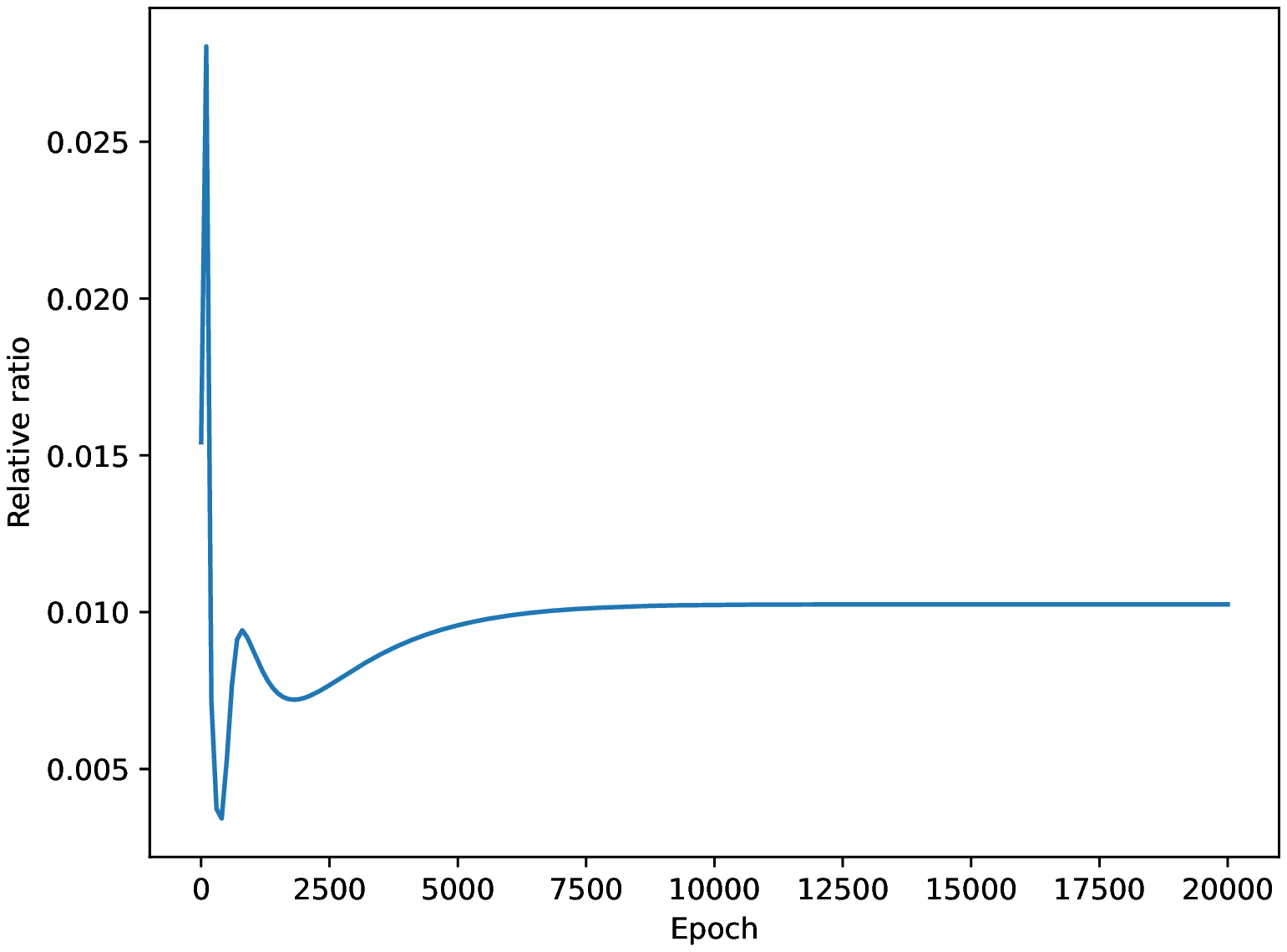}
            \includegraphics[scale=0.4]{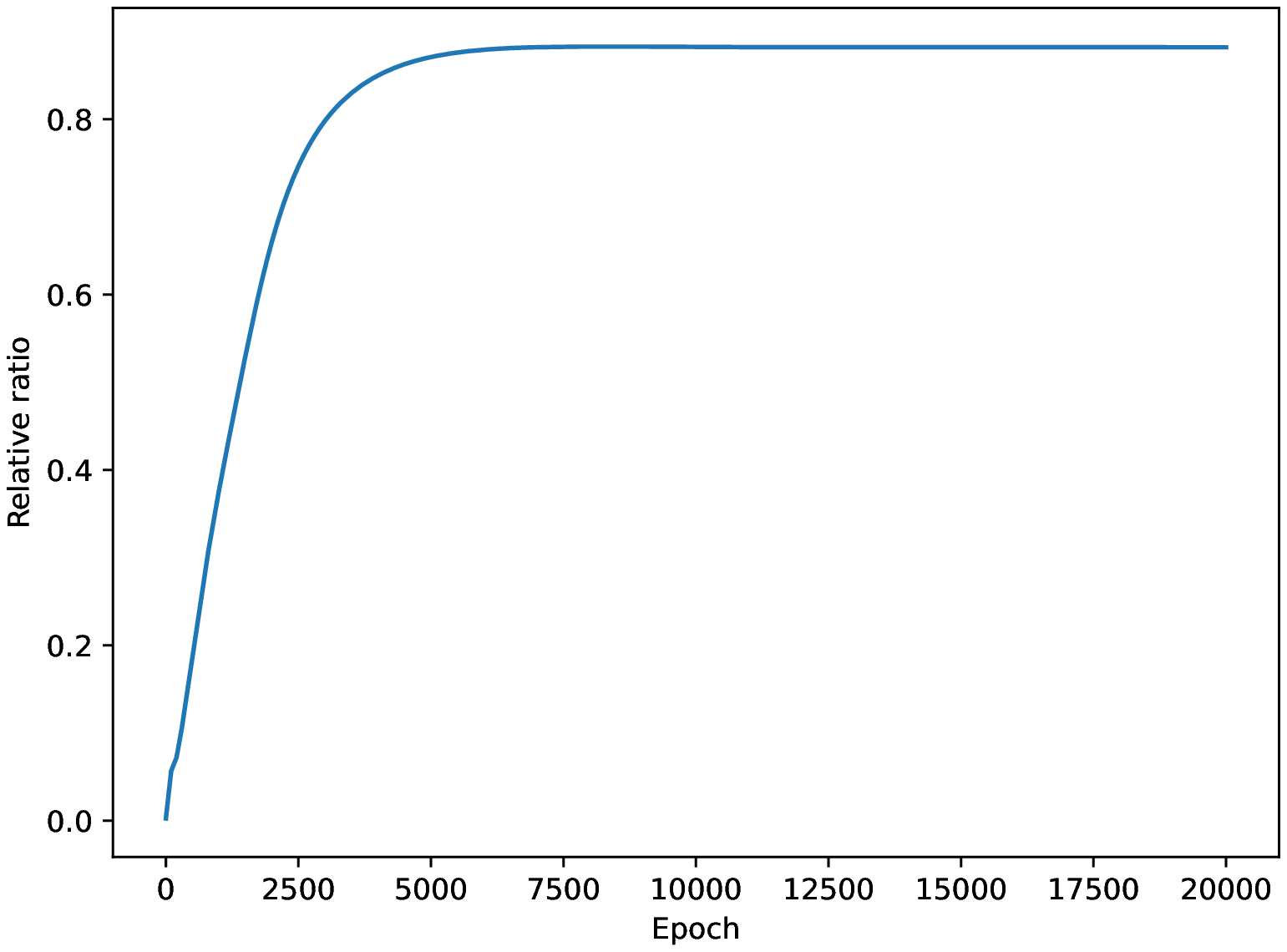}
            \includegraphics[scale=0.4]{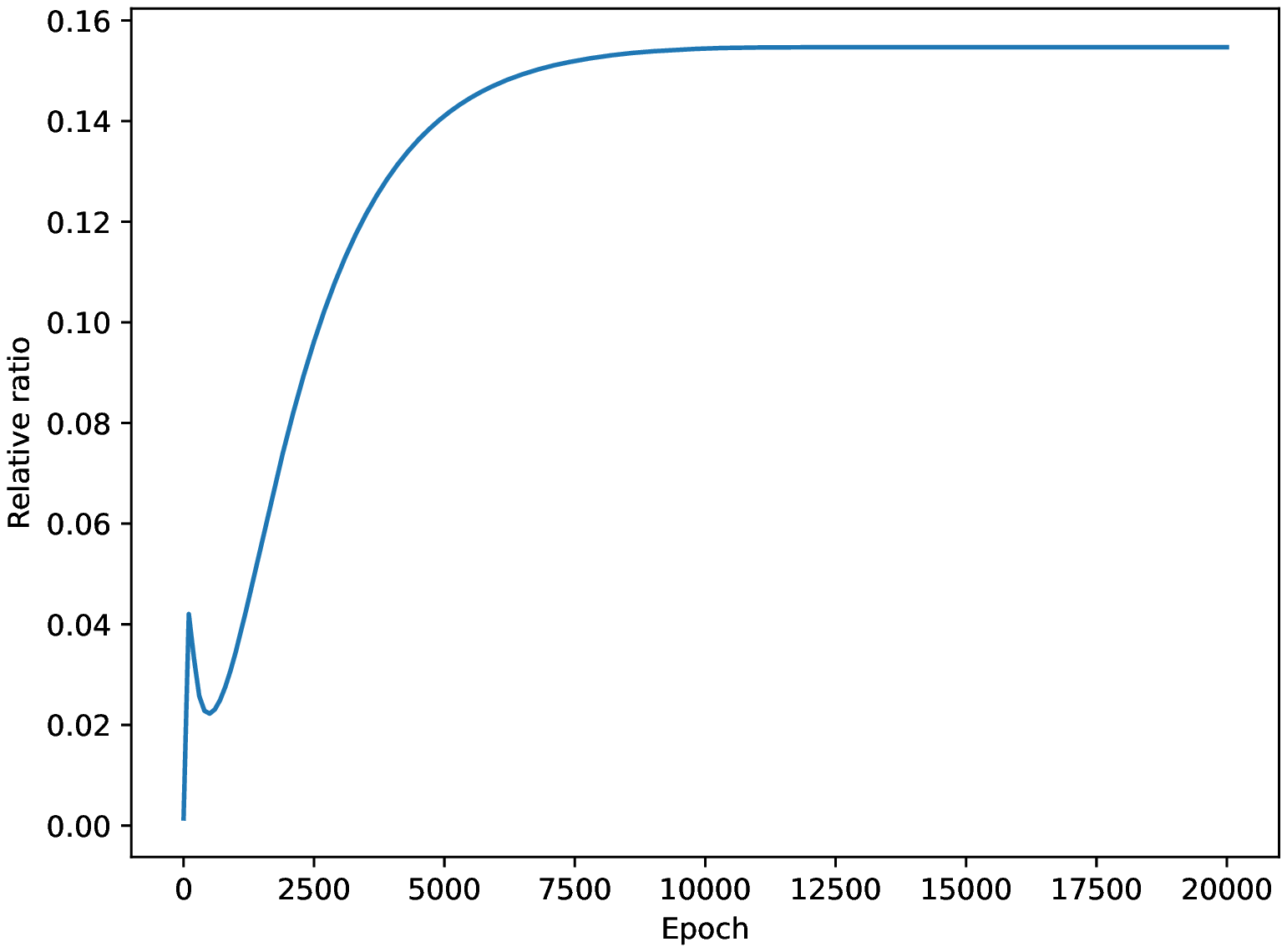}
            \caption{Poisson equation: $\frac{\mathcal{E}_T}{\mathcal{E}_{asymp}}$ for $L^p$, $p\in\{2,3,4,5\}$}\label{elliptic2}
        \end{figure} 
        Again, we see the correlation between asymptotic estimate and total error behaviour, and that exact estimate overestimates total error (Figures \ref{elliptic1}, \ref{elliptic2}). Thus, even for equations with singularities, we can apply the estimation. 

\section{Conclusion}
    Hence, we gave an operator description of three PDE classes to which the error estimation, due to S. Mishra et al., can be applied. These classes cover a wide range of equations. Furthermore, we dealt with a more general real $p$-form instead of the inner product. For each class, we obtained an exact error estimate. Also, replacing the mean value with a projection, we extended the Bramble-Hilbert lemma to the non-Hilbert Sobolev spaces. Finally, we stated the theorem on neural network approximation existence in such a space. 

    However, obtained estimates remain a posteriori since we cannot guarantee that the training error will be small. It is still an open question under what conditions the gradient descent algorithm converges to a global minimum in the context of the PINN training.

\begin{appendices}
    \section{Functional spaces and auxiliary lemmas}
        For any space $\mathcal{Y}$ of functions $\Omega\to\mathcal{Z}$, we use denotation $\mathcal{Y}(\Omega;\mathcal{Z})$. If $\mathcal{Z}=\mathbb{R}$, we simply denote $\mathcal{Y}(\Omega)$. For the conventional functional spaces, we use the following norms:
        \begin{equation*}
            \|\mathbf{y}\|_{L^p(\Omega;\mathbb{C}^{n})}=\left(\sum_{i=1}^{n}\|y_i\|^p_{L^p(\Omega;\mathbb{C})}\right)^{\frac{1}{p}}
        \end{equation*}
        \begin{equation*}
            \|f\|_{W^{\sigma,p}(\Omega;\mathbb{R}^n)}=\left(\sum_{i=1}^{n}\sum_{\nu=0}^{k}\sum_{|\iota|=\nu}\|\partial^\iota f_i\|^p_{L^p(\Omega)} \right)^{\frac{1}{p}}
        \end{equation*}
        \begin{equation*}
            \|f\|_{W^{\sigma,\infty}(\Omega;\mathbb{R}^n)}=\max_{1\leq i\leq n}\max_{0\leq \nu\leq \sigma}\max_{|\iota|=\nu}\|\partial^{\iota}f\|_{L^{\infty}(\Omega)}
        \end{equation*}
        \begin{equation*}
            \|f\|_{C^{\sigma}(\Omega;\mathbb{R}^n)}=\max_{1\leq i\leq n}\max_{0\leq \nu\leq \sigma}\max_{|\iota|=\nu}\|\partial^{\iota}f\|_{C(\Omega)}
        \end{equation*}

        \begin{lemma}\label{1dpoincare} (One-dimensional Poincare inequality with trace term)
            The following inequality holds
            \begin{equation*}
                \|y\|^p_{L^p((a;b))}\leq \frac{b-a}{2}\left(y(a)+y(b)\right)^p+2^{p-2}(b-a)^p\|y'\|^{p}_{L^p((a;b))}
            \end{equation*} 
        \end{lemma}

        \begin{lemma}\label{traceinequality} (Trace inequality on parallelepiped)
            Let $\Omega=\prod_{j=1}^{m}[a_j;b_j]$, then 
            \begin{equation*}
                \left\|y\Big|_{x_j=x_{0,j}}\right\|_{L^p(\Omega\setminus[a_j;b_j])}\leq \frac{2^{\frac{p-1}{p}}}{(b_j-a_j)^{\frac{1}{p}}}\max\{1,(b_j-a_j)\}\|y\|_{W^{1,p}(\Omega)}
            \end{equation*}
        \end{lemma}

        \begin{lemma}\label{midpointrule}(Midpoint rule)
            Let $M\in\mathbb{N}$ and $M^{\frac{1}{m}}\in\mathbb{N}$, $\Omega=\prod_{j=1}^{m}[a_j;b_j]$, and dividing each $[a_j;b_j]$ into $M^{\frac{1}{m}}$ equal parts, we obtain set of centers $\mathcal{X}\subset\Omega$. Then 
            \begin{equation*}
                \left|\int_{\Omega}y(x)dx-\frac{\prod_{j=1}^{m}(b_j-a_j)}{M}\sum_{x_0\in\mathcal{X}}y(x_0)\right|\leq \frac{\prod_{j=1}^{m}(b_j-a_j)}{24M^{\frac{2}{m}}}\sum_{j=1}^{m}(b_j-a_j)^2\left\|\partial^2_{x_j}y\right\|_{C(\Omega)}
            \end{equation*}
        \end{lemma}

        \begin{lemma}\label{pc2norm}
            Let $p\geq2$, then
            \begin{equation*}
                \left\|(|y|^p)''\right\|_{C([a;b])}\leq  p^2\|y\|^p_{C^2([a;b])}
            \end{equation*}
        \end{lemma}
\end{appendices}

\bibliographystyle{unsrtnat}

\end{document}